\documentclass{amsart}

\usepackage{amssymb}
\usepackage{tikz-cd}
\usepackage{array,booktabs}
\usepackage{microtype}
\usepackage{enumitem}
\usepackage{xcolor}
\usepackage[bookmarksnumbered=true,bookmarksopen=true,bookmarksopenlevel=2,bookmarksdepth=3]{hyperref}
\definecolor{AMSDarkBlue}{RGB}{0,51,102}
\hypersetup{
  colorlinks=true,
  linkcolor=AMSDarkBlue,
  citecolor=AMSDarkBlue,
  urlcolor=AMSDarkBlue,
  breaklinks=true,
  pdfpagemode=UseOutlines,
  pdftitle={Exceptional covers of the projective line with genus-one Galois closure: arithmetic forms over finite fields},
  pdfauthor={Xiang Fan}
}
\numberwithin{equation}{section}

\newtheorem{theorem}{Theorem}[section]
\newtheorem{proposition}[theorem]{Proposition}
\newtheorem{lemma}[theorem]{Lemma}
\newtheorem{corollary}[theorem]{Corollary}
\newtheorem{definition}[theorem]{Definition}
\theoremstyle{remark}
\newtheorem{remark}[theorem]{Remark}
\newtheorem{example}[theorem]{Example}
\newtheorem*{theoremA}{Theorem A}
\newtheorem*{theoremB}{Theorem B}
\newtheorem*{theoremC}{Theorem C}

\title[Genus-one exceptional covers]{Exceptional covers of the projective line with genus-one Galois closure: arithmetic forms over finite fields}
\author{Xiang Fan}
\address{School of Mathematics, Sun Yat-sen University, Guangzhou 510275, China}


\begin{document}

\begin{abstract}
Let $k=\mathbb F_q$.  We classify, up to two-sided $k$-M\"obius
equivalence, all separable indecomposable tame exceptional maps
$\mathbf P^1_k\to\mathbf P^1_k$ whose geometric Galois closure has
genus one.  The resulting arithmetic fixed-field forms are encoded by
Frobenius-stable affine elliptic quotient data recovered from the cover;
we prove a converse and an exact equivalence criterion.  The same data
determine branch arithmetic, geometric and arithmetic monodromy, the
constant field, and behavior over every finite extension.  In particular,
over each $\mathbb F_{q^r}$, permutation is equivalent to exceptionality.
We obtain explicit rank-two support formulas, necessary and sufficient
occurrence criteria, and exact two-sided class counts in every tame
signature, including the surviving cases in characteristics $2$ and $3$.
In characteristic greater than $3$, tameness is automatic.  A sharp
$3$-adic obstruction arises only for a stronger cubic self-endomorphism
realization problem, not for the fixed-field classification itself.
\end{abstract}

\maketitle

\section{Introduction}
\label{sec:intro}

\subsection{The fixed-field problem}

Exceptional covers lie at the intersection of finite-field arithmetic,
monodromy, and the geometry of curves.  Let $k=\mathbb F_q$.  A rational map
$f:\mathbf P^1_k\to\mathbf P^1_k$ is \emph{exceptional} if it induces a
bijection on $\mathbf P^1(\mathbb F_{q^r})$ for infinitely many positive
integers $r$.  We study the genus-one case: the geometric Galois closure of
$f$ is assumed to have genus one, while $f$ itself is separable and
indecomposable.

Over $\bar k$, the genus-one Galois-closure curve becomes an elliptic
curve after choosing an origin.  In the tame affine cases the lower cover is
obtained from an isogeny by quotienting source and target by cyclic groups of
automorphisms.  Over the prescribed field $k$, however, this geometric picture
does not by itself determine the form of the cover.  The
cyclic point stabilizer need only descend as a Frobenius-stable \emph{affine}
subgroup, and its translation component can carry arithmetic information
which disappears after base change to $\bar k$.

The purpose of this paper is to identify and classify these arithmetic
forms.  We show that every cover in the tame genus-one setting is encoded by
a Frobenius-stable affine elliptic quotient datum
$\mathcal D=(E,\mathcal N,H)$.  Here $E/k$ is an elliptic form of the
geometric Galois closure, $\mathcal N$ is the Frobenius-stable translation
kernel, and $H$ is a cyclic affine point stabilizer.  Its translation part is
recorded by a shift class $\mathfrak s_E(H)$.  The same datum separates three
layers of the problem: fixed-field forms and branch arithmetic; monodromy,
constant fields, and exact permutation--exceptionality support under base
change; and finally occurrence and enumeration for prescribed parameters.

Choosing source and target coordinates identifies the resulting covers with
rational functions.  We write
\[
 f_1\sim_k f_2
 \quad\Longleftrightarrow\quad
 f_1=\mu\circ f_2\circ\eta
 \quad\text{for some }\mu,\eta\in\operatorname{PGL}_2(k),
\]
and classify the corresponding two-sided $k$-M\"obius classes.  Let
$\mathbf{x}$ be transcendental over $k$, put $\mathbf t=f(\mathbf x)$, and
let $\bar\Omega$ be the Galois closure of
$\bar k(\mathbf x)/\bar k(\mathbf t)$.  We write
\[
 G_f:=\operatorname{Gal}(\bar\Omega/\bar k(\mathbf t))
\]
and call the cover \emph{Euclidean} when the smooth projective curve with
function field $\bar\Omega$ has genus one.  The term is used below only as a
compact name for this genus-one Galois-closure condition.

\subsection{Main results}

The tame Euclidean reduction yields a common mechanism.  For an odd prime
$\ell$ one has
\[
 \deg f=\ell^e,\qquad e\in\{1,2\},
\]
and the geometric monodromy has a regular normal subgroup
$N\cong C_\ell^e$ with cyclic complement $C_m$.  We call $e=1$ and $e=2$
the rank-one and rank-two cases.  Roman $N$ denotes this regular normal
subgroup; calligraphic $\mathcal N$ denotes the corresponding subgroup of
points on the descended elliptic curve.  For an elliptic curve $E/k$, write
$\mathrm O_E$ for its origin and $\sigma_q$ for arithmetic $q$-Frobenius.

The fixed-field mechanism is summarized by the diagram
\[
\begin{tikzcd}
 E \arrow[r,"\varphi"] \arrow[d,"{/H}"'] & E'=E/\mathcal N \arrow[d,"{/H'}"]\\
 \mathbf P^1 \arrow[r,"f"'] & \mathbf P^1.
\end{tikzcd}
\]
Here $\mathcal N$ is Frobenius-stable, $H$ is a Frobenius-stable cyclic
affine subgroup, and $H'$ is its induced action on $E'$.  A Frobenius section
descends the genus-one closure; the regular subgroup becomes translations,
and the affine stabilizer contributes the shift class.  Together, the cyclic
linear action and the induced Frobenius action on $\mathcal N$ supply the
finite arithmetic readout.  The formal definition and lower-map construction
are given in Section~\ref{sec:reduction-descent}.

\begin{theoremA}[Arithmetic classification of fixed-field forms]
Let $f\in k(X)$ be separable, indecomposable, tame, and exceptional, and
suppose that its geometric Galois closure has genus one.  Then $f$ is represented, up to two-sided $k$-M\"obius equivalence, by a
Frobenius-stable affine elliptic quotient datum
$\mathcal D=(E,\mathcal N,H)$.  Conversely, every such datum gives a
separable lower cover of degree $\ell^e$ with geometric monodromy
$T_{\mathcal N}\rtimes H$, where $T_{\mathcal N}$ is the translation
subgroup attached to $\mathcal N$, and with the ramification signature
determined by $m$.  Two data determine the same two-sided class precisely when a marked geometric isomorphism preserves
the translation kernel and cyclic linear action and is compatible with both
Frobenius descent and the shift class.  In rank two, the translation kernel is
$E[\ell](\bar k)$ and the quotient isogeny is $[\ell]$ up to a $k$-isomorphism
of the target.
\end{theoremA}

The precise datum and the exact four-condition equivalence criterion are stated in
Theorem~\ref{thm:intro-forms}.  This distinction is important: the
classification concerns arithmetic forms of the lower cover, not merely the
existence of an elliptic model after extension of scalars.

The underlying geometric reduction leaves exactly four possibilities:
\begin{equation}
\label{eq:four-types-intro}
\begin{array}{c|c|c}
 m&\text{ramification type}&\text{degree condition}\\ \hline
 2&(2,2,2,2)&\ell^e\geqslant5,\\
 6&(2,3,6)&\ell^e\equiv1\pmod6,\\
 3&(3,3,3)&\ell^e\equiv1\pmod6,\\
 4&(2,4,4)&\ell^e\equiv1\pmod4.
\end{array}
\end{equation}
Thus the four signatures are geometric realizations of the same fixed-field
mechanism rather than separate classification problems.

Once $\mathcal D$ is fixed, write
$L_E=L_E(H)=\langle\beta\rangle$ for the cyclic linear image of $H$, let
$\mathcal F_E$ be the $q$-power Frobenius endomorphism of $E$, and put
$\mathcal F_{\mathcal N}:=\mathcal F_E|_{\mathcal N}$.  Write $G_g$ and
$A_g$ for the geometric and arithmetic monodromy groups of an associated lower
map $g$.  The finite linear pair $(\beta,\mathcal F_{\mathcal N})$ is the
arithmetic readout of the datum.

\begin{theoremB}[Monodromy, constant fields, and exact extension support]
For an associated lower map $g$, indecomposability and exceptionality are
determined by the action of $\langle\beta,\mathcal F_{\mathcal N}\rangle$ on
$\mathcal N$.  Moreover
\[
 G_g\cong\mathcal N\rtimes\langle\beta\rangle,
 \qquad
 A_g\cong\mathcal N\rtimes\langle\beta,\mathcal F_{\mathcal N}\rangle.
\]
For every $r\geqslant1$,
\[
\begin{aligned}
&g:\mathbf P^1(\mathbb F_{q^r})\longrightarrow
       \mathbf P^1(\mathbb F_{q^r})\text{ is bijective}\\
&\qquad\Longleftrightarrow\
g/\mathbb F_{q^r}\text{ is exceptional}\\
&\qquad\Longleftrightarrow\
\ker(\mathcal F_{\mathcal N}^r-\beta^j)|_{\mathcal N}=0
\quad(0\leqslant j<m).
\end{aligned}
\]
Thus permutation and exceptionality have exactly the same extension support,
with no accidental permutation degrees.  If
\[
 d_g=\min\{\nu\geqslant1:\mathcal F_{\mathcal N}^\nu\in\langle\beta\rangle\},
\]
then the constant field of the arithmetic Galois closure is
$\mathbb F_{q^{d_g}}$.  In rank two the common support has closed formulas in
both the split and nonsplit cases.
\end{theoremB}

The precise statement, including the branch criterion and the exact
permutation--exceptionality support theorem, is
Theorem~\ref{thm:intro-arithmetic} and Section~\ref{sec:converse}.  Thus the
fixed-field datum controls not only the form of the cover but also its
arithmetic monodromy and its complete behavior under finite extension of
constants.

The surviving tame signatures in small characteristic are
\begin{equation}
\label{eq:characteristic-range-intro}
\begin{array}{c|c}
 \operatorname{char}k&\text{signatures}\\ \hline
 >3&(2,2,2,2),\ (3,3,3),\ (2,4,4),\ (2,3,6),\\
 3&(2,2,2,2),\ (2,4,4),\\
 2&(3,3,3).
\end{array}
\end{equation}
In characteristic greater than $3$, tameness is automatic from the genus-one
hypothesis.  In characteristics $2$ and $3$, the common descent mechanism
depends on the normalizer of the marked cyclic action rather than on the full
automorphism group of the elliptic curve.  This marked-normalizer principle
is what allows the tame theory to remain uniform despite the enlarged
special-$j$ automorphism groups.

\begin{corollary}[Characteristic range]
\label{cor:intro-characteristic-range}
Under the hypotheses of Theorem~\ref{thm:intro-forms}, the tame signatures are
exactly those in \eqref{eq:characteristic-range-intro}.  If
$\operatorname{char}k>3$, the tameness hypothesis is redundant.
\end{corollary}

The fixed-field classification has a further arithmetic closure.  For each
admissible quadruple $(q,\ell,m,e)$ one can ask whether the corresponding
cover occurs and, if it does, how many two-sided $k$-M\"obius classes occur.

\begin{theoremC}[Occurrence and enumeration]
For every admissible $(q,\ell,m,e)$ in the surviving tame signatures, there is
an exact necessary and sufficient occurrence criterion, and every occurring
parameter has an exact two-sided class count.  In the three CM signatures the
answers reduce to cyclotomic congruences and finite-step CM Frobenius data.  In
the involutive signature they reduce to the finite Waterhouse trace set and an
exact finite residue-degree-weighted class-number sum with the required
special-$j$ corrections.  The tame cases in characteristics $2$ and $3$
are included.
\end{theoremC}

This closes the classification arithmetically: the fixed-field forms lead to
finite marked Frobenius data, from which occurrence and enumeration become
finite arithmetic problems.  The exact occurrence and counting statements are
given in Theorems~\ref{thm:master-occurrence} and~\ref{thm:master-enumeration}.
Table~\ref{tab:arithmetic-roadmap} records how these arithmetic problems close
by signature and rank.  The degree conditions in
\eqref{eq:four-types-intro} are understood; in the support column, ``split''
and ``nonsplit'' refer to the two rank-two Frobenius types treated in
Section~\ref{sec:rank-two-support}.

\begin{table}[t]
\centering
\small
\setlength{\tabcolsep}{1.5pt}
\begin{tabular}{@{}>{\raggedright\arraybackslash}p{1.55cm}
>{\raggedright\arraybackslash}p{1.25cm}c
>{\raggedright\arraybackslash}p{2.0cm}
>{\raggedright\arraybackslash}p{3.15cm}
>{\raggedright\arraybackslash}p{3.3cm}@{}}
\toprule
Signature & Tame char. & Rank & Common support & Occurrence engine & Exact-count engine\\
\midrule
$(2,2,2,2)$ & $\ne2$ & 1 & $d_g\nmid r$ & Waterhouse trace + good Frobenius eigenline; reduced line when $\ell=\operatorname{char}k$ & Waterhouse/Schoof order-refined sum + special-$j$ Burnside\\
$(2,2,2,2)$ & $\ne2$ & 2 & $d_g\nmid r$ & $\ell\ne\operatorname{char}k$; Waterhouse trace with irreducible $E[\ell]$ & same global Waterhouse/Schoof--Burnside formula\\
$(3,3,3)$ & $\ne3$ & 1 & $d_g\nmid r$ & $q\equiv1\pmod3$, $\ell\equiv1\pmod3$ & three-step Frobenius multiplicities + marked shift weights\\
$(3,3,3)$ & $\ne3$ & 2 & split: $d_g\nmid r$; nonsplit: parity-residue law & split cyclotomic condition; nonsplit $\ell\nmid q+1$ & split three-step formula; nonsplit exactly $3$ classes\\
$(2,4,4)$ & $\ne2$ & 1 & $d_g\nmid r$ & split cyclotomic condition + finite $q^4$-Frobenius test & four-step Frobenius multiplicity; shift factor $2$\\
$(2,4,4)$ & $\ne2$ & 2 & split: $d_g\nmid r$; nonsplit: parity-residue law & split finite $q^4$-Frobenius test; nonsplit $\ell\nmid q+1$ & exactly $2$ classes after occurrence\\
$(2,3,6)$ & $>3$ & 1 & $d_g\nmid r$ & split cyclotomic condition + finite $q^6$-Frobenius test & six-step Frobenius multiplicity; shift factor $1$\\
$(2,3,6)$ & $>3$ & 2 & split: $d_g\nmid r$; nonsplit: parity-residue law & split finite $q^6$-Frobenius test; nonsplit $\ell\nmid q+1$ & exactly $1$ class after occurrence\\
\bottomrule
\end{tabular}
\caption{Roadmap for support, occurrence, and enumeration.}
\label{tab:arithmetic-roadmap}
\end{table}

The table is a roadmap rather than a replacement for the exact statements:
the support formulas are proved in Section~\ref{sec:converse}, occurrence in
Section~\ref{sec:occurrence}, and the literal two-sided class counts in
Section~\ref{sec:CM-counts}.  The $3$-adic condition discussed next belongs
only to the stronger same-$C_3$ realization problem and does not enter the
fixed-field occurrence or class counts.

One sharper phenomenon must be kept separate from the classification itself.
In the cubic signature $(3,3,3)$, the fixed-field classification is complete
with no additional obstruction.  A stronger
requirement is that, after two-sided $k$-M\"obius equivalence, the elliptic
map be a $k$-self-endomorphism and the same Frobenius-stable affine $C_3$ act
on source and target; we call this a \emph{same-$C_3$ self-endomorphism
realization} (Definition~\ref{def:same-c3}).  Such a realization can fail only
in the nonsplit rank-two, nonzero-shift case.  In characteristic different
from $3$ it holds exactly when
\[
 v_3(\operatorname{sgn}_3(\ell)\ell-1)\geqslant v_3(q+1),
\]
where $\operatorname{sgn}_3(\ell)\ell\equiv1\pmod3$.  The degree-$49$ example
over $\mathbb F_{17}$ shows that this boundary is genuine.  This condition is
therefore a boundary for the stronger realization problem, not an exception
to the fixed-field classification.

\subsection{Relation to previous work}

The geometric and group-theoretic possibilities are well established.
Fried formulated the rational-function Schur problem as an
extension-of-constants problem and, in prime degree, isolated the four tame
Euclidean ramification patterns and realized them by elliptic isogeny
quotients \cite[Theorem~2.1, Lemma~2.1, and Theorem~2.2]{Fried78}.
Fried--Guralnick--Saxl developed the arithmetic framework for exceptional
covers \cite{FGS93}.  Guralnick--Tucker--Zieve proved that an exceptional
cover over a finite field is bijective on rational points over that field,
whereas their general converses from injectivity or surjectivity require
explicit largeness hypotheses.  They also exhibited nonexceptional rational
maps that are nevertheless bijective over the ground field
\cite[Theorems~1--2 and Example~5.2]{GTZ07}.  Guralnick--M\"uller--Saxl
(GMS) placed the Euclidean cases in a broader affine classification; their
characteristic-zero results include the corresponding elliptic isogeny and
endomorphism normal forms
\cite[Lemma~4.4, Theorem~4.14(b), and Theorems~6.5--6.6]{GMS03}.  Pakovich
gives a general complex elliptic-quotient description for rational functions
with genus-one normalization \cite[Theorem~1.2 and Theorem~5.1]{Pak18}.

These geometric results do not by themselves classify forms over a prescribed
finite field.  The cyclic stabilizer may descend affinely, and the resulting
shift is invisible over $\bar k$.  M\"uller isolated a related descent issue
in the number-field setting and treated the $(2,2,2,2)$ series under a
rational-branch hypothesis \cite[\S5]{Muller99}.  The present paper supplies
the prescribed-field classification in all tame Euclidean signatures and in
both possible ranks, together with the exact equivalence relation and the
arithmetic consequences of the finite linear pair
$(\beta,\mathcal F_{\mathcal N})$ on the translation kernel.  Related
finite-field isogeny constructions occur in Bisson--Tibouchi
\cite[Theorem~2.1 and \S\S3--4]{BT18}.  Exact permutation criteria for
origin-preserving Latt\`es families were obtained by K\"u\c{c}\"uksakall\i{}
for CM isogeny quotients \cite[Corollary~2.8]{Kuc14} and, for multiplication
maps, by Bell et al. \cite[Corollary~2.6]{Bell22}.
Panraksa--Samart--Sriwongsa study a different arithmetic quantifier for
Latt\`es maps over $\mathbb Q$, varying the reduction prime \cite{PSS26}.

A companion paper studies the complementary forward problem for a prescribed
Frobenius-stable equivariant elliptic-isogeny datum in which the finite group
acts faithfully by origin-preserving automorphisms.  It proves an exact
all-extension kernel criterion, a weighted Frobenius-sector identity valid
above ramified quotient fibres, an affine monodromy description, and a purely
periodic permutation-support theorem
\cite[Theorems~1.3--1.4, Theorem~2.4, and Proposition~4.2]{Fan26}.  The
present paper addresses the inverse fixed-field problem: starting from a tame
indecomposable exceptional cover with genus-one geometric Galois closure, it
reconstructs the elliptic quotient datum, whose cyclic point stabilizer may
descend affinely with a nonzero shift, and determines its fixed-field forms,
shift classes, exact two-sided equivalence, branch arithmetic, occurrence, and
enumeration.  The nonzero affine shifts occurring here lie outside the
prescribed origin-preserving setup of \cite{Fan26}.
Section~\ref{sec:converse} therefore establishes the required affine
finite-extension criterion self-containedly, and the classification and
counting results do not depend on the companion paper.

The non-CM involutive row of our enumeration naturally meets the finite-field
elliptic-curve strata of Waterhouse \cite{Waterhouse69} and Schoof's
corrected fixed-order counts \cite{Schoof87}.

\subsection{Strategy of proof and organization}

The proofs are organized around one mechanism.  A Frobenius section in the
arithmetic point stabilizer descends the genus-one Galois closure to the
prescribed field.  The regular normal subgroup is then forced to act by
translations, producing a Frobenius-stable elliptic kernel.  The cyclic point
stabilizer descends only affinely, and its fixed-point torsor gives the shift
class.  The exact equivalence theorem records how this shift interacts with
marked geometric isomorphisms.

Once the datum is in place, exceptionality, actual permutation behavior over
every finite extension, constant fields, and base change reduce to the same
finite linear datum $(\beta,\mathcal F_{\mathcal N})$ on the translation
kernel.  An affine sector argument shows that the permutation and exceptionality
supports coincide, so no Weil-bound qualification remains in this family.  The
signature-specific elliptic models and marked normalizers then make the
occurrence and orbit counts explicit.  Only the stronger cubic same-$C_3$ realization requires an
additional saturation argument, leading to the $3$-adic boundary above.

Sections~\ref{sec:reduction-descent}--\ref{sec:converse} develop the common
descent, classification, and arithmetic mechanism.  Sections~\ref{sec:236-244}
and~\ref{sec:333} give the elliptic realizations, marked small-characteristic
normalizers, and the cubic realization boundary.  Section~\ref{sec:occurrence}
solves existence for prescribed parameters, and Section~\ref{sec:CM-counts}
completes the arithmetic closure by enumeration.  Section~\ref{sec:outlook}
records the remaining perspective and natural extensions.

\subsection{Notation and conventions}

Throughout the rest of the paper, $p=\operatorname{char}k$.  In the elliptic
realization, $H$ denotes the cyclic affine point stabilizer and
\[
 L_E(H):=\operatorname{im}\bigl(
 H\longrightarrow\operatorname{Aut}_{\bar k}(E,\mathrm O_E)
 \bigr)
\]
its origin-preserving linear image; when the datum is fixed, we abbreviate this
to $L_E$.  Let $\beta$ be a generator of $L_E(H)$ and write
\[
 \Delta(H):=E(\bar k)^{L_E(H)}=E[1-\beta].
\]
The affine descent of $H$ is measured by the shift class
$\mathfrak s_E(H)$, defined formally in Section~\ref{sec:affine}.  We write
$\mathcal F_E:E\to E$ for the $q$-power Frobenius endomorphism and, for a
Frobenius-stable point kernel $\mathcal N$, put
$\mathcal F_{\mathcal N}:=\mathcal F_E|_{\mathcal N}$.  Although
$\mathcal F_E$ is not an isomorphism of curves, its induced map on
$E(\bar k)$ is bijective; $\mathcal F_E^{-1}$ always denotes the inverse of
this group automorphism on geometric points.  We retain $\sigma_q$ for
arithmetic Galois Frobenius.  The integer $d_g$ is the order of the coset of
$\mathcal F_{\mathcal N}$ modulo $\langle\beta\rangle$; the precise
arithmetic definitions are given in Section~\ref{sec:converse}.  We use
\emph{fixed-field form} for a form over the prescribed base field $k$;
fixed fields of subgroups are always described explicitly as such.

\section{From exceptional covers to elliptic quotient data}
\label{sec:reduction-descent}

\subsection{The Euclidean reduction}
\label{sec:gms}

The first reduction removes a hypothesis that is redundant in characteristic
greater than $3$.

\begin{lemma}[Automatic tameness in characteristic greater than $3$]
\label{lem:automatic-tameness}
Let $k$ have characteristic $p>3$, let $f\in k(X)$ be separable, and suppose
that the smooth projective curve with function field the geometric Galois
closure of $\bar k(X)/\bar k(f(X))$ has genus one.  Then the geometric Galois
cover is tame.
\end{lemma}

\begin{proof}
Let $C/\bar k$ be the genus-one Galois-closure curve, and let $P\in C(\bar k)$.
The inertia group $I_P$ is the stabilizer of $P$ in the geometric Galois
group.  Taking $P$ as the origin makes $C$ an elliptic curve and identifies
$I_P$ with a finite subgroup of $\operatorname{Aut}(C,P)$.  In characteristic
$p>3$, the origin-preserving automorphism group of an elliptic curve has order
$2$, $4$, or $6$ \cite[Theorem~III.10.1]{Silverman09}.  Hence $p\nmid |I_P|$
for every $P$.  Thus every inertia group is prime to $p$, which is exactly
tameness of the Galois cover.
\end{proof}

The Guralnick--M\"uller--Saxl classification supplies the remaining
geometric reduction \cite{GMS03}.  Let $\Omega/k(\mathbf{t})$ be the
arithmetic Galois closure of $k(\mathbf{x})/k(\mathbf{t})$, where $\mathbf{t}=f(\mathbf{x})$, and put
\[
 A_f:=\operatorname{Gal}(\Omega/k(\mathbf{t})),\qquad
 k_\Omega:=\Omega\cap\bar k,\qquad
 G_f:=\operatorname{Gal}(\Omega/k_\Omega(\mathbf{t})).
\]
After extension of constants, this group identifies with the geometric monodromy
group $G_f$ defined above.

\begin{lemma}[Tame Euclidean reduction]
\label{lem:gms-reduction}
Let $k=\mathbb{F}_q$ and let $f\in k(X)$ be separable, indecomposable, tame, and
exceptional.  Suppose that the geometric Galois closure has genus one.  Then
there are an odd prime $\ell$ and $e\in\{1,2\}$ such that
\[
 \deg f=\ell^e,
\]
and
\[
 G_f=N\rtimes H,
 \qquad
 N\cong C_\ell^e,
\]
where $N$ acts regularly and $H$ is cyclic and semiregular on
$N\setminus\{0\}$.  The pair $(m,\text{ramification type})$ is one of the
four rows in \eqref{eq:four-types-intro}, and necessarily
$\operatorname{char}k\nmid m$.
\end{lemma}

\begin{proof}
Indecomposability of $f$ over $k$ makes the arithmetic monodromy group $A_f$
primitive in its natural degree-$\deg f$ action.  For the branch-cycle input,
GMS invoke Grothendieck's existence theorem \cite[\S2.8]{GMS03}.  They cite
SGA~1, Exp.~XIII, Cor.~2.12 \cite{SGA1} for the tame positive-characteristic
analogue of the characteristic-zero branch-cycle description.  Thus the
geometric inertia generators form a genus-zero system in the natural action
of $G_f$.  We use this transfer only for the group-theoretic candidate
reduction.  The finite-field exceptional-pair dictionary in the same section
identifies exceptionality of $f$ with exceptionality of the pair $(A_f,G_f)$,
while $A_f/G_f\cong\operatorname{Gal}(k_\Omega/k)$ is cyclic.

If an inertia generator has order $d$, then its index in the regular action is
\[
 |G_f|\left(1-\frac1d\right).
\]
Tame Riemann--Hurwitz for the geometric Galois closure therefore identifies
the regular genus of the GMS tuple with the genus of the top geometric curve,
namely one.  Applying \cite[Lemma~4.4]{GMS03} gives
\[
 \deg f=\ell^e,\qquad e\in\{1,2\},
\]
with $\ell$ odd, together with
\[
 G_f=N\rtimes H,
 \qquad N\cong C_\ell^e
\]
regular and $H$ cyclic.  The proof of that lemma also gives semiregularity of
$H$ on $N\setminus\{0\}$ and $\ell^e\equiv1\pmod{|H|}$, so
$\ell\nmid|H|$.  Hence $N$ is the unique Sylow-$\ell$ subgroup of $G_f$, is characteristic in
$G_f$, and is therefore normal in $A_f$.  Since $N$ is regular and $A_f$ is
primitive, no nonzero proper subgroup of $N$ can be $A_f$-invariant: its
orbits would be a nontrivial block system.  Thus $N$ is a regular elementary
abelian minimal normal subgroup of $A_f$, so $A_f$ is an affine primitive group.
The affine hypotheses of \cite[Theorem~4.14(b)]{GMS03} are therefore
automatic.  Since $A_f/G_f$ is cyclic and $(A_f,G_f)$ is exceptional, the auxiliary
exceptional group in that theorem may be taken to be $A_f$.  We then obtain the
four possibilities in \eqref{eq:four-types-intro}.  Since the
cover is tame, any row with $\operatorname{char}k\mid m$ is excluded.
\end{proof}

\begin{remark}
Lemma~\ref{lem:gms-reduction} itself does not require
$\operatorname{char}k>3$.  For $p>3$, Lemma~\ref{lem:automatic-tameness}
shows that its tameness hypothesis is automatic in the present genus-one
setting.  In characteristic $3$, tameness removes the rows $(3,3,3)$ and
$(2,3,6)$; in characteristic $2$, it removes every row with even inertia and
leaves only $(3,3,3)$.  All subsequent fixed-field descent arguments use only
the surviving prime-to-characteristic complement.
\end{remark}

For the covers considered in this paper, Lemma~\ref{lem:gms-reduction}
applies directly.  When $\operatorname{char}k>3$,
Lemma~\ref{lem:automatic-tameness} shows that the tameness assumption is
redundant.

\subsection{Descent of the genus-one Galois closure}
\label{sec:descent}

We next descend the genus-one Galois closure and its marked point stabilizer.
Let $k=\mathbb{F}_q$, let $f\in k(X)$ be separable, and write $\sigma_q$ for
arithmetic Frobenius.  Let
$\bar\Omega/\bar k(\mathbf{t})$ be the geometric Galois closure of
$\bar k(\mathbf{x})/\bar k(\mathbf{t})$ and put
\[
 G_f=\operatorname{Gal}(\bar\Omega/\bar k(\mathbf{t})),
 \qquad
 H=\operatorname{Gal}(\bar\Omega/\bar k(\mathbf{x})).
\]
Define
\[
 \widetilde G:=\operatorname{Aut}_{k(\mathbf{t})}(\bar\Omega),
 \qquad
 \widetilde H:=\{g\in\widetilde G:g(\mathbf{x})=\mathbf{x}\}.
\]

\begin{lemma}[Arithmetic exact sequences]
\label{lem:arithmetic-exact}
Restriction to the constants gives exact sequences
\begin{equation}
1\to G_f\to\widetilde G\to\Gamma_k\to1,
\qquad
1\to H\to\widetilde H\to\Gamma_k\to1,
\label{eq:exact-sequences}
\end{equation}
where $\Gamma_k=\operatorname{Gal}(\bar k/k)\cong\widehat{\mathbb Z}$.
\end{lemma}

\begin{proof}
The kernels are $G_f$ and $H$ by definition.  To prove surjectivity, let
$\sigma\in\Gamma_k$.  Since $f\in k(X)$, coefficientwise action of $\sigma$ on
$\bar k(\mathbf{x})$ fixes $\mathbf{x}$ and $\mathbf{t}=f(\mathbf{x})$.
This action extends to the normal closure $\bar\Omega$, so $\sigma$ has a lift
to $\widetilde G$.  If $\widetilde g$ is any lift, then
$\widetilde g(\mathbf{x})$ is another geometric conjugate of $\mathbf{x}$ over
$\bar k(\mathbf{t})$.  Transitivity of $G_f$ on these conjugates gives
$g\in G_f$ with
$g\widetilde g(\mathbf{x})=\mathbf{x}$.  Thus $g\widetilde g\in\widetilde H$ is
a lift of the same $\sigma$.
\end{proof}

\begin{lemma}[Procyclic point-stabilizer splitting]
\label{lem:procyclic-splitting}
There is a closed subgroup $B\subset\widetilde H$ such that restriction gives
$B\cong\Gamma_k$.  Moreover
\[
 \widetilde G=G_f\rtimes B,
 \qquad
 \widetilde H=H\rtimes B.
\]
\end{lemma}

\begin{proof}
Choose a lift $\widetilde F\in\widetilde H$ of $\sigma_q$.  The homomorphism
$\mathbb Z\to\widetilde H$ sending $1$ to $\widetilde F$ is continuous for the
profinite topology: modulo every open normal subgroup the image of
$\widetilde F$ has finite order.  It therefore extends uniquely to a continuous
homomorphism $\widehat{\mathbb Z}\to\widetilde H$.  Composing this map with restriction to $\Gamma_k$ gives the identity; hence
its image $B$ is a closed section.  The decompositions follow from Lemma~\ref{lem:arithmetic-exact} and
$B\cap G_f=1$.
\end{proof}

Put
\[
 K:=\bar\Omega^B.
\]

\begin{proposition}[Descent of stable quotients]
\label{prop:stable-quotients}
One has
\begin{equation}
 K\cap\bar k=k,
 \qquad
 K\bar k=\bar\Omega,
 \qquad
 [K:k(\mathbf{t})]=|G_f|.
\label{eq:K-descent}
\end{equation}
If $J\le G_f$ is normalized by $B$, then
\[
 K_J:=(\bar\Omega^J)^B=\bar\Omega^{JB}
\]
is a regular $k$-form of $\bar\Omega^J$, with
\[
 [K_J:k(\mathbf{t})]=[G_f:J].
\]
In particular,
\begin{equation}
 (\bar\Omega^H)^B=k(\mathbf{x}).
\label{eq:H-descent}
\end{equation}
\end{proposition}

\begin{proof}
The subgroups of $\widetilde G$ that fix $K$ and $\bar k(\mathbf{t})$ are
$B$ and $G_f$, respectively.  Since $B\cap G_f=1$, their compositum is $\bar\Omega$, which proves the
second equality in \eqref{eq:K-descent}; the first follows because
$B\to\Gamma_k$ is an isomorphism.  Finally,
$[\widetilde G:B]=|G_f|$, so Galois correspondence gives
$[K:k(\mathbf{t})]=|G_f|$.

If $J$ is $B$-stable, then $JB$ is a subgroup and
\[
 (\bar\Omega^J)^B=\bar\Omega^{JB}.
\]
Moreover, $(JB)\cap G_f=J$, so extension of constants to $\bar k$ recovers
$\bar\Omega^J$.  Since $JB\supseteq B$, one has $K_J\subseteq K$, and hence
$K_J\cap\bar k=k$.  Also
$[K_J:k(\mathbf{t})]=[\widetilde G:JB]=[G_f:J]$.  For $J=H$, the group $HB=\widetilde H$ fixes precisely $k(\mathbf{x})$,
giving \eqref{eq:H-descent}.
\end{proof}

\begin{remark}
For $|H|>2$, the formula \eqref{eq:H-descent} cannot in general be replaced
by $K^H=k(\mathbf{x})$.  The section $B$ need only normalize $H$, and the individual
geometric automorphisms in $H$ need not preserve $K$.  The canonical descended
object is the cyclic subgroup together with its quotient.
\end{remark}

Assume now that the geometric Galois closure has genus one.  Let $C_f/k$ be
the smooth projective curve with function field $K$; it has genus one.  The
Hasse--Weil bound \cite[Theorem~5.2.3]{Stichtenoth09} gives
\[
 |C_f(k)|\geqslant q+1-2\sqrt q>0.
\]
Choose $\mathrm O_E\in C_f(k)$ and write
\[
 E:=(C_f,\mathrm O_E).
\]
Thus $E/k$ is an elliptic curve with geometric function field $\bar\Omega$.

\subsection{The regular subgroup is a translation kernel}

On the descended genus-one curve, the regular normal subgroup $N$ has the
Euclidean GMS form
\[
 G_f=N\rtimes H,
 \qquad
 N\cong C_\ell^e,
 \qquad
 H\cong C_m.
\]
Here Roman $N$ denotes the regular subgroup of the geometric monodromy group;
calligraphic $\mathcal N$ will denote its point kernel on $E$.  Since
$\ell^e\equiv1\pmod m$, one has $\ell\nmid m$.  The subgroup $N$ is
the unique Sylow-$\ell$ subgroup of $G_f$, hence characteristic in $G_f$ and
therefore normalized by $B$.

\begin{lemma}[Fixed-point freeness]
\label{lem:N-free}
Every nonidentity element of $N$ acts fixed-point freely on $E_{\bar k}$.
\end{lemma}

\begin{proof}
For $P\in E(\bar k)$, the stabilizer $G_P$ is a geometric inertia group of the
top Galois cover.  Every nontrivial inertia order in the four Euclidean
signatures divides $m$, whereas $N$ is an $\ell$-group and $\ell\nmid m$.
Hence $G_P\cap N=1$ for every $P$.
\end{proof}

\begin{lemma}[Fixed-point-free elliptic automorphisms]
\label{lem:fp-free-translation}
Let $\alpha\in\operatorname{Aut}(E_{\bar k})$.  If $\alpha$ has no geometric fixed point,
then $\alpha$ is a nontrivial translation.
\end{lemma}

\begin{proof}
Put $Q=\alpha(\mathrm O_E)$ and
$\alpha_0=\tau_{-Q}\alpha$.  Then $\alpha_0$ fixes the origin and is therefore a group
automorphism \cite[Theorem~III.4.8]{Silverman09}; thus
$\alpha(P)=\alpha_0(P)+Q$.  If $\alpha_0\ne1$, then $1-\alpha_0$ is a nonzero endomorphism.  It
is a nonconstant morphism of complete curves and hence is surjective
\cite[Theorem~II.2.3]{Silverman09}.  Therefore $(1-\alpha_0)P=Q$ has a solution,
which is a fixed point of $\alpha$.  Hence $\alpha_0=1$.
\end{proof}

\begin{theorem}[Translation kernel]
\label{thm:translation-kernel}
There is a finite Frobenius-stable subgroup
\[
 \mathcal{N}\subset E(\bar k),
 \qquad
 \mathcal{N}\cong C_\ell^e,
\]
whose associated subgroup scheme is reduced, such that
\[
 N=T_{\mathcal{N}}.
\]
The quotient by $N$ descends to a separable $k$-isogeny
\[
 \varphi:E\longrightarrow E':=E/\mathcal{N}
\]
of degree $\ell^e$.
If $e=2$, then $\operatorname{char}k\ne\ell$,
$\mathcal{N}=E[\ell](\bar k)$, and $\varphi=\theta\circ[\ell]$ for a
$k$-isomorphism $\theta:E\to E'$.
\end{theorem}

\begin{proof}
By Lemmas~\ref{lem:N-free} and~\ref{lem:fp-free-translation}, every element
of $N$ is a translation.  Hence
\[
 N\longrightarrow E(\bar k),\qquad \tau_P\longmapsto P
\]
is an injective homomorphism; let $\mathcal{N}$ be its image.  Since $B$ normalizes
$N$ and $\mathrm O_E\in E(k)$,
\[
 {}^{\sigma_q}\tau_P=\tau_{\sigma_q(P)},
\]
so $\mathcal{N}$ is Frobenius-stable.

Proposition~\ref{prop:stable-quotients} applied to $N$ gives a $k$-morphism
$\varphi:E\to E'$ whose geometric form is the free quotient by
$T_{\mathcal{N}}$.  The kernel is reduced, so the quotient is separable.  The standard quotient
theorem for elliptic curves gives the $k$-isogeny structure
\cite[Proposition~III.4.12 and Remark~III.4.13.2]{Silverman09}.  Its degree is
$|\mathcal{N}|=\ell^e$.

If $e=2$, then $\mathcal{N}\cong C_\ell^2\subset E[\ell](\bar k)$.  In
characteristic $\ell$, the reduced geometric $\ell$-torsion has at most
$\ell$ points \cite[Corollary~III.6.4(c)]{Silverman09}; hence
$\operatorname{char}k\ne\ell$.  Both $\mathcal{N}$ and $E[\ell](\bar k)$ then
have $\ell^2$ points, so they coincide.  The two separable isogenies
$\varphi$ and $[\ell]$ have the same kernel, and the universal property of the
quotient gives a unique $k$-isomorphism $\theta:E\xrightarrow{\sim}E'$ with
$\varphi=\theta\circ[\ell]$.
\end{proof}

\begin{proposition}[The induced complement]
\label{prop:induced-complement}
The action of $H$ on $E_{\bar k}$ induces a Frobenius-stable cyclic subgroup
$H'\subset\operatorname{Aut}(E'_{\bar k})$ of order $m$ such that
\[
 \varphi h=h'\varphi\qquad(h\in H).
\]
Moreover, the descended quotients have function fields $k(\mathbf{x})$ and $k(\mathbf{t})$,
respectively.  Hence, after identifying both quotient curves with $\mathbf{P}^1_k$,
the induced lower map is two-sided $k$-M\"obius equivalent to the original
$f$.
\end{proposition}

\begin{proof}
Because $H$ normalizes $N=T_{\mathcal{N}}$, it permutes the $N$-orbits and hence
acts on the geometric quotient $E'=E/N$.  If an element $h\in H$ induces the
identity on $E'$, then $\varphi h=\varphi$; since $\varphi$ is the geometric
Galois quotient with deck group $N$, this forces $h\in N$.  As
$H\cap N=1$, the induced action of $H$ is faithful, so $H'\cong H\cong C_m$.
Frobenius-stability follows from that of $H$ and $N$ and from the fact that
$\varphi$ is defined over $k$.

For the source quotient, Proposition~\ref{prop:stable-quotients} gives
$(\bar\Omega^H)^B=k(\mathbf{x})$.  Geometrically,
$E'/H'=E/(N\rtimes H)=E/G_f$, and its descended function field is
\[
 (\bar\Omega^{G_f})^B=(\bar k(\mathbf{t}))^B=k(\mathbf{t}).
\]
Both quotient curves are geometrically rational and contain the image of a
$k$-point; hence they are $k$-isomorphic to $\mathbf{P}^1$
\cite[Proposition~1.6.3]{Stichtenoth09}.  The resulting extension of function
fields is therefore exactly $k(\mathbf{x})/k(\mathbf{t})$, up to the two chosen $k$-rational
coordinates on the quotient lines.
\end{proof}

The preceding construction motivates the abstract object used from this point
onward.  We now package precisely the properties needed independently of the
original cover.

\begin{definition}[Euclidean quotient datum]
\label{def:datum}
A \emph{Euclidean quotient datum} over $k$ is a triple
$\mathcal D=(E,\mathcal N,H)$ consisting of an elliptic curve $E/k$, a finite
Frobenius-stable subgroup $\mathcal N\subset E(\bar k)$ isomorphic to
$C_\ell^e$ for an odd prime $\ell$ and $e\in\{1,2\}$, whose associated finite
subgroup scheme is reduced, and a Frobenius-stable cyclic affine subgroup
$H\subset\operatorname{Aut}(E_{\bar k})$ of order
$m\in\{2,3,4,6\}$ with $\operatorname{char}k\nmid m$.  Let
\[
 L_E(H):=\operatorname{im}\!\left(
 H\longrightarrow\operatorname{Aut}(E_{\bar k},\mathrm O_E)
 \right)
\]
denote the image of the linear-part homomorphism.  We require this
homomorphism to be faithful, $L_E(H)\mathcal N=\mathcal N$ (equivalently,
$L_E(H)$ normalizes $T_{\mathcal N}=\{\tau_P:P\in\mathcal N\}$), and
$L_E(H)$ to act semiregularly on $\mathcal N\setminus\{0\}$.
\end{definition}

\subsection{Affine linearization and the shift class}
\label{sec:affine}

From this point to the end of the section, let
$\mathcal D=(E,\mathcal N,H)$ be a Euclidean quotient datum.  Write
$H=\langle h\rangle\cong C_m$ and, uniquely,
\[
 h(P)=\beta(P)+Q,
 \qquad
 \beta\in\operatorname{Aut}(E_{\bar k},\mathrm O_E).
\]

\begin{proposition}[Geometric linearization]
\label{prop:linearization}
The automorphism $h$ has a geometric fixed point $P_H$, and
\begin{equation}
 Q=(1-\beta)P_H,
 \qquad
 h=\tau_{P_H}\beta\tau_{-P_H}.
\label{eq:h-linearized}
\end{equation}
In particular, $\operatorname{ord}(\beta)=m$.
\end{proposition}

\begin{proof}
If $h$ were fixed-point free, Lemma~\ref{lem:fp-free-translation} would make
it a nontrivial translation, whose linear part is the identity.  This
contradicts faithfulness of the linear-part map in
Definition~\ref{def:datum}.  Thus $h(P_H)=P_H$ for some $P_H$, which yields
\eqref{eq:h-linearized}; conjugate automorphisms have the same order.
\end{proof}

\begin{remark}
\label{rem:faithful-linear-part}
For the data recovered from a Euclidean cover earlier in this section, the
conditions in Definition~\ref{def:datum} are automatic.  Under
$N=T_{\mathcal N}$ one has
\[
 h\tau_P h^{-1}=\tau_{\beta(P)}\qquad(P\in\mathcal N),
\]
so the semiregular conjugation action of $H$ on $N\setminus\{0\}$ from
Lemma~\ref{lem:gms-reduction} is exactly the action of $L_E(H)$ on
$\mathcal N\setminus\{0\}$.  In particular, $L_E(H)$ preserves
$\mathcal N$ and acts semiregularly there, and the linear-part map is
faithful: a nonidentity element in its kernel would act trivially on
$N\setminus\{0\}$.  Together with Theorem~\ref{thm:translation-kernel}, this
verifies Definition~\ref{def:datum} for the recovered triple.  More
generally, faithfulness excludes nontrivial pure translations; together with
Proposition~\ref{prop:linearization}, it identifies $H$, up to geometric
translation conjugacy, with its linear image $L_E(H)$, and hence
$E/H\cong E/L_E(H)\cong\mathbf P^1_{\bar k}$.
\end{remark}

For a Euclidean quotient datum $\mathcal D=(E,\mathcal N,H)$, let
\[
 \varphi_{\mathcal D}:E\longrightarrow E'_{\mathcal D}:=E/\mathcal N
\]
be the quotient isogeny.  Since $L_E(H)$ preserves $\mathcal N$, the group
$H$ normalizes $T_{\mathcal N}$ and induces a Frobenius-stable cyclic affine
group $H'_{\mathcal D}\subset\operatorname{Aut}((E'_{\mathcal D})_{\bar k})$
with $\varphi_{\mathcal D}H=H'_{\mathcal D}\varphi_{\mathcal D}$.  This
induced action is faithful: an element of $H$ acting trivially on
$E'_{\mathcal D}$ lies in $T_{\mathcal N}$, and hence is trivial by
faithfulness of the linear-part map.  If $P_H$ is the fixed point supplied by
Proposition~\ref{prop:linearization}, then $\varphi_{\mathcal D}(P_H)$ is
fixed by the induced generator of $H'_{\mathcal D}$.  Thus both cyclic affine
actions are geometrically translation-conjugate to nontrivial linear cyclic
actions, so both quotient curves are geometrically rational.

Writing $\pi_H$ and $\pi_{H'_{\mathcal D}}$ for the quotient maps, let
$\bar g_{\mathcal D}$ be the unique morphism making
\[
\begin{tikzcd}
 E \arrow[r,"\varphi_{\mathcal D}"] \arrow[d,"\pi_H"'] & E'_{\mathcal D} \arrow[d,"\pi_{H'_{\mathcal D}}"]\\
 E/H \arrow[r,"\bar g_{\mathcal D}"'] & E'_{\mathcal D}/H'_{\mathcal D}
\end{tikzcd}
\]
commute.  The diagram is defined over $k$.  Each lower quotient contains the
image of a $k$-rational origin, hence both genus-zero curves are
$k$-isomorphic to $\mathbf P^1$ \cite[Proposition~1.6.3]{Stichtenoth09}.
Choosing $k$-coordinates gives $g_{\mathcal D}\in k(X)$; its two-sided class
\[
 \mathcal M(\mathcal D):=[g_{\mathcal D}]_{\sim_k}
\]
is independent of the chosen coordinates.  We call $\bar g_{\mathcal D}$ the
canonical lower morphism and $g_{\mathcal D}$ an associated lower rational
map.  For the datum recovered above, Proposition~\ref{prop:induced-complement}
identifies this canonical lower morphism with the original lower extension,
so $g_{\mathcal D}\sim_k f$.  This proves the existence assertion announced
in Theorem~A.

For the surviving tame rows one has
\begin{equation}
 |E[1-\beta]|=
 \begin{cases}
 4,&m=2,\\
 3,&m=3,\\
 2,&m=4,\\
 1,&m=6.
 \end{cases}
\label{eq:fixed-counts}
\end{equation}
Indeed, for $m=2$ one has $1-\beta=[2]$; for $m=3$,
$(1-\beta)(1-\beta^2)=[3]$ and the two factors have equal degree; for $m=4$,
$(1-\beta)(1+\beta)=[2]$ with equal degrees; and for $m=6$,
$\beta^2-\beta+1=0$, whence $\beta(1-\beta)=1$.  The corresponding degrees are $4,3,2,1$.  In each surviving tame row they are
prime to the characteristic, so the endomorphisms are separable and their
geometric kernels have exactly these cardinalities.  When
$\operatorname{char}k>3$, the usual automorphism classification further gives
$\boldsymbol j(E)$ arbitrary for $m=2$, $\boldsymbol j(E)=0$ for $m=3,6$, and
$\boldsymbol j(E)=1728$ for $m=4$ \cite[Theorem~III.10.1]{Silverman09}.

\begin{proposition}[Ramification of the cyclic quotient]
\label{prop:cyclic-signature}
The quotient $E\to E/H$ has ramification type
\[
\begin{array}{c|c}
 m&\text{type}\\ \hline
 2&(2,2,2,2),\\
 3&(3,3,3),\\
 4&(2,4,4),\\
 6&(2,3,6).
\end{array}
\]
\end{proposition}

\begin{proof}
By Proposition~\ref{prop:linearization}, translation conjugacy reduces the
calculation to $H=\langle\beta\rangle$.  For $m=2$, the four points of
$E[2]$ are the four points with stabilizer $C_2$.  For $m=3$, the generator
has three fixed points, giving three branch points of index $3$.  For $m=4$,
the generator has two fixed points, giving two index-$4$ branch points; its
square is $[-1]$ and fixes four points, and the two remaining points form one
$H$-orbit with stabilizer $C_2$, giving the index-$2$ branch point.  For
$m=6$, the generator has one fixed point, its square has three fixed points,
and its cube $[-1]$ has four fixed points.  Removing points already counted,
the two additional order-three fixed points form one orbit with stabilizer
$C_3$, and the three additional order-two fixed points form one orbit with
stabilizer $C_2$.  This gives $(2,3,6)$.
\end{proof}

\subsubsection{Cyclotomic Frobenius}

Frobenius-stability of $H$ means
\[
 {}^{\sigma_q}h=h^j
\]
for a unique $j\in(\mathbb Z/m\mathbb Z)^\times$.

\begin{theorem}[Prime-to-characteristic cyclotomic action]
\label{thm:cyclotomic}
For every Euclidean quotient datum one has
\[
 j\equiv q\pmod m,
\]
so
\[
 {}^{\sigma_q}h=h^q,
 \qquad
 {}^{\sigma_q}\beta=\beta^q.
\]
Consequently, a chosen generator is defined over
$\mathbb{F}_{q^{\operatorname{ord}_m(q)}}$.
\end{theorem}

\begin{proof}
Comparing linear parts gives ${}^{\sigma_q}\beta=\beta^j$.  Since
$\operatorname{char}k\nmid m$, the tangent character of
$\langle\beta\rangle$ is faithful.  Indeed, if a nontrivial element $\gamma$
of order $n$ prime to the characteristic had derivative $1$, then in a local
parameter $t$ at the origin one could write
\[
 \gamma(t)=t+c_rt^r+O(t^{r+1}),
 \qquad c_r\ne0,
\]
with $r>1$.  Iteration gives
$\gamma^n(t)=t+n c_rt^r+O(t^{r+1})$, contradicting $\gamma^n=1$ because
$n$ is invertible in $k$.  Thus $\beta$ acts on the tangent space by a
primitive $m$th root $\zeta_m$.  Frobenius conjugation sends this eigenvalue
to $\zeta_m^q$, while $\beta^j$ acts by $\zeta_m^j$.  Hence
$\zeta_m^q=\zeta_m^j$ and $j\equiv q\pmod m$.
\end{proof}

\subsubsection{The arithmetic shift class}

Recall that $L_E(H)$ denotes the linear image of $H$.  By
Proposition~\ref{prop:linearization}, the linear-part map identifies $H$ with
a cyclic origin-preserving group of order $m$.  Put
\[
 \Delta(H):=E(\bar k)^{L_E(H)}.
\]
If $\beta$ is any generator of $L_E(H)$, then
\[
 \Delta(H)=E[1-\beta];
\]
in particular the subgroup is independent of the chosen generator.  If $P_H$
is a common fixed point of $H$, then
\[
 \operatorname{Fix}(H)=P_H+\Delta(H).
\]
Since Frobenius normalizes $H$, one has
$\sigma_q(P_H)-P_H\in\Delta(H)$, and changing $P_H$ changes this difference by a
coboundary.

\begin{definition}
The \emph{shift class} (or \emph{shift-descent class}) of $H$ is
\begin{equation}
 \mathfrak{s}_E(H):=[\sigma_q(P_H)-P_H]
 \in \Delta(H)/(\sigma_q-1)\Delta(H).
\label{eq:shift-class}
\end{equation}
\end{definition}

\begin{proposition}
\label{prop:shift-class}
The class \eqref{eq:shift-class} vanishes if and only if $H$ has a common fixed
point in $E(k)$, equivalently if and only if a $k$-rational translation
conjugates $H$ to $L_E(H)$.  For $m=6$, the class always vanishes.
\end{proposition}

\begin{proof}
If $P_H'=P_H+T$ with $T\in\Delta(H)$, then
\[
 \sigma_q(P_H')-P_H'=\sigma_q(P_H)-P_H+(\sigma_q-1)T,
\]
so the class is well defined.  It vanishes exactly when one can choose
$T\in\Delta(H)$ with $P_H+T\in E(k)$.  For $m=6$,
$\Delta(H)=0$ by \eqref{eq:fixed-counts}.
\end{proof}

We next identify the shift class on the branch locus of the lower rational
map.  For brevity write
\[
 \varphi=\varphi_{\mathcal D},\qquad
 E'=E'_{\mathcal D},\qquad
 H'=H'_{\mathcal D}.
\]
Since $|\mathcal{N}|=\ell^e$ is coprime to $|\Delta(H)|$---by the degree
conditions in \eqref{eq:four-types-intro} and
$|\Delta(H)|\in\{4,3,2,1\}$---the restriction of $\varphi$ to $\Delta(H)$ is
injective.  Equivariance sends it into $\Delta(H')$, and the two fixed-point
groups have the same cardinality by \eqref{eq:fixed-counts}.  Thus
\[
 \varphi:\Delta(H)\xrightarrow{\sim}\Delta(H')
\]
is a Frobenius-equivariant isomorphism.  It carries the fixed-point torsor of
$H$ to that of $H'$ and therefore carries $\mathfrak s_E(H)$ to
$\mathfrak s_{E'}(H')$.  In particular, the two classes vanish
simultaneously.

\begin{theorem}[Rational branch criterion]
\label{thm:rational-branches}
Let $\mathcal D=(E,\mathcal N,H)$ be a Euclidean quotient datum, let
$g=g_{\mathcal D}\in k(X)$ be an associated lower rational map, and let
$\operatorname{Br}_d(g)$ denote the branch points whose
inertia group in the geometric Galois closure has order $d$.  Then:
\begin{enumerate}[label=\textup{(\roman*)}]
\item for type $(2,2,2,2)$,
\[
 \operatorname{Br}(g)\cap\mathbf{P}^1(k)\ne\varnothing
 \quad\Longleftrightarrow\quad
 \mathfrak s_E(H)=0;
\]
\item for type $(3,3,3)$,
\[
 \operatorname{Br}(g)\cap\mathbf{P}^1(k)\ne\varnothing
 \quad\Longleftrightarrow\quad
 \mathfrak s_E(H)=0;
\]
\item for type $(2,4,4)$, the unique point of
$\operatorname{Br}_2(g)$ belongs to $\mathbf{P}^1(k)$, and
\[
 \operatorname{Br}_4(g)\subset\mathbf{P}^1(k)
 \quad\Longleftrightarrow\quad
 \mathfrak s_E(H)=0.
\]
If the class is nonzero, the two order-$4$ branch points form a single
Frobenius orbit of length $2$;
\item for type $(2,3,6)$, all three branch points belong to $\mathbf{P}^1(k)$.
\end{enumerate}
\end{theorem}

\begin{proof}
Because the isogeny $E\to E'$ is separable, it is unramified.  Thus the branch
locus of the top Galois cover
\[
 E\longrightarrow E'/H'
\]
is exactly the branch locus of the cyclic quotient $E'\to E'/H'$.  Every one
of these top branch points remains branched in the lower degree-$|\mathcal{N}|$
cover.  Indeed, let
$I'=\langle (h')^j\rangle\le H'$ be the stabilizer of a point above such a
branch point in the cyclic quotient, and let
$I=\langle h^j\rangle\le H$ be the corresponding subgroup under
$H\xrightarrow{\sim}H'$.  The inertia subgroup in the top
$T_{\mathcal{N}}\rtimes H$-cover is conjugate to $I$, and its linear part is
$\langle\beta^j\rangle$.  Semiregularity of $L_E(H)$ on
$\mathcal{N}\setminus\{0\}$ implies that, in the lower coset action
\[
 (T_{\mathcal{N}}\rtimes H)/H\cong\mathcal{N},
\]
this inertia subgroup has one fixed point and all other orbits have length
$|I|>1$.  Hence the lower map is ramified over the same branch point.  It
follows that $\operatorname{Br}(g)$ is precisely the branch locus of
$E'\to E'/H'$.

In types $(2,2,2,2)$ and $(3,3,3)$, every branch point of this quotient is
the image of a unique common fixed point of $H'$.  Since the quotient map is
defined over $k$, such a branch point is $k$-rational exactly when its unique
point above it is $k$-rational.  Proposition~\ref{prop:shift-class}
and the Frobenius-equivariant identification of the two shift classes then
give \textup{(i)} and \textup{(ii)}.

In type $(2,4,4)$, the two inertia-$4$ branch points likewise correspond
bijectively to the two common fixed points of $H'$.  Here
$\Delta(H')\cong C_2$, so once the fixed-point torsor has one $k$-point it
has two: the nonzero element of $\Delta(H')$ is fixed by Frobenius.  Thus the
two inertia-$4$ branch points are both $k$-rational exactly when the shift
class vanishes.  If it does not vanish, Frobenius has no fixed point on this
two-element set and therefore interchanges its two elements.
The remaining branch point is the unique branch point with inertia order $2$;
Frobenius preserves inertia order, so this point is $k$-rational.

For type $(2,3,6)$, the three inertia orders $2,3,6$ are pairwise distinct.
Frobenius preserves inertia order, hence fixes each of the three branch points.
This also agrees with $\Delta(H)=0$ and the automatic vanishing of the shift
class.
\end{proof}

\section{Affine descent and classification of fixed-field forms}
\label{sec:forms}

Section~\ref{sec:reduction-descent} recovers Euclidean quotient data and
constructs their associated lower maps.  We now determine exactly when two
such data define the same two-sided $k$-M\"obius class, separating geometric
equivalence of the marked data from the affine shift that survives over $k$.

\begin{theorem}[Fixed-field forms]
\label{thm:intro-forms}
Let $k=\mathbb F_q$ be arbitrary, and let
$f\in k(X)$ be separable, indecomposable, tame, and exceptional.  Suppose
that the geometric Galois closure of $f$ has genus one.
\begin{enumerate}[label=\textup{(\roman*)}]
\item There exists a Euclidean quotient datum
$\mathcal D=(E,\mathcal N,H)$ over $k$ such that
\[
 g_{\mathcal D}\sim_k f.
\]
Under the resulting realization of the geometric Galois closure on $E$, the
regular normal subgroup is
\[
 N=T_{\mathcal N}.
\]
If $e=2$, then $\operatorname{char}k\ne\ell$,
$\mathcal N=E[\ell](\bar k)$, and
$\varphi_{\mathcal D}=\theta\circ[\ell]$ for a $k$-isomorphism
$\theta:E\xrightarrow{\sim}E'_{\mathcal D}$.

\item Conversely, for every Euclidean quotient datum
$\mathcal D=(E,\mathcal N,H)$ over $k$, an associated lower rational map
$g_{\mathcal D}$ is separable of degree $\ell^e$, has geometric monodromy
$T_{\mathcal N}\rtimes H$, and has the Euclidean ramification type determined
by $m$.

\item Let $\mathcal D_i=(E_i,\mathcal N_i,H_i)$, $i=1,2$, be Euclidean
quotient data of the same degree and type.  Write
$L_{E_i}=L_{E_i}(H_i)$ and $\Delta_i=\Delta(H_i)$, and let
$\mathcal F_{E_i}$ denote $q$-power Frobenius on $E_i(\bar k)$.  Then
\[
 \mathcal M(\mathcal D_1)=\mathcal M(\mathcal D_2)
\]
if and only if there is an origin-preserving geometric isomorphism
$\alpha:E_{1,\bar k}\xrightarrow{\sim}E_{2,\bar k}$.  Define its Frobenius
conjugate by
\[
 {}^{\sigma_q}\alpha:=\mathcal F_{E_2}\circ\alpha\circ\mathcal F_{E_1}^{-1}.
\]
The required conditions are
\begin{enumerate}[label=\textup{(\alph*)}]
\item $\alpha(\mathcal N_1)=\mathcal N_2$;
\item $\alpha L_{E_1}\alpha^{-1}=L_{E_2}$;
\item $({}^{\sigma_q}\alpha)\alpha^{-1}\in L_{E_2}$;
\item
\[
 \alpha_*\mathfrak s_{E_1}(H_1)=\mathfrak s_{E_2}(H_2)
 \quad\text{in }\Delta_2/(\sigma_q-1)\Delta_2.
\]
\end{enumerate}
Here conditions \textup{(b)}--\textup{(c)} make
$\alpha|_{\Delta_1}:\Delta_1\to\Delta_2$ Frobenius-equivariant, and
$\alpha_*$ denotes the induced map
\[
 \Delta_1/(\sigma_q-1)\Delta_1
 \longrightarrow
 \Delta_2/(\sigma_q-1)\Delta_2.
\]
\end{enumerate}
\end{theorem}

Part \textup{(i)} of Theorem~\ref{thm:intro-forms} was proved in
Section~\ref{sec:reduction-descent}; part \textup{(iii)} is proved below, and
the converse in part \textup{(ii)} is established in Section~\ref{sec:converse}.

\begin{lemma}[Core-free affine complement]
\label{lem:core-free-affine-complement}
Let $\mathcal D=(E,\mathcal N,H)$ be a Euclidean quotient datum and put
\[
 G_{\mathcal D}:=\langle T_{\mathcal N},H\rangle.
\]
Then
\[
 G_{\mathcal D}=T_{\mathcal N}\rtimes H,
 \qquad
 \operatorname{core}_{G_{\mathcal D}}(H)=1.
\]
Moreover, the map
\[
 \mathcal N\longrightarrow G_{\mathcal D}/H,
 \qquad P\longmapsto\tau_PH,
\]
is a bijection.  Under this identification $T_{\mathcal N}$ acts regularly by
translations, while $H$ acts through its linear image $L_E(H)$.
\end{lemma}

\begin{proof}
The group $H$ normalizes $T_{\mathcal N}$ because its linear image preserves
$\mathcal N$, and faithfulness of the linear-part map gives
$H\cap T_{\mathcal N}=1$.  Hence
$G_{\mathcal D}=T_{\mathcal N}\rtimes H$.

Choose $0\ne P\in\mathcal N$.  If
$x\in H\cap\tau_PH\tau_{-P}$, write
$x=h^j=\tau_Ph^a\tau_{-P}$ for a generator $h$ of $H$, with linear part
$\beta$.  Comparing linear parts gives $\beta^j=\beta^a$, hence
$a\equiv j\pmod m$ by faithfulness.  Thus $h^j$ commutes with $\tau_P$, so
$\beta^j(P)=P$.  Semiregularity of $L_E(H)$ on
$\mathcal N\setminus\{0\}$ then forces $h^j=1$.  Therefore
$H\cap\tau_PH\tau_{-P}=1$, and the core of $H$ is trivial.

Every coset has a unique representative $\tau_P$, giving the stated
bijection.  Finally,
$\tau_Q(\tau_PH)=\tau_{P+Q}H$, while
$h\tau_Ph^{-1}=\tau_{\beta(P)}$, which proves the asserted sheet action.
\end{proof}

\begin{lemma}[Lifting isomorphisms to normal closures]
\label{lem:normal-closure-lifting}
Let $K_i\subset L_i$ be finite separable field extensions, let
$\Omega_i/K_i$ be the normal closure of $L_i/K_i$, and suppose that there are
compatible field isomorphisms
\[
 \kappa:K_1\xrightarrow{\sim}K_2,
 \qquad
 \lambda:L_1\xrightarrow{\sim}L_2,
 \qquad
 \lambda|_{K_1}=\kappa.
\]
Then $\lambda$ extends to an isomorphism
$\widetilde\lambda:\Omega_1\xrightarrow{\sim}\Omega_2$.  If
$H_i=\operatorname{Gal}(\Omega_i/L_i)$, then
\[
 \widetilde\lambda H_1\widetilde\lambda^{-1}=H_2.
\]
For fixed $\lambda$, the set of such extensions is a torsor under $H_2$ by
postcomposition.  Equivalently, if $\widetilde\lambda_1$ and
$\widetilde\lambda_2$ induce the same map on $L_1$, then
\[
 \widetilde\lambda_2\widetilde\lambda_1^{-1}\in H_2,
\]
and every element of $H_2$ arises in this way.  Dually, compatible
isomorphisms of finite separable covers lift to their normal closures, with
the same target-stabilizer ambiguity.
\end{lemma}

\begin{proof}
Identify $K_1$ with $K_2$ through $\kappa$.  The embedding $\lambda$ extends
to an isomorphism of algebraic closures.  Since $\Omega_i$ is the compositum
of the conjugates of $L_i$ over $K_i$, the extension carries $\Omega_1$ onto
$\Omega_2$.  The fixed-field identities $\Omega_i^{H_i}=L_i$ then give
$\widetilde\lambda H_1\widetilde\lambda^{-1}=H_2$.  Finally, two
extensions of the same $\lambda$ differ by an automorphism of $\Omega_2$
fixing $L_2$, hence by an element of $H_2$; conversely, postcomposition by an
element of $H_2$ preserves the restriction to $L_1$.
\end{proof}

It remains only to prove the exact equivalence criterion in
Theorem~\ref{thm:intro-forms}\textup{(iii)}.

\begin{proof}[Proof of Theorem~\ref{thm:intro-forms}\textup{(iii)}]

\noindent\emph{Necessity.}  Assume
\[
 \mathcal M(\mathcal D_1)=\mathcal M(\mathcal D_2).
\]
By Lemma~\ref{lem:core-free-affine-complement}, each $H_i$ is core-free in
$T_{\mathcal N_i}\rtimes H_i$.  Since the normal closure of the intermediate
cover corresponds to the core of its point stabilizer, $E_i$ is the geometric
Galois closure of the corresponding lower cover.

Choose a representative two-sided equivalence
$g_{\mathcal D_1}=\mu\circ g_{\mathcal D_2}\circ\eta$.  The source
transformation $\eta$, together with the target identification induced by
$\mu$, gives compatible isomorphisms of the target and source function
fields.  By Lemma~\ref{lem:normal-closure-lifting}, these lift, dually, to a
geometric isomorphism
\[
 \widetilde\alpha:E_{1,\bar k}\xrightarrow{\sim}E_{2,\bar k}
\]
with
\[
 \widetilde\alpha H_1\widetilde\alpha^{-1}=H_2.
\]  Since $T_{\mathcal{N}_i}$ is the unique Sylow-$\ell$ subgroup of
$T_{\mathcal{N}_i}\rtimes H_i$, it follows that
\[
 \widetilde\alpha T_{\mathcal{N}_1}\widetilde\alpha^{-1}=T_{\mathcal{N}_2}.
\]
Write uniquely
\[
 \widetilde\alpha=\tau_Q\alpha,
\]
with $\alpha(\mathrm O_{E_1})=\mathrm O_{E_2}$.  Conjugating translations gives
$\widetilde\alpha\tau_P\widetilde\alpha^{-1}=\tau_{\alpha(P)}$, hence \textup{(a)}.  Comparing linear
parts in $\widetilde\alpha H_1\widetilde\alpha^{-1}=H_2$ gives \textup{(b)}.

Let $\pi_i:E_i\to E_i/H_i$ be the descended source quotient maps.  The
lower source isomorphism is defined over $k$, so
${}^{\sigma_q}\widetilde\alpha$ and $\widetilde\alpha$ induce the
same source intermediate-field isomorphism.  The torsor assertion in
Lemma~\ref{lem:normal-closure-lifting} therefore gives
\[
 h_\sigma:=({}^{\sigma_q}\widetilde\alpha)\widetilde\alpha^{-1}\in H_2,
\]
not merely in its normalizer.  Taking linear parts yields
\[
 ({}^{\sigma_q}\alpha)\alpha^{-1}\in L_{E_2},
\]
which is \textup{(c)}.

Choose $P_1\in\operatorname{Fix}(H_1)$ and put
$P_2=\widetilde\alpha(P_1)\in\operatorname{Fix}(H_2)$.  Since
${}^{\sigma_q}\widetilde\alpha=h_\sigma\widetilde\alpha$ and every element of $H_2$ fixes every point of
$\operatorname{Fix}(H_2)$, we obtain
\[
 \sigma_q(P_2)=\widetilde\alpha(\sigma_qP_1).
\]
The translation part of $\widetilde\alpha$ cancels in differences, hence
\[
 \sigma_q(P_2)-P_2
 =\alpha(\sigma_qP_1-P_1),
\]
which gives \textup{(d)}.

It remains to note that \textup{(c)} makes the restriction of $\alpha$ to
the fixed-point group Galois-equivariant.  Indeed, if $P_\Delta\in\Delta_1$, then
$\alpha(P_\Delta)\in\Delta_2$ by \textup{(b)}, while
$({}^{\sigma_q}\alpha)\alpha^{-1}\in L_{E_2}$ fixes $\Delta_2$ pointwise.  Hence
\[
 \sigma_q(\alpha(P_\Delta))
 =({}^{\sigma_q}\alpha)(\sigma_qP_\Delta)
 =\alpha(\sigma_qP_\Delta).
\]

\medskip\noindent\emph{Sufficiency.}  Assume \textup{(a)}--\textup{(d)}.  Choose
$P_i\in\operatorname{Fix}(H_i)$ and set
\[
 s_i:=\sigma_q(P_i)-P_i\in\Delta_i.
\]
Condition \textup{(d)} gives $Q_\Delta\in\Delta_2$ such that
\[
 s_2-\alpha(s_1)=(\sigma_q-1)Q_\Delta.
\]
Replace $P_2$ by $P_2':=P_2-Q_\Delta$.  Since $Q_\Delta\in\Delta_2$, the point $P_2'$
is still fixed by $H_2$, and now
\begin{equation}
 \sigma_q(P_2')-P_2'=\alpha(\sigma_qP_1-P_1).
\label{eq:twisted-cocycle-match}
\end{equation}
Put
\[
 Q:=P_2'-\alpha(P_1),
 \qquad
 \widetilde\alpha:=\tau_Q\alpha.
\]
Then \textup{(a)}--\textup{(b)} give
\[
 \widetilde\alpha T_{\mathcal{N}_1}\widetilde\alpha^{-1}=T_{\mathcal{N}_2},
 \qquad
 \widetilde\alpha H_1\widetilde\alpha^{-1}=H_2.
\]
Define
\[
 h_\alpha:=\tau_{P_2'}({}^{\sigma_q}\alpha)\alpha^{-1}\tau_{-P_2'}\in H_2.
\]
We claim that
\begin{equation}
 {}^{\sigma_q}\widetilde\alpha=h_\alpha\widetilde\alpha.
\label{eq:twisted-descent-identity}
\end{equation}
Indeed,
\[
 {}^{\sigma_q}\widetilde\alpha=\tau_{\sigma_qQ}\,({}^{\sigma_q}\alpha),
\]
whereas the translation part of $h_\alpha\widetilde\alpha$ is
\[
 P_2'-({}^{\sigma_q}\alpha)\alpha^{-1}(P_2')+({}^{\sigma_q}\alpha)\alpha^{-1}(Q)
=P_2'-({}^{\sigma_q}\alpha)(P_1).
\]
Using \eqref{eq:twisted-cocycle-match}, we compute
\[
 \sigma_qQ
 =P_2'-({}^{\sigma_q}\alpha)(P_1)
  +\alpha(s_1)-({}^{\sigma_q}\alpha)\alpha^{-1}(\alpha(s_1)).
\]
Now $\alpha(s_1)\in\Delta_2$, and
$({}^{\sigma_q}\alpha)\alpha^{-1}\in L_{E_2}$ fixes $\Delta_2$ pointwise.  Hence the last two terms cancel, proving
\eqref{eq:twisted-descent-identity}.

Since $h_\alpha\in H_2$, equation \eqref{eq:twisted-descent-identity} shows that
$\widetilde\alpha$ induces a $k$-defined isomorphism
\[
 E_1/H_1\xrightarrow{\sim}E_2/H_2.
\]
Condition \textup{(a)} gives a unique induced isomorphism
\[
 \bar\alpha:E_1/\mathcal N_1\xrightarrow{\sim}E_2/\mathcal N_2
\]
satisfying
\[
 \varphi_{\mathcal D_2}\circ\widetilde\alpha
 =\bar\alpha\circ\varphi_{\mathcal D_1}.
\]
Since $\widetilde\alpha H_1\widetilde\alpha^{-1}=H_2$, we also have
$\bar\alpha H_1'\bar\alpha^{-1}=H_2'$.  Passing
\eqref{eq:twisted-descent-identity} to the quotient by
$T_{\mathcal N_2}$ gives
\[
 {}^{\sigma_q}\bar\alpha=h_\alpha'\bar\alpha
\]
for the image $h_\alpha'\in H_2'$.  Hence $\bar\alpha$ induces a $k$-defined
isomorphism of the target quotients.  Together with the already descended
source quotient isomorphism, the displayed quotient identity makes the
induced square commute.  Thus the two lower maps are two-sided
$k$-M\"obius equivalent, proving
$\mathcal M(\mathcal D_1)=\mathcal M(\mathcal D_2)$.
\end{proof}

\begin{remark}
For $m=2$, condition \textup{(c)} becomes
${}^{\sigma_q}\alpha=\pm\alpha$, while $\Delta(H)=E[2]$; thus the theorem
specializes to the familiar signed-descent pattern for shifted involutions.
For $m=6$, $\Delta(H)=0$ and condition \textup{(d)} disappears.  The
orders $3$ and $4$ leave fixed-point torsors with structure groups $C_3$ and
$C_2$, respectively; the shift class lives on these torsors.
\end{remark}

\subsection{Marked descent and finite orbit spaces}
\label{sec:marked-orbits}

The equivalence theorem packages all later form and shift counts into one
finite orbit problem.  Fix a geometric marked pair $(E,L_E)$, where
$L_E\leqslant\operatorname{Aut}(E_{\bar k},\mathrm O_E)$ is one of the cyclic linear
parts above, and put
\[
 U^\sharp:=N_{\operatorname{Aut}(E_{\bar k},\mathrm O_E)}(L_E),
 \qquad
 W:=U^\sharp/L_E.
\]
Choose a reference descent and let $F_0$ denote its semilinear arithmetic
Frobenius.  Conjugation by $F_0$ induces an automorphism $\theta$ of $W$.
For $c\in U^\sharp$ write $F_c=cF_0$ and let $\bar c$ denote its image in
$W$.  On $W$ use Frobenius-twisted conjugacy
\[
 \bar c'\sim_\theta\bar c
 \quad\Longleftrightarrow\quad
 \bar c'=\bar u\,\bar c\,\theta(\bar u)^{-1}
 \quad\text{for some }\bar u\in W.
\]
For a representative $c$ define its twisted stabilizer
\[
 R_c:=\{\bar u\in W:
 \bar u\,\bar c\,\theta(\bar u)^{-1}=\bar c\}.
\]
Let $\Delta=E(\bar k)^{L_E}$.  The shift module in this stratum is
\[
 S_c:=\Delta/(F_c-1)\Delta.
\]
Finally let $K_c$ be any $R_c$-stable family of finite kernels on the marked
curve that are both $F_c$-stable and $L_E$-stable.  This formulation is
deliberately independent of exceptionality; in the counting sections $K_c$
will be specialized to the good kernels of a prescribed rank.

\begin{proposition}[Marked twisted-orbit principle]
\label{prop:marked-orbit}
The map-level marked $k$-forms of $(E,L_E)$ are parametrized by the
$\theta$-twisted conjugacy classes in $W$.  For a fixed stratum represented by
$c$, the corresponding two-sided classes with kernel in $K_c$ are naturally
parametrized by
\begin{equation}
 (K_c\times S_c)/R_c.
 \label{eq:marked-orbit-space}
\end{equation}
Consequently,
\begin{equation}
 |(K_c\times S_c)/R_c|
 =\frac1{|R_c|}\sum_{u\in R_c}|K_c^u|\,|S_c^u|.
 \label{eq:marked-burnside}
\end{equation}
\end{proposition}

\begin{proof}
Represent a marked descent by the cocycle value $c\in U^\sharp$ on arithmetic
Frobenius.  Replacing the geometric identification by $u\in U^\sharp$ changes
$c$ to $u c\,{}^{\sigma_q}u^{-1}$, while condition
\textup{(c)} of Theorem~\ref{thm:intro-forms}\textup{(iii)} allows an additional
factor in $L_E$.  Modulo $L_E$ this is precisely the twisted
conjugacy relation above.

It remains to verify that every class in $S_c$ is realized by an
$F_c$-stable affine subgroup with linear part $L_E$.  Given $s\in\Delta$,
Lang's theorem for the finite-field form defined by $F_c$ gives
$P\in E(\bar k)$ with
\[
 (F_c-1)P=s.
\]
Put $H_P=\tau_P L_E\tau_{-P}$.  Since
$s\in\Delta=E(\bar k)^{L_E}$, one has
\[
 F_cH_PF_c^{-1}=H_P.
\]
Replacing $P$ by $P+T$ with $T\in\Delta$ does not change $H_P$ and changes
the representative $s$ by $(F_c-1)T$.  Conversely, if $H_P=H_{P'}$, uniqueness of the element with a prescribed
linear part gives $(1-\lambda)(P-P')=0$ for every $\lambda\in L_E$; hence
$P-P'\in\Delta$.  Thus the affine shift classes in the stratum are exactly
\[
 \Delta/(F_c-1)\Delta=S_c.
\]

Fix a stratum $c$.  If $\bar u\in R_c$, then for a lift $u\in U^\sharp$ there
is $\lambda\in L_E$ such that
$uF_cu^{-1}F_c^{-1}=\lambda$.  Since $L_E$ fixes $\Delta$ pointwise, $u$
commutes with $F_c$ on $\Delta$ and hence acts on $S_c$.  The same identity
shows that $u$ transports every $F_c$-stable, $L_E$-stable kernel to
another such kernel.  Replacing $u$ by a different lift modulo $L_E$
does not change either action: $L_E$ fixes $\Delta$ pointwise and
preserves every $L_E$-stable kernel.  Thus the prescribed $R_c$-action on $K_c\times S_c$ is well defined and
compatible with descent.  Conditions \textup{(a)} and \textup{(d)} of
Theorem~\ref{thm:intro-forms}\textup{(iii)} then identify exactly the pairs in
the same $R_c$-orbit.  This proves \eqref{eq:marked-orbit-space};
\eqref{eq:marked-burnside} is Burnside's lemma.
\end{proof}

\section{Arithmetic monodromy and exact behavior over finite extensions}
\label{sec:converse}

The fixed-field classification leaves a finite arithmetic readout on the
translation kernel.  The next theorem collects its monodromy, constant-field,
and exact extension consequences; the equality between permutation and
exceptionality is proved by the affine sector argument later in this section.

\begin{theorem}[Arithmetic of the fixed-field datum]
\label{thm:intro-arithmetic}
Let $\mathcal D=(E,\mathcal N,H)$ be a Euclidean quotient datum over $k$,
and let $g=g_{\mathcal D}$ be an associated lower rational map.  Choose source
and target coordinates $\mathbf{x}_g,\mathbf{t}_g$ with
$\mathbf{t}_g=g(\mathbf{x}_g)$.  Let
$\Omega_g/k(\mathbf t_g)$ be the arithmetic Galois closure and put
\[
 k_{\Omega_g}:=\Omega_g\cap\bar k,\qquad
 A_g:=\operatorname{Gal}(\Omega_g/k(\mathbf t_g)),\qquad
 G_g:=\operatorname{Gal}(\Omega_g/k_{\Omega_g}(\mathbf t_g)).
\]
In the natural action on the sheets, let $\widehat H\le A_g$ be the arithmetic
point stabilizer and $\widehat H_0:=\widehat H\cap G_g$ its geometric part.
Write $L_E=L_E(H)=\langle\beta\rangle$ and
$\mathcal F_{\mathcal N}=\mathcal F_E|_{\mathcal N}$.  Under the descended
realization, coefficientwise arithmetic $q$-Frobenius gives an element of
$\widehat H$ whose action on translations is
$\tau_P\mapsto\tau_{\mathcal F_{\mathcal N}(P)}$.  In particular,
$\mathcal F_{\mathcal N}\beta\mathcal F_{\mathcal N}^{-1}=\beta^q$.
\begin{enumerate}[label=\textup{(\roman*)}]
\item If $H=\langle h\rangle$, then the linear part $\beta$ has exact order
$m$ and
\[
 {}^{\sigma_q}h=h^q,\qquad {}^{\sigma_q}\beta=\beta^q.
\]
The class $\mathfrak s_E(H)$ vanishes exactly when $H$ is conjugate by a
$k$-rational translation to its linear part.  In types $(2,2,2,2)$ and
$(3,3,3)$ a $k$-rational branch point exists exactly when this class vanishes;
in type $(2,4,4)$ the inertia-$2$ branch point is always $k$-rational and
the two inertia-$4$ points are rational exactly when the class vanishes; in
type $(2,3,6)$ all branch points are $k$-rational and the class vanishes.

\item Put
\[
 M:=\langle\beta,\mathcal F_{\mathcal N}\rangle\le\operatorname{GL}(\mathcal N).
\]
Then $g$ is indecomposable over $k$ if and only if $M$ acts irreducibly on
$\mathcal N$, and it is exceptional if and only if
\[
 \ker(\mathcal F_{\mathcal N}-\beta^j)|_{\mathcal N}=0
 \qquad(0\leqslant j<m).
\]
In rank one, if $\mathcal F_{\mathcal N}$ and $\beta$ act by $\lambda$ and $\rho$,
respectively, this is equivalent to $\lambda^m\ne1$.

\item The arithmetic and geometric monodromy groups are
\[
 G_g\cong\mathcal N\rtimes\langle\beta\rangle,
 \qquad
 A_g\cong\mathcal N\rtimes M
 =\mathcal N\rtimes\langle\beta,\mathcal F_{\mathcal N}\rangle.
\]
If
\[
 d_g:=\min\{\nu\geqslant1:\mathcal F_{\mathcal N}^\nu\in\langle\beta\rangle\},
\]
then
\[
 [A_g:G_g]=d_g,
 \qquad
 k_{\Omega_g}=\mathbb F_{q^{d_g}}.
\]
For $r\geqslant1$, put $k_r:=\mathbb F_{q^r}$ and write $A_g^{(r)}$ for the
arithmetic monodromy group of the same rational map over $k_r$.  Then
\[
 [A_g^{(r)}:G_g]=\frac{d_g}{\gcd(d_g,r)}.
\]

\item Define
\[
 \operatorname{Supp}_{\mathrm{exc}}(g)
 :=\{r\geqslant1:g/k_r\text{ is exceptional}\},
\]
and
\[
 \operatorname{Supp}_{\mathrm{perm}}(g)
 :=\{r\geqslant1:
 g:\mathbf P^1(k_r)\to\mathbf P^1(k_r)
 \text{ is bijective}\}.
\]
For every $r\geqslant1$,
\[
\begin{aligned}
 r\in\operatorname{Supp}_{\mathrm{perm}}(g)
 &\Longleftrightarrow
 r\in\operatorname{Supp}_{\mathrm{exc}}(g)\\
 &\Longleftrightarrow
 \ker(\mathcal F_{\mathcal N}^r-\beta^j)|_{\mathcal N}=0
 \quad(0\leqslant j<m).
\end{aligned}
\]
Thus
$\operatorname{Supp}_{\mathrm{perm}}(g)=\operatorname{Supp}_{\mathrm{exc}}(g)$.
This common support is periodic modulo $d_g$ and contains no positive multiple of $d_g$.
If $g/k$ is exceptional, it contains every $r$ coprime to $d_g$.  If $e=1$,
then
\[
 d_g=\operatorname{ord}_{\mathbb F_\ell^\times}(\lambda^m),
 \qquad
 \operatorname{Supp}_{\mathrm{perm}}(g)=\operatorname{Supp}_{\mathrm{exc}}(g)
 =\{r\geqslant1:d_g\nmid r\}.
\]
\end{enumerate}
\end{theorem}

Part \textup{(i)} is Theorem~\ref{thm:rational-branches}.  Parts
\textup{(ii)}--\textup{(iii)} follow from the finite linear action of
$\langle\beta,\mathcal F_{\mathcal N}\rangle$ on $\mathcal N$, while part
\textup{(iv)} combines the same linear criterion with the affine sector
argument developed below.

\subsection{Converse, indecomposability, and exceptionality}

The exact equivalence theorem classifies the data once they are realized.  We
first establish the converse direction: every admissible datum produces a cover
of the asserted degree and ramification type.

Let $\mathcal D=(E,\mathcal N,H)$ be a Euclidean quotient datum over $k$,
and let $g=g_{\mathcal D}$ be an associated lower rational map.  Choose coordinates
$\mathbf{x}_g,\mathbf{t}_g$ with $\mathbf{t}_g=g(\mathbf{x}_g)$.

\begin{theorem}[Converse]
\label{thm:converse}
For a Euclidean quotient datum, every associated lower rational map
$g=g_{\mathcal D}\in k(X)$ is separable of degree $\ell^e$ and has geometric
monodromy group
\[
 T_{\mathcal{N}}\rtimes H.
\]
Its geometric ramification type is the row corresponding to $m$ in
\eqref{eq:four-types-intro}.
\end{theorem}

\begin{proof}
Over $\bar k$, the group $T_{\mathcal N}\rtimes H$ acts on $E$, and the fixed
field of $H$ is the source quotient.  By
Lemma~\ref{lem:core-free-affine-complement}, $H$ is core-free, so $E$ is the
geometric Galois closure of the lower extension.  The degree is $|\mathcal N|$.  All extensions here are fixed fields of finite
groups of automorphisms and are therefore separable.  The translation quotient
is unramified, so all ramification comes from the cyclic affine quotient.  Proposition~\ref{prop:cyclic-signature}
gives exactly the four Euclidean signatures.
\end{proof}

This proves Theorem~\ref{thm:intro-forms}\textup{(ii)}.

For the arithmetic analysis, use
$\Omega_g,k_{\Omega_g},A_g,G_g,\widehat H,\widehat H_0$ as in
Theorem~\ref{thm:intro-arithmetic}.  Put
$\mathcal F_{\mathcal N}:=\mathcal F_E|_{\mathcal N}$.  By
Lemma~\ref{lem:arithmetic-exact}, coefficientwise arithmetic $q$-Frobenius
has a lift in the arithmetic source stabilizer, and on translations it acts by
$\tau_P\mapsto\tau_{\mathcal F_{\mathcal N}(P)}$.  Since $H$ is
Frobenius-stable,
\begin{equation}
 \mathcal F_{\mathcal N}\beta\mathcal F_{\mathcal N}^{-1}=\beta^q.
\label{eq:normalizer-linear}
\end{equation}

\begin{lemma}[Faithful point-stabilizer action on the kernel]
\label{lem:arithmetic-stabilizer-kernel}
The translation subgroup $T_{\mathcal N}$ is normal in $A_g$.  Conjugation on
$T_{\mathcal N}\cong\mathcal N$ is faithful on the arithmetic point
stabilizer $\widehat H$ and identifies
\[
 \widehat H_0\cong L_E(H)=\langle\beta\rangle,
 \qquad
 \widehat H\cong
 M:=\langle\beta,\mathcal F_{\mathcal N}\rangle.
\]
Under this identification the coefficientwise arithmetic $q$-Frobenius lift
maps to $\mathcal F_{\mathcal N}$.
\end{lemma}

\begin{proof}
By Theorem~\ref{thm:converse},
$G_g=T_{\mathcal N}\rtimes\widehat H_0$ in the natural sheet action.  Since
$|\mathcal N|=\ell^e$ and $\ell\nmid m$, the subgroup $T_{\mathcal N}$ is the
unique Sylow-$\ell$ subgroup of $G_g$.  It is therefore characteristic in
$G_g$ and hence normal in $A_g$.

Lemma~\ref{lem:core-free-affine-complement} identifies the geometric sheets
with $\mathcal N$.  Conjugation by the geometric stabilizer acts on
translations through the linear image $L_E(H)=\langle\beta\rangle$, so
$\widehat H_0$ has this image.  Since $G_g$ is transitive,
$A_g=G_g\widehat H$ and
\[
 \widehat H/\widehat H_0\cong A_g/G_g.
\]
The latter cyclic quotient is generated by the coefficientwise arithmetic
$q$-Frobenius class.  Lemma~\ref{lem:arithmetic-exact} supplies a lift in
$\widehat H$, and its conjugation action on translations is
$\tau_P\mapsto\tau_{\mathcal F_{\mathcal N}(P)}$.  Hence the conjugation image
of $\widehat H$ is exactly
$M=\langle\beta,\mathcal F_{\mathcal N}\rangle$.

It remains to prove faithfulness.  If $a\in\widehat H$ centralizes every
translation, then $a$ fixes the base sheet and, for every $P\in\mathcal N$,
\[
 a(\tau_P\omega_0)
 =a\tau_Pa^{-1}(a\omega_0)
 =\tau_P\omega_0.
\]
Thus $a$ fixes every sheet.  The monodromy action of the normal closure is
faithful, so $a=1$.  The displayed identifications follow.
\end{proof}

\begin{theorem}[Indecomposability and exceptionality]
\label{thm:exceptionality}
For the lower map $g=g_{\mathcal D}$ constructed from a Euclidean quotient
datum, put
\[
 M:=\langle\beta,\mathcal F_{\mathcal N}\rangle\le\operatorname{GL}(\mathcal{N}).
\]
Then $g$ is indecomposable over $k$ if and only if $M$ acts irreducibly on
$\mathcal{N}$.  It is exceptional if and only if
\begin{equation}
 \ker(\mathcal F_{\mathcal N}-\beta^j)|_{\mathcal{N}}=0
 \qquad(0\leqslant j<m).
\label{eq:exceptionality-final}
\end{equation}
\end{theorem}

\begin{proof}
By Lemma~\ref{lem:arithmetic-stabilizer-kernel},
$T_{\mathcal N}\triangleleft A_g$ is regular on the sheets and
$\widehat H\cong M=\langle\beta,\mathcal F_{\mathcal N}\rangle$.  Since $G_g$
is transitive, $A_g=T_{\mathcal N}\rtimes\widehat H$, so the arithmetic
monodromy group is affine with regular elementary abelian normal subgroup
$\mathcal N$.  Such an affine permutation group is primitive exactly when its
point stabilizer is irreducible on $\mathcal N$.  By Galois correspondence
together with L\"uroth's theorem for intermediate fields of
$k(\mathbf{x}_g)/k(\mathbf{t}_g)$, primitivity of the arithmetic monodromy
action is equivalent to indecomposability over $k$.

For exceptionality, use the finite-field exceptional-pair criterion
\cite[\S2.8 and Lemma~3.3]{GMS03}; see also the General Exceptionality Lemma
of \cite[\S10]{FGS93}.  In the affine identification with $\mathcal{N}$, the
nontrivial geometric point-stabilizer orbits are precisely the nonzero
$L_E(H)=\langle\beta\rangle$-orbits.  Moreover
\[
 \mathcal F_{\mathcal N}(\langle\beta\rangle P)=\langle\beta\rangle\mathcal F_{\mathcal N} P
\]
because $\mathcal F_{\mathcal N}\beta\mathcal F_{\mathcal N}^{-1}=\beta^q$.  Hence a nonzero
$\langle\beta\rangle$-orbit is also an orbit of the arithmetic point
stabilizer $M=\langle\beta,\mathcal F_{\mathcal N}\rangle$ exactly when Frobenius carries that
geometric orbit to itself.  This occurs precisely when
\[
 \mathcal F_{\mathcal N}P=\beta^jP
\]
for some $P\ne0$ and some $j$.  Therefore the geometric and arithmetic point
stabilizers have no common nontrivial orbit exactly when
\eqref{eq:exceptionality-final} holds.
\end{proof}

\begin{corollary}[Rank-one Frobenius criterion]
\label{cor:rank-one-exceptionality}
Assume $e=1$.  Write the actions of $\mathcal F_{\mathcal N}$ and $\beta$ on
$\mathcal{N}\cong\mathbb{F}_\ell$ as multiplication by $\lambda$ and $\rho$,
respectively, where $\rho$ has order $m$.  Then
\[
 g\text{ is exceptional}
 \quad\Longleftrightarrow\quad
 \lambda\notin\langle\rho\rangle
 \quad\Longleftrightarrow\quad
 \lambda^m\ne1\quad\text{in }\mathbb{F}_\ell.
\]
For $m=2$ this is $\lambda\ne\pm1$.
\end{corollary}

\begin{proof}
On the one-dimensional $\mathbb{F}_\ell$-space $\mathcal{N}$, condition
\eqref{eq:exceptionality-final} says precisely that
$\lambda\ne\rho^j$ for every $j$.  Since $\rho$ has order $m$, its
powers are exactly the $m$th roots of unity in $\mathbb{F}_\ell^\times$.
\end{proof}

\subsection{Arithmetic monodromy and constant fields}

The Frobenius coset modulo $\langle\beta\rangle$ controls not only
exceptionality but also arithmetic monodromy and the constant field.  After
extending constants, $G_g$ identifies with the geometric monodromy
group of Theorem~\ref{thm:converse}.  The standard quotient
$A_g/G_g\cong\operatorname{Gal}(k_{\Omega_g}/k)$ is cyclic; compare
\cite[\S2.8]{GMS03}.

\begin{theorem}[Arithmetic monodromy and constants]
\label{thm:arithmetic-monodromy}
Let $g$ be the lower rational map attached to a Euclidean quotient datum
$\mathcal D=(E,\mathcal{N},H)$.  Let
\[
 L_E(H)=\langle\beta\rangle,\qquad
 M:=\langle\beta,\mathcal F_{\mathcal N}\rangle\le\operatorname{GL}(\mathcal{N}),
\]
where
$\mathcal F_{\mathcal N}=\mathcal F_E|_{\mathcal N}$.  Since
$\mathcal F_{\mathcal N}\beta\mathcal F_{\mathcal N}^{-1}=\beta^q$, the
subgroup $\langle\beta\rangle$ is normal in $M$.  Then, in their natural permutation representations on the sheets,
\[
 G_g\cong\mathcal{N}\rtimes\langle\beta\rangle,\qquad
 A_g\cong\mathcal{N}\rtimes M.
\]
Consequently
\[
 A_g/G_g\cong M/\langle\beta\rangle.
\]
Define
\begin{equation}
 d_g:=
 \operatorname{ord}_{M/\langle\beta\rangle}
       \bigl(\mathcal F_{\mathcal N}\langle\beta\rangle\bigr)
 =
 \min\{\nu\geqslant1:\mathcal F_{\mathcal N}^\nu\in\langle\beta\rangle\}.
\label{eq:dg-definition}
\end{equation}
Then
\begin{equation}
 [A_g:G_g]=d_g,\qquad
 k_{\Omega_g}=\mathbb{F}_{q^{d_g}}.
\label{eq:constant-field}
\end{equation}

For $r\geqslant1$, let $k_r=\mathbb{F}_{q^r}$ and regard the same rational map over
$k_r$.  Under the common geometric permutation representation its arithmetic
monodromy group is
\begin{equation}
 A_g^{(r)}
 \cong
 \mathcal{N}\rtimes M_r,\qquad
 M_r:=\langle\beta,\mathcal F_{\mathcal N}^r\rangle.
\label{eq:Ar}
\end{equation}
The geometric monodromy group remains $G_g$.  The constant field of
the base-changed arithmetic Galois closure is
\begin{equation}
 k_{\Omega_g,r}
 =
 k_{\Omega_g}k_r
 =
 \mathbb{F}_{q^{\operatorname{lcm}(d_g,r)}},
\label{eq:constant-field-r}
\end{equation}
and hence
\begin{equation}
 [A_g^{(r)}:G_g]
 =
 [k_{\Omega_g,r}:k_r]
 =
 \frac{d_g}{\gcd(d_g,r)}.
\label{eq:index-r}
\end{equation}
In particular,
\[
 d_g\mid r
 \quad\Longleftrightarrow\quad
 A_g^{(r)}=G_g.
\]
\end{theorem}

\begin{proof}
By Theorem~\ref{thm:converse}, the geometric monodromy group is
$G_g=T_{\mathcal N}\rtimes\widehat H_0$ in the natural sheet action.  Since
$G_g$ is transitive,
\[
 A_g=G_g\widehat H=T_{\mathcal N}\rtimes\widehat H.
\]
Lemma~\ref{lem:arithmetic-stabilizer-kernel} identifies
$\widehat H_0\cong\langle\beta\rangle$ and $\widehat H\cong M$.  Therefore
\[
 G_g\cong\mathcal N\rtimes\langle\beta\rangle,
 \qquad
 A_g\cong\mathcal N\rtimes M,
\]
and consequently $A_g/G_g\cong M/\langle\beta\rangle$.

The exact arithmetic/geometric monodromy sequence identifies
\[
 A_g/G_g\cong\operatorname{Gal}(k_{\Omega_g}/k).
\]
Under the preceding identification its Frobenius generator is
$\mathcal F_{\mathcal N}\langle\beta\rangle$.  Its order is therefore $d_g$, proving
\eqref{eq:constant-field}.

Now let $\Omega_{g,r}:=\Omega_gk_r$.  After adjoining $k_r$, the normal closure is generated over $k_r(\mathbf{t}_g)$ by the same
geometric conjugates of $\mathbf{x}_g$; therefore $\Omega_{g,r}$ is the arithmetic Galois closure over $k_r(\mathbf{t}_g)$.
Since $k_{\Omega_g}$ is the full constant field of $\Omega_g$, the extension
$\Omega_g/k_{\Omega_g}(\mathbf{t}_g)$ is regular; hence the full constant field
after base change is
\[
 k_{\Omega_g,r}:=\Omega_{g,r}\cap\bar k
 =k_{\Omega_g}k_r.
\]
After base change to $k_r$, Lemma~\ref{lem:arithmetic-stabilizer-kernel} applies
with arithmetic Frobenius $\mathcal F_{\mathcal N}^r$ and gives
\eqref{eq:Ar}.  Since finite fields have a unique
extension of each degree,
\[
 k_{\Omega_g}k_r
 =
 \mathbb{F}_{q^{\operatorname{lcm}(d_g,r)}}.
\]
Taking the degree over $k_r$ gives \eqref{eq:index-r}.  The final equivalence
follows because this index equals $1$ exactly when $d_g\mid r$.
\end{proof}

\begin{corollary}
\label{cor:constants-coprime}
If $\gcd(r,d_g)=1$, then
\[
 M_r:=\langle\beta,\mathcal F_{\mathcal N}^r\rangle=M.
\]
Consequently, under the common affine identification on $\mathcal N$, the
permutation images of $A_g^{(r)}$ and $A_g$ coincide.  If $d_g\mid r$, then
$A_g^{(r)}=G_g$.
\end{corollary}

\begin{proof}
The second assertion is part of
Theorem~\ref{thm:arithmetic-monodromy}.  If $\gcd(r,d_g)=1$, then
$\mathcal F_{\mathcal N}^r\langle\beta\rangle$ generates the same cyclic
quotient $M/\langle\beta\rangle$ as
$\mathcal F_{\mathcal N}\langle\beta\rangle$.  Hence
$M_r=\langle\beta,\mathcal F_{\mathcal N}^r\rangle=M$, and the assertion about
permutation images follows from \eqref{eq:Ar}.
\end{proof}

\begin{remark}
\label{rem:dg-exceptional}
If $g$ is exceptional and $\deg g>1$, then $d_g>1$.  Otherwise
$\mathcal F_{\mathcal N}\in\langle\beta\rangle$, contrary to
Theorem~\ref{thm:exceptionality}.
\end{remark}

\subsection{Permutation and exceptionality over finite extensions}

The finite-extension arithmetic is stronger than the usual eventual
exceptionality--permutation comparison.  For $k_r=\mathbb F_{q^r}$, define
the \emph{exceptionality support}
\begin{equation}
 \operatorname{Supp}_{\mathrm{exc}}(g):=
 \{\,r\geqslant1:g/k_r\text{ is exceptional}\,\},
\label{eq:supp-exc-def}
\end{equation}
and the \emph{permutation support}
\[
 \operatorname{Supp}_{\mathrm{perm}}(g):=
 \{\,r\geqslant1:
 g:\mathbf P^1(k_r)\to\mathbf P^1(k_r)
 \text{ is bijective}\,\}.
\]

\begin{lemma}[Affine sector mass]
\label{lem:affine-sector-mass}
Let $\mathcal D=(E,\mathcal N,H)$ be a Euclidean quotient datum, put
$m=|H|$, fix $r\geqslant1$, and let $\pi:E\to E/H$ be the quotient map.  For
$\gamma\in H$ write
\[
 \gamma(P)=u_{\gamma,E}(P)+a_{\gamma,E},
 \qquad
 \mathcal S_{r,\gamma}(E)
 :=\{P\in E(\bar k):\mathcal F_E^r(P)=\gamma(P)\}.
\]
Then every $\mathcal S_{r,\gamma}(E)$ is nonempty and is a torsor under
$\ker(\mathcal F_E^r-u_{\gamma,E})$.  For $Q\in(E/H)(k_r)$ put
\[
 m_Q(\gamma)
 :=\bigl|\mathcal S_{r,\gamma}(E)\cap\pi^{-1}(Q)\bigr|,
\]
and, for $P\in\pi^{-1}(Q)$, put $A_Q=\operatorname{Stab}_H(P)$.  Then
\begin{equation}
 \sum_{\gamma\in H}m_Q(\gamma)=m.
 \label{eq:affine-sector-mass}
\end{equation}
If $m_Q(\gamma)>0$, then
\begin{equation}
 m_Q(\gamma)=\gcd\!\left(\frac{m}{|A_Q|},q^r-1\right),
 \qquad
 m_Q(\gamma)\mid\frac{m}{|A_Q|}\mid m.
 \label{eq:affine-sector-divisibility}
\end{equation}
\end{lemma}

\begin{proof}
The endomorphism $\mathcal F_E^r-u_{\gamma,E}$ is nonzero and separable:
its differential is $-du_{\gamma,E}\ne0$, because Frobenius has zero
differential and $u_{\gamma,E}$ is an automorphism.  Hence it is surjective,
so the displayed affine fibre is a nonempty torsor under its kernel.  Let
$P\in\pi^{-1}(Q)$.  Rationality of $Q$ gives one sector label
$\gamma_0$, the labels at $P$ form the coset $\gamma_0A_Q$, and the fibre has
$m/|A_Q|$ points.  This count also covers ramified fibres: the smaller orbit
size $m/|A_Q|$ is exactly offset by the $|A_Q|$ labels carried by each point.
Double counting point-label incidences gives
\eqref{eq:affine-sector-mass}.

Now assume $m_Q(\gamma)>0$ and choose $P$ in that sector fibre.  By
Theorem~\ref{thm:cyclotomic},
\[
 {}^{\sigma_{q^r}}b=b^{q^r}\qquad(b\in H).
\]
Since $\gcd(q,m)=1$, this Frobenius action preserves every subgroup of the
cyclic group $H$, in particular $A_Q$.  Moreover, $bP$ remains in the same
sector exactly when $b^{-1}{}^{\sigma_{q^r}}b=b^{q^r-1}$ lies in $A_Q$.
Therefore
\[
 m_Q(\gamma)=
 \left|\ker\bigl(H/A_Q\to H/A_Q,\ bA_Q\mapsto b^{q^r-1}A_Q\bigr)\right|.
\]
Since $H/A_Q$ is cyclic of order $m/|A_Q|$, this is
\eqref{eq:affine-sector-divisibility}.
\end{proof}

\begin{lemma}[Affine-sector permutation criterion]
\label{lem:sector-permutation}
Let $\mathcal D=(E,\mathcal N,H)$ be a Euclidean quotient datum, let
$\varphi:E\to E'=E/\mathcal N$ be the quotient isogeny, and let $H'$ be the
induced affine group.  The quotient isogeny identifies $H$ with $H'$; for
$\gamma\in H$ write $\gamma'\in H'$ for its image and set
\[
 K_{r,\gamma}:=\mathcal N\cap
 \ker(\mathcal F_E^r-u_{\gamma,E}).
\]
Then $\varphi$ maps $\mathcal S_{r,\gamma}(E)$ into
$\mathcal S_{r,\gamma'}(E')$, and every nonempty fibre of the restricted map
is a $K_{r,\gamma}$-torsor.  Moreover,
\[
 \varphi:\mathcal S_{r,\gamma}(E)\xrightarrow{\sim}
 \mathcal S_{r,\gamma'}(E')
 \quad\Longleftrightarrow\quad
 K_{r,\gamma}=0.
\]
For the induced lower map $g:E/H\to E'/H'$ one has
\[
 g:(E/H)(k_r)\longrightarrow(E'/H')(k_r)\text{ bijective}
 \quad\Longleftrightarrow\quad
 K_{r,\gamma}=0\text{ for every }\gamma\in H.
\]
\end{lemma}

\begin{proof}
Equivariance and $k$-rationality give
\[
 (\mathcal F_{E'}^r-u_{\gamma',E'})\varphi
 =\varphi(\mathcal F_E^r-u_{\gamma,E}).
\]
Hence $\varphi$ maps the source sector into the target sector, and subtracting
two points in one fibre shows that its nonempty fibres are precisely
$K_{r,\gamma}$-torsors.  Taking degrees in the displayed intertwining identity
shows that the two separable endomorphisms have the same degree, so the source
and target sectors have the same finite cardinality.  Thus the sector map is
bijective exactly when $K_{r,\gamma}=0$.

Suppose first that all $K_{r,\gamma}$ vanish.  For
$Q'\in(E'/H')(k_r)$ let $m'_{Q'}(\gamma')$ denote the target sector mass.
Sectorwise bijectivity gives
\[
 m'_{Q'}(\gamma')=
 \sum_{\substack{Q\in(E/H)(k_r)\\ g(Q)=Q'}}m_Q(\gamma).
\]
Summing over $\gamma$ and applying Lemma~\ref{lem:affine-sector-mass} on both
sides yields
\[
 m=m\,\bigl|g^{-1}(Q')\cap(E/H)(k_r)\bigr|.
\]
Thus every rational $Q'$ has exactly one rational preimage, so $g$ is
bijective.

Conversely, assume that $g$ is injective and that $K_{r,\gamma}\ne0$ for some
$\gamma$.  Choose $P\in\mathcal S_{r,\gamma}(E)$ and put $Q=\pi(P)$.  For
every $D\in K_{r,\gamma}$, the point $P+D$ remains in the same sector and has
the same image under $\varphi$.  Thus $\pi(P+D)$ is rational and has the same
lower image as $Q$, so injectivity forces $\pi(P+D)=Q$.  Consequently
$K_{r,\gamma}$ acts freely on that sector fibre and
\[
 |K_{r,\gamma}|\mid m_Q(\gamma)\mid m
\]
by Lemma~\ref{lem:affine-sector-mass}.  This is impossible because
$K_{r,\gamma}$ is a nontrivial $\ell$-group whereas $\ell\nmid m$.
\end{proof}

\begin{theorem}[Exact permutation and exceptionality support]
\label{thm:extension-support}
In the notation of Theorem~\ref{thm:arithmetic-monodromy}, for every
$r\geqslant1$,
\begin{equation}
 r\in\operatorname{Supp}_{\mathrm{perm}}(g)
 \Longleftrightarrow
 r\in\operatorname{Supp}_{\mathrm{exc}}(g)
 \Longleftrightarrow
 \ker(\mathcal F_{\mathcal N}^r-\beta^j)|_{\mathcal N}=0
 \quad(0\leqslant j<m).
\label{eq:extension-support}
\end{equation}
Consequently
\begin{equation}
 \operatorname{Supp}_{\mathrm{perm}}(g)=\operatorname{Supp}_{\mathrm{exc}}(g).
\label{eq:support-equality}
\end{equation}
This common support is periodic modulo $d_g$.  More precisely, if
\[
 \mathcal F_{\mathcal N}^{d_g}=\beta^{j_0}
\]
for some $j_0\in\mathbb Z/m\mathbb Z$, then
\[
 r\in\operatorname{Supp}_{\mathrm{exc}}(g)
 \quad\Longleftrightarrow\quad
 r+d_g\in\operatorname{Supp}_{\mathrm{exc}}(g).
\]
Moreover,
\begin{equation}
 d_g\mathbb Z_{>0}\cap\operatorname{Supp}_{\mathrm{exc}}(g)=\varnothing.
\label{eq:multiples-out}
\end{equation}
If $g/k$ is exceptional, then
\begin{equation}
 \{\,r\geqslant1:\gcd(r,d_g)=1\,\}
 \subseteq
 \operatorname{Supp}_{\mathrm{exc}}(g).
\label{eq:units-in-support}
\end{equation}

If $e=1$, write the actions of $\mathcal F_{\mathcal N}$ and $\beta$ on
$\mathcal N\cong\mathbb F_\ell$ as multiplication by
$\lambda,\rho\in\mathbb F_\ell^\times$, where $\rho$ has order $m$.  Then
\begin{equation}
 d_g=\operatorname{ord}_{\mathbb F_\ell^\times}(\lambda^m)
\label{eq:rank-one-d}
\end{equation}
and, without any additional hypothesis,
\begin{equation}
 \operatorname{Supp}_{\mathrm{perm}}(g)=\operatorname{Supp}_{\mathrm{exc}}(g)
 =\{\,r\geqslant1:d_g\nmid r\,\}.
\label{eq:rank-one-support}
\end{equation}
Equivalently,
\[
 r\in\operatorname{Supp}_{\mathrm{perm}}(g)
 \Longleftrightarrow
 r\in\operatorname{Supp}_{\mathrm{exc}}(g)
 \Longleftrightarrow
 \lambda^{mr}\ne1.
\]
\end{theorem}

\begin{proof}
\noindent\emph{Core equivalence.}
By Lemma~\ref{lem:sector-permutation},
\[
 g\text{ permutes }\mathbf P^1(k_r)
 \quad\Longleftrightarrow\quad
 K_{r,\gamma}=0\quad\text{for every }\gamma\in H.
\]
The linear parts of the elements of $H$ are precisely the powers of $\beta$,
and $\mathcal F_{\mathcal N}^r=\mathcal F_E^r|_{\mathcal N}$.  Hence this is
exactly the kernel condition in \eqref{eq:extension-support}.  Applying
Theorem~\ref{thm:exceptionality} over $k_r$ gives the three-way equivalence and
\eqref{eq:support-equality}.

\medskip\noindent\emph{Periodicity.}  Let
$\mathcal F_{\mathcal N}^{d_g}=\beta^{j_0}$.  Since
$\mathcal F_{\mathcal N}^r\beta^{j_0}\mathcal F_{\mathcal N}^{-r}=\beta^{j_0q^r}$,
\[
 \mathcal F_{\mathcal N}^{r+d_g}
 =
 \mathcal F_{\mathcal N}^r\beta^{j_0}
 =
 \beta^{j_0q^r}\mathcal F_{\mathcal N}^r.
\]
Hence, for $P\in\mathcal N$,
\[
 \mathcal F_{\mathcal N}^{r+d_g}P=\beta^jP
 \quad\Longleftrightarrow\quad
 \mathcal F_{\mathcal N}^rP=\beta^{j-j_0q^r}P.
\]
As $j$ runs modulo $m$, so does $j-j_0q^r$, proving periodicity.

\medskip\noindent\emph{Multiples of $d_g$.}  If $d_g\mid r$, then $\mathcal F_{\mathcal N}^r\in\langle\beta\rangle$, say
$\mathcal F_{\mathcal N}^r=\beta^j$.  Thus
$\ker(\mathcal F_{\mathcal N}^r-\beta^j)=\mathcal N\ne0$, which proves
\eqref{eq:multiples-out}.

\medskip\noindent\emph{Degrees coprime to $d_g$.}  Assume that $g/k$ is exceptional and $\gcd(r,d_g)=1$.  By
Corollary~\ref{cor:constants-coprime},
\[
 M_r=\langle\beta,\mathcal F_{\mathcal N}^r\rangle=M.
\]
Thus over $k_r$ the geometric and arithmetic point-stabilizer actions are the
same pair $(\langle\beta\rangle,M)$ as over $k$.  The exceptional-pair
criterion used in Theorem~\ref{thm:exceptionality} therefore shows that
$g/k_r$ is exceptional, proving \eqref{eq:units-in-support}.

\medskip\noindent\emph{Rank one.}  Finally, assume $e=1$.  Condition \eqref{eq:extension-support} becomes
\[
 \lambda^r\notin\langle\rho\rangle.
\]
Since $\langle\rho\rangle$ is the subgroup of the $m$th roots of unity in
$\mathbb F_\ell^\times$,
\[
 \lambda^r\in\langle\rho\rangle
 \quad\Longleftrightarrow\quad
 (\lambda^m)^r=1.
\]
The order of $\lambda^m$ is therefore the order of the coset
$\lambda\langle\rho\rangle$ in
$\mathbb F_\ell^\times/\langle\rho\rangle$, namely $d_g$.  Equations
\eqref{eq:rank-one-d} and \eqref{eq:rank-one-support} follow.
\end{proof}

\begin{corollary}[Base-change package]
\label{cor:base-change-package}
For every $r\geqslant1$,
\[
 A_g^{(r)}
 \cong
 \mathcal N\rtimes\langle\beta,\mathcal F_{\mathcal N}^r\rangle,
 \qquad
 [A_g^{(r)}:G_g]
 =
 \frac{d_g}{\gcd(d_g,r)}.
\]
Moreover, the following are equivalent:
\[
 g/\mathbb F_{q^r}\text{ is exceptional},
 \qquad
 g\text{ permutes }\mathbf P^1(\mathbb F_{q^r}),
\]
and
\[
 \mathcal F_{\mathcal N}^rP\ne\beta^jP
 \qquad
 \text{for every }0\ne P\in\mathcal N
 \text{ and every }0\leqslant j<m.
\]
\end{corollary}

For a general finite separable cover, exceptionality implies bijectivity on
rational points over the same finite field
\cite[Theorem~1]{GTZ07}.  The converse can fail over a small field:
nonexceptional rational maps which are nevertheless bijective occur in
\cite[Example~5.2]{GTZ07}, while the general converse theorem imposes
explicit largeness hypotheses \cite[Theorem~2]{GTZ07}.  In the
origin-preserving involutive case, the permutation criterion above
specializes to K\"u\c{c}\"uksakall\i{}'s Corollary~2.8
\cite{Kuc14}; for the multiplication-by-$\ell$ rank-two case it is
equivalent to Corollary~2.6 of Bell et al. \cite{Bell22}, where
$\tau_r$ is the trace of $q^r$-Frobenius and
$(q^r+1)^2-\tau_r^2=|E(\mathbb F_{q^{2r}})|$.  For prescribed Frobenius-stable, origin-preserving equivariant-isogeny data,
the companion paper proves the all-extension kernel and support criterion by
weighted Frobenius sectors
\cite[Theorems~1.3--1.4 and Theorem~2.4]{Fan26}.  The argument above is the
self-contained affine extension needed for the cyclic Euclidean data recovered
in the present paper: the descended point stabilizer may have a nonzero
translation part, whereas the finite-kernel obstruction depends on its linear
part.  We include this argument so that the complete finite-extension
arithmetic used in the occurrence and enumeration results below is established
internally.

\subsection{Closed rank-two support formulas}
\label{sec:rank-two-support}

In rank two the uniform kernel criterion is still a matrix test.  The
Frobenius characteristic polynomial turns it into closed support formulas.

In rank two, $\mathcal N=E[\ell]$, so
$\mathcal F_{\mathcal N}=\mathcal F_E|_{E[\ell]}$.  Write
\[
 T^2-a_qT+q,
 \qquad
 a_q=q+1-|E(k)|,
\]
for the characteristic polynomial of $\mathcal F_E$ on $E[\ell]$, and let
$\alpha_1,\alpha_2\in\overline{\mathbb F}_\ell^\times$ be its roots modulo
$\ell$.

\begin{lemma}[Determinant and symmetry-twist invariance]
\label{lem:det-lift}
On $E[\ell]$ one has
\[
 \det(\beta)=1,
 \qquad
 \det(\mathcal F_{\mathcal N})=q\pmod\ell.
\]
If $q\equiv1\pmod m$, then for every $a\in\mathbb Z$ the operator
$\beta^a\mathcal F_{\mathcal N}$ has the same $m$th powers of eigenvalues
as $\mathcal F_{\mathcal N}$.  If $m>2$ and $q\equiv-1\pmod m$, then
\[
 (\beta^a\mathcal F_{\mathcal N})^2=\mathcal F_{\mathcal N}^2.
\]
\end{lemma}

\begin{proof}
For $m=2$, $\beta=-I$.  For $m=3,4,6$, the two eigenvalues of $\beta$ are
$\rho$ and $\rho^{-1}$ with $\rho$ of exact order $m$, so
$\det(\beta)=1$.  The Weil pairing gives
$\det(\mathcal F_{\mathcal N})=q$ on $E[\ell]$.  If
$q\equiv1\pmod m$, then $\mathcal F_{\mathcal N}$ commutes with $\beta$,
so multiplication by $\beta^a$ multiplies each eigenvalue by an $m$th root
of unity.  If $q\equiv-1\pmod m$, then
$\mathcal F_{\mathcal N}\beta^a=\beta^{-a}\mathcal F_{\mathcal N}$, whence
$(\beta^a\mathcal F_{\mathcal N})^2=\mathcal F_{\mathcal N}^2$.
\end{proof}

\subsection{Split rank-two support}

Assume
\[
 q\equiv1\pmod m;
\]
for $m=2$ this condition is automatic.  Put
\begin{equation}
 d_i=\operatorname{ord}_{\overline{\mathbb F}_\ell^\times}(\alpha_i^m),
 \qquad i=1,2.
 \label{eq:di}
\end{equation}

\begin{theorem}[Split spectral support]
\label{thm:split-support}
For every rank-two Euclidean datum in the split case and every $r\geqslant1$,
\begin{align}
 r\in\operatorname{Supp}_{\mathrm{perm}}(g)
 &\Longleftrightarrow
 r\in\operatorname{Supp}_{\mathrm{exc}}(g)
 \Longleftrightarrow
 \det(\mathcal F_E^{mr}-I\mid E[\ell])\ne0
 \label{eq:split-det}\\
 &\Longleftrightarrow
 d_1\nmid r\ \text{ and }\ d_2\nmid r
 \label{eq:split-order}\\
 &\Longleftrightarrow
 \ell\nmid|E(\mathbb F_{q^{mr}})|.
 \label{eq:split-pointcount}
\end{align}
Its natural density is
\begin{equation}
 1-\frac1{d_1}-\frac1{d_2}
 +\frac1{\operatorname{lcm}(d_1,d_2)},
 \label{eq:split-density}
\end{equation}
and the least positive support period is
\begin{equation}
 \begin{cases}
 d_1,&d_1\mid d_2,\\
 d_2,&d_2\mid d_1,\\
 \operatorname{lcm}(d_1,d_2),&\text{otherwise}.
 \end{cases}
 \label{eq:split-period}
\end{equation}
\end{theorem}

\begin{proof}
Assume first $m>2$.  Since $\ell\nmid m$, the operator $\beta$ is semisimple
with two distinct eigenvalues $\rho,\rho^{-1}$ of exact order $m$.  The split
relation makes the two $\beta$-eigenlines Frobenius-stable.  On either line,
\[
 \mathcal F_{\mathcal N}^rv=\beta^jv
\]
for some $j$ if and only if the corresponding Frobenius eigenvalue
$\alpha_i$ satisfies $\alpha_i^{mr}=1$.  Thus
\eqref{eq:extension-support} is equivalent to \eqref{eq:split-det} and
\eqref{eq:split-order}.

For $m=2$, the forbidden operators are $\mathcal F_{\mathcal N}^r-I$ and
$\mathcal F_{\mathcal N}^r+I$; one is singular exactly when some eigenvalue
satisfies $\alpha_i^{2r}=1$, so the same conclusion holds.  Finally,
\[
 \det(I-\mathcal F_E^s\mid E[\ell])
 \equiv|E(\mathbb F_{q^s})|\pmod\ell,
\]
which proves \eqref{eq:split-pointcount}.  The bad set is the union of the
multiples of $d_1$ and $d_2$; inclusion--exclusion gives the density, and the
least-period formula follows from the elementary structure of this union.
\end{proof}

\begin{corollary}[Indecomposable split collapse]
\label{cor:split-collapse}
If the rank-two lower map is indecomposable over $k$, then
\[
 d_1=d_2=d_g=:d
\]
and
\begin{equation}
 \operatorname{Supp}_{\mathrm{perm}}(g)=\operatorname{Supp}_{\mathrm{exc}}(g)
 =\{r\geqslant1:d_g\nmid r\}.
 \label{eq:split-collapse}
\end{equation}
Thus the least period is $d_g$ and the density is $1-1/d_g$.
If $m>2$, one may identify $E[\ell]$ with the one-dimensional
$\mathbb F_{\ell^2}$-space $\mathbb F_{\ell^2}$ so that $\beta$ and
Frobenius act by multiplication by $\rho$ and $\alpha$; then
\begin{equation}
 d_g=\operatorname{ord}_{\mathbb F_{\ell^2}^\times}(\alpha^m).
 \label{eq:split-dg}
\end{equation}
For $m=2$ the same formula holds after identifying an irreducible Frobenius
action with multiplication by $\alpha\in\mathbb F_{\ell^2}^\times$.
\end{corollary}

\begin{proof}
For $m=2$, indecomposability is irreducibility of $\mathcal F_{\mathcal N}$.  The two
Frobenius eigenvalues are Galois conjugate in $\mathbb F_{\ell^2}$ and hence
have the same multiplicative order.  Moreover
$\mathcal F_{\mathcal N}^\nu\in\{\pm I\}$ exactly when $(\alpha^2)^\nu=1$.

Let $m>2$.  If $\beta$ split over $\mathbb F_\ell$, its two eigenlines would
be invariant under both $\beta$ and $\mathcal F_{\mathcal N}$, contradicting
indecomposability.  Thus $\beta$ is irreducible and its centralizer is the
field $\mathbb F_{\ell^2}$.  In this field
\[
 \mathcal F_{\mathcal N}^\nu\in\langle\beta\rangle
 \Longleftrightarrow
 \alpha^\nu\in\langle\rho\rangle
 \Longleftrightarrow
 (\alpha^m)^\nu=1.
\]
This gives $d_g$ and simultaneously shows $d_1=d_2=d_g$.
\end{proof}

\begin{example}[One split isogeny, two quotient orders]
\label{ex:split-C3-C6-support}
Let
\[
 E/\mathbb F_7:y^2=x^3+1.
\]
Here $|E(\mathbb F_7)|=12$, so the Frobenius polynomial is
$T^2+4T+7$.  Modulo $13$ it factors as
\[
 T^2+4T+7=(T-4)(T-5).
\]
Let $\mathcal N\subset E[13]$ be the Frobenius eigenspace with multiplier
$\alpha=5$.  Since $13\equiv1\pmod6$, this line is stable under both the
order-three and order-six CM complements, and it gives the same
Frobenius-stable cyclic $13$-isogeny kernel in the two Euclidean quotient
structures.  The element $5\in\mathbb F_{13}^{\times}$ has order $4$, so
\[
 d_3=\operatorname{ord}(5^3)=4,
 \qquad
 d_6=\operatorname{ord}(5^6)=2.
\]
The split support theorem therefore gives
\[
 \operatorname{Supp}_{\mathrm{perm}}(g_3)=\operatorname{Supp}_{\mathrm{exc}}(g_3)
 =\{r\geqslant1:4\nmid r\},
\]
\[
 \operatorname{Supp}_{\mathrm{perm}}(g_6)=\operatorname{Supp}_{\mathrm{exc}}(g_6)
 =\{r\geqslant1:2\nmid r\}.
\]
Thus the same elliptic isogeny kernel can produce genuinely different
common permutation--exceptionality periods after changing only the Euclidean
complement.  This example makes explicit that, even in the split case, the quotient order $m$
enters arithmetically through the power $\alpha^m$.
\end{example}

\subsection{Nonsplit rank-two support}

Assume now
\[
 m\in\{3,4,6\},
 \qquad
 q\equiv-1\pmod m,
\]
and put
\begin{equation}
 c=-q\pmod\ell,
 \qquad
 h=\operatorname{ord}_{\mathbb F_\ell^\times}(c).
 \label{eq:h}
\end{equation}

\begin{lemma}[The nonsplit Frobenius square]
\label{lem:F2}
On $E[\ell]$ one has
\begin{equation}
 \mathcal F_{\mathcal N}^2=-qI.
 \label{eq:F2}
\end{equation}
Moreover $\langle\beta,\mathcal F_{\mathcal N}\rangle$ acts irreducibly on $E[\ell]$.
\end{lemma}

\begin{proof}
Over $\overline{\mathbb F}_\ell$, the normalizer relation exchanges the two
$\beta$-eigenlines.  Hence $\operatorname{tr}(\mathcal F_{\mathcal N})=0$, while
$\det(\mathcal F_{\mathcal N})=q$ by Lemma~\ref{lem:det-lift}.  Cayley--Hamilton gives
\eqref{eq:F2}.  If $\beta$ is irreducible there is no invariant line; if it
splits, $\mathcal F_{\mathcal N}$ exchanges its two eigenlines.  In either case the full
action is irreducible.
\end{proof}

\begin{theorem}[Nonsplit parity-residue support]
\label{thm:nonsplit-support}
For every $r\geqslant1$,
\begin{equation}
 r\in\operatorname{Supp}_{\mathrm{perm}}(g)
 \Longleftrightarrow
 r\in\operatorname{Supp}_{\mathrm{exc}}(g)
 \Longleftrightarrow
 \begin{cases}
 h\nmid r,&r\text{ odd},\\[1mm]
 h\nmid \dfrac{mr}{2},&r\text{ even}.
 \end{cases}
 \label{eq:nonsplit-support}
\end{equation}
Equivalently,
\begin{equation}
 r\in\operatorname{Supp}_{\mathrm{perm}}(g)
 \Longleftrightarrow
 r\in\operatorname{Supp}_{\mathrm{exc}}(g)
 \Longleftrightarrow
 \begin{cases}
 \ell\nmid q^r+1,&r\text{ odd},\\[1mm]
 \ell\nmid(-q)^{mr/2}-1,&r\text{ even}.
 \end{cases}
 \label{eq:nonsplit-residue}
\end{equation}
In particular,
\begin{equation}
 g\text{ permutes }\mathbf P^1(k)
 \Longleftrightarrow
 g/k\text{ is exceptional}
 \Longleftrightarrow
 \ell\nmid q+1.
 \label{eq:nonsplit-ground}
\end{equation}
\end{theorem}
\begin{proof}
Let $r$ be odd.  For each $j$ put
$A_j=\beta^{-j}\mathcal F_{\mathcal N}^r$.  The operator $A_j$ still exchanges the two
$\beta$-eigenlines, so $\operatorname{tr}(A_j)=0$, while
$\det(A_j)=q^r$.  Hence
\[
 \det(\mathcal F_{\mathcal N}^r-\beta^j)
 =\det(A_j-I)=1+q^r,
\]
independently of $j$.  Thus \eqref{eq:extension-support} is equivalent to
$q^r+1\not\equiv0\pmod\ell$, which for odd $r$ is $h\nmid r$.

If $r=2s$, Lemma~\ref{lem:F2} gives
$\mathcal F_{\mathcal N}^r=c^sI$.  The operator $c^sI-\beta^j$ is singular for some $j$
exactly when $c^s$ is an $m$th root of unity, or equivalently when
$h\mid ms$, or equivalently $h\mid mr/2$.  This proves both formulas and the ground-field
criterion.
\end{proof}

\begin{corollary}[Constant field, least period, and density]
\label{cor:nonsplit-invariants}
Let
\[
 a=\gcd(h,m),
 \qquad
 \varepsilon_h=
 \begin{cases}
 1,&h\text{ odd},\\
 0,&h\text{ even}.
 \end{cases}
\]
Then
\begin{equation}
 d_g=\frac{2h}{\gcd(h,\gcd(m,2))}.
 \label{eq:nonsplit-dg}
\end{equation}
The least positive period of the common support
$\operatorname{Supp}_{\mathrm{perm}}(g)=\operatorname{Supp}_{\mathrm{exc}}(g)$ is
\begin{equation}
 \begin{cases}
 \dfrac{2h}{a},&h\text{ even},\\[2mm]
 h,&h\text{ odd and }a=1,\\
 2h,&h\text{ odd and }a>1,
 \end{cases}
 \label{eq:nonsplit-period}
\end{equation}
and the natural density is
\begin{equation}
 1-\frac{a+\varepsilon_h}{2h}.
 \label{eq:nonsplit-density}
\end{equation}
The least support period can therefore be a proper divisor of $d_g$.
\end{corollary}
\begin{proof}
No odd power of $\mathcal F_{\mathcal N}$ belongs to $\langle\beta\rangle$, because odd
powers invert $\beta$ whereas powers of $\beta$ centralize it.  For an even
power $2s$, one has $\mathcal F_{\mathcal N}^{2s}=c^sI$.  The scalar subgroup of
$\langle\beta\rangle$ is trivial for odd $m$ and is $\{\pm I\}$ for even
$m$.  This gives \eqref{eq:nonsplit-dg}.
The bad odd degrees are the odd multiples of $h$, while the bad even degrees
are precisely the multiples of $2h/a$.  Counting these disjoint progressions
gives the density.  Their stabilizer under translation modulo a common period yields
\eqref{eq:nonsplit-period}; in particular, if $h$ is odd and
$a=1$, the two progressions merge into the set of all multiples of $h$.
\end{proof}

\begin{corollary}[Characteristic-two nonsplit cubic support]
\label{cor:char2-rank2-support}
Let $q=2^a$ with $a$ odd, and let $g$ be a nonsplit rank-two cubic datum of
translation degree $\ell^2$, where $\ell\ne2,3$.  Put
\[
 h=\operatorname{ord}_{\mathbb F_\ell^\times}(-q).
\]
Then
\[
 r\in\operatorname{Supp}_{\mathrm{perm}}(g)
 \Longleftrightarrow
 r\in\operatorname{Supp}_{\mathrm{exc}}(g)
 \Longleftrightarrow
 \begin{cases}
 h\nmid r,&r\text{ odd},\\
 h\nmid3r/2,&r\text{ even},
 \end{cases}
 \qquad
 d_g=2h.
\]
In particular, $g$ permutes $\mathbf P^1(k)$ exactly when it is exceptional
over $k$, equivalently when $\ell\nmid q+1$.
\end{corollary}

\begin{proof}
This is Theorem~\ref{thm:nonsplit-support} and
Corollary~\ref{cor:nonsplit-invariants} with $m=3$.
\end{proof}

\section{Elliptic models and marked automorphisms}
\label{sec:236-244}

\begin{theorem}[The four Euclidean signatures]
\label{thm:intro-signatures}
Assume the hypotheses of Theorem~\ref{thm:intro-forms}, and fix a datum
$\mathcal D=(E,\mathcal N,H)$ supplied by
Theorem~\ref{thm:intro-forms}\textup{(i)}, together with its associated
$E'_{\mathcal D}$, $H'_{\mathcal D}$, and $\varphi_{\mathcal D}$.  For this
theorem write $E':=E'_{\mathcal D}$, $H':=H'_{\mathcal D}$, and
$\varphi:=\varphi_{\mathcal D}$.  After two-sided $k$-M\"obius
equivalence the following hold.
\begin{enumerate}[label=\textup{(\roman*)}]
\item In type $(2,2,2,2)$,
\[
 H=\langle\iota_\varepsilon\rangle,
 \qquad
 \iota_\varepsilon(P)=\varepsilon-P,
 \qquad
 \varepsilon\in E(k),
\]
and the target involution is $P'\mapsto\varphi(\varepsilon)-P'$.

\item In type $(2,3,6)$ one may take
\[
 E:V^2=U^3-1,
 \qquad
 \psi(U,V)=U^3,
\]
and there is a separable $\Phi\in\operatorname{End}_k(E)$ of degree
$\deg f$ such that $\psi\Phi=f\psi$.  The order-six affine shift is removable
over $k$.

\item In type $(2,4,4)$, for some $a\in k^\times$ one may take
\[
 E:y^2=x^3+ax,
 \qquad
 W=(0,0),
 \qquad
 \psi(x,y)=\frac{x^2-a}{4x},
\]
and there is a separable $\Phi\in\operatorname{End}_k(E)$ of degree
$\deg f$ such that, over $k(\zeta_4)$ with $\zeta_4^2=-1$,
\[
 \beta(x,y)=(-x,\zeta_4y),\qquad
 h(P)=\beta(P+W),
\]
where $h$ generates the cyclic quotient, and
\[
 \Phi\beta=\beta\Phi,\qquad
 \Phi(W)=W,
\]
with $\psi\Phi=f\psi$.

\item In type $(3,3,3)$,
$\boldsymbol{j}(E)=\boldsymbol{j}(E')=0$, the curves $E$ and $E'$ are
$k$-isomorphic, and the translation kernel is the kernel of a separable
$k$-endomorphism of degree $\deg f$.  If $q\equiv1\pmod3$, a same-$C_3$
self-endomorphism realization always exists.  The only possible obstruction is
\[
 \deg f=\ell^2,
 \qquad
 q\equiv-1\pmod3,
 \qquad
 \mathfrak s_E(H)\ne0.
\]
Define $\operatorname{sgn}_3(\ell)\in\{\pm1\}$ by
$\operatorname{sgn}_3(\ell)\ell\equiv1\pmod3$.  In this case a same-$C_3$
self-endomorphism realization exists if and only if
\[
 v_3(\operatorname{sgn}_3(\ell)\ell-1)\geqslant v_3(q+1).
\]
The inequality fails for $(q,\ell)=(17,7)$, yielding an obstructed
degree-$49$ form over $\mathbb F_{17}$.
\end{enumerate}
\end{theorem}

Parts \textup{(i)}--\textup{(iii)} are proved here; part \textup{(iv)} is proved
in Section~\ref{sec:333}.  Fried and GMS provide the classical geometric models
\cite[Lemma~2.1]{Fried78} \cite[Theorems~6.5--6.6]{GMS03}; the contribution
here is their prescribed-field refinement, including affine shifts, rank-two
kernels, and the marked normalizers needed in small characteristic.  Write
\[
 n:=\deg f=\ell^e.
\]

\subsection{The \texorpdfstring{$(2,2,2,2)$}{(2,2,2,2)} case}

For $m=2$, Theorem~\ref{thm:cyclotomic} shows that the involution itself is
defined over $k$.  Since its linear part is $[-1]$, it has the form
\[
 \iota_\varepsilon(P)=\varepsilon-P,
 \qquad
 \varepsilon=\iota_\varepsilon(\mathrm O_E)\in E(k).
\]
From $\varphi\iota_\varepsilon=\iota'\varphi$, evaluation at the origin gives
\[
 \iota'(P')=\varphi(\varepsilon)-P'.
\]
This proves Theorem~\ref{thm:intro-signatures}(i).  In prime degree the
unshifted involution quotient is already part of Fried's Schur/CM picture
\cite[Lemma~2.1]{Fried78}; GMS gives the shifted characteristic-zero form
\cite[Theorem~6.5 and the subsequent remark]{GMS03}.  The contribution here is
the prescribed-field recovery and converse, including the arithmetic shift and
the rank-two kernel.

M\"uller's characteristic-zero treatment of the $(2,2,2,2)$ series
obtains a ground-field elliptic isogeny when a branch point is rational over
the ground field and describes the associated constant field through the
kernel \cite[\S5, especially Proposition~5.3]{Muller99}.
The present fixed-field theorem does not assume a rational branch point; its
failure is measured by the shift class in Theorem~\ref{thm:rational-branches}.

In prime degree, the unshifted case $\varepsilon=0$ is closely related to the
finite-field construction of Bisson--Tibouchi: their Theorem~2.1 treats the
standard involution quotient in characteristic different from $2$ and $3$,
and their algorithms in \S\S3--4 additionally assume
$\ell\ne\operatorname{char}k$ \cite[Theorem~2.1 and \S\S3--4]{BT18}.  Our
statement is a converse/descent theorem for the Euclidean cover and allows the
shifted involution as well as the rank-two kernel.

\subsection{The \texorpdfstring{$(2,3,6)$}{(2,3,6)} case}

The prime-degree order-six quotient with $\boldsymbol{j}(E)=0$ belongs to Fried's classical CM
description \cite[Lemma~2.1]{Fried78}.  We follow the algebraic normal form of
\cite[Theorem~6.6(a) and its proof]{GMS03} and establish it over the prescribed
finite field.
The three inertia orders are distinct, hence the three branch points and the
unique unramified point in each branch fibre are $k$-rational.  After
independent $k$-M\"obius changes, put the branch points of orders $6,3,2$ at
$\infty,0,1$ and their unique unramified preimages at the same points.
The divisors give monic $a_0,a_1,D\in k[X]$ and constants $c,c_1\in k^\times$
with
\[
 f(X)=cX\frac{a_0(X)^3}{D(X)^6},
 \qquad
 f(X)-1=c_1(X-1)\frac{a_1(X)^2}{D(X)^6}.
\]
Because $\infty$ is a simple pole, the leading terms agree, so $c_1=c$.
Substituting $X=1$ and $X=0$ gives
\[
 c=\left(\frac{D(1)^2}{a_0(1)}\right)^3,
 \qquad
 c=\left(\frac{D(0)^3}{a_1(0)}\right)^2.
\]
Thus $c$ is both a cube and a square.  Since $k^\times$ is cyclic,
$k^{\times 2}\cap k^{\times 3}=k^{\times 6}$, so $c$ is a sixth power in
$k$.  Absorbing it into $D$ yields
\begin{align}
 f(X)&=X\frac{a_0(X)^3}{D(X)^6},
\label{eq:236-f}\\
 f(X)-1&=(X-1)\frac{a_1(X)^2}{D(X)^6}.
\label{eq:236-fminus1}
\end{align}

Put
\[
 E:V^2=U^3-1,
 \qquad
 \psi(U,V)=U^3.
\]
If $\zeta_3^2+\zeta_3+1=0$, then over $k(\zeta_3)$
\[
 \beta(U,V)=(\zeta_3 U,-V)
\]
has order $6$, $\langle\beta\rangle$ is Frobenius-stable, and $\psi$ is the
quotient by this cyclic group.  Define
\begin{equation}
 U'=U\frac{a_0(U^3)}{D(U^3)^2},
 \qquad
 V'=V\frac{a_1(U^3)}{D(U^3)^3}.
\label{eq:236-Phi}
\end{equation}
Then \eqref{eq:236-f}--\eqref{eq:236-fminus1} give
\[
 U'^3-1=V'^2,
\]
so \eqref{eq:236-Phi} extends to a $k$-morphism $\Phi:E\to E$.  By
construction,
\[
 \psi\Phi=f\psi.
\]
Now $\psi(\mathrm O_E)=\infty$ and $f(\infty)=\infty$.  Since
$\mathrm O_E$ is the unique point above $\infty$ under $\psi$, we have
$\Phi(\mathrm O_E)=\mathrm O_E$.  Therefore
$\Phi\in\operatorname{End}_k(E)$ \cite[Theorem~III.4.8]{Silverman09}.  Taking degrees in
$\psi\Phi=f\psi$ gives $\deg\Phi=\deg f$.  Since $\psi$ and $f$ are
separable, so is $\Phi$.  Finally, direct substitution in
\eqref{eq:236-Phi} gives $\Phi\beta=\beta\Phi$.  This proves
Theorem~\ref{thm:intro-signatures}(ii); Proposition~\ref{prop:shift-class}
explains abstractly why the order-six shift is removable over $k$.

\subsection{The \texorpdfstring{$(2,4,4)$}{(2,4,4)} case}

The argument has three steps.  The branch fibres first produce a quartic
Kummer identity over $k$; that identity lifts the lower map to the source and
target genus-one $C_4$ covers; finally a Gaussian-CM endomorphism with the
same kernel identifies the target equivariantly with the source.  Throughout
this subsection, the subscripts $s$ and $t$ denote source and target objects,
while the subscript $a$ refers to the common one-parameter elliptic model
before specialization.

We follow \cite[Theorem~6.6(c) and its proof]{GMS03}, retaining the source
and target parameters until the final equivariant identification.  Let $P_2$
be the unique branch point with inertia order $2$.  Frobenius preserves inertia
orders, so $P_2$ and its unique unramified preimage are $k$-rational; normalize
both to $\infty$.  Write the two order-four branch values as
$\{\xi,-\xi\}$ and their unique unramified preimages as
$\{\upsilon,-\upsilon\}$, where
\[
 a_t:=\xi^2\in k^\times,
 \qquad
 a_s:=\upsilon^2\in k^\times.
\]
After a source scaling, arrange $f(X)=X+O(1)$ at infinity.  Since $n\equiv1\pmod4$, the fibre over infinity gives
\[
 f(X)=\frac{N_4(X)}{D_4(X)^2},
 \qquad
 \deg N_4=n,
 \qquad
 \deg D_4=\frac{n-1}{2},
\]
with $N_4,D_4$ monic.  The two order-four fibres yield unique monic polynomials
$A_\pm$ of degree $(n-1)/4$ such that
\[
 N_4-\xi D_4^2=(X-\upsilon)A_+^4,
 \qquad
 N_4+\xi D_4^2=(X+\upsilon)A_-^4.
\]
Multiplying gives
\[
 N_4^2-a_tD_4^4=(X^2-a_s)(A_+A_-)^4.
\]
The quotient of the left-hand side by $X^2-a_s$ lies in $k[X]$, and
$A_+A_-$ is its unique monic fourth root; hence $A_+A_-\in k[X]$.  Therefore
\begin{equation}
 f(X)^2-a_t=(X^2-a_s)R(X)^4,
 \qquad
 R=\frac{A_+A_-}{D_4}\in k(X).
\label{eq:244-identity}
\end{equation}

Let $C_s$ (source) and $C_t$ (target) be the smooth projective models of
\[
 C_s:Y^4=X^2-a_s,
 \qquad
 C_t:V^4=U^2-a_t.
\]
They have genus one: the coordinate map to $\mathbf{P}^1$ is a tame cyclic cover of degree
$4$ with the $(2,4,4)$ branching, so Riemann--Hurwitz gives genus one.
Equation \eqref{eq:244-identity} gives a separable degree-$n$ morphism
\begin{equation}
 (X,Y)\longmapsto(f(X),YR(X))
\label{eq:244-lift}
\end{equation}
from $C_s$ to $C_t$, equivariant for $Y\mapsto \zeta_4 Y$.

The two Kummer models $C_s$ and $C_t$ fit into a single one-parameter elliptic
model $E_a$; the specializations $a=a_s$ and $a=a_t$ recover the source and
target curves.
For $a\in k^\times$, the birational change
\[
 u=2(X+Y^2),
 \qquad
 v=4Y(X+Y^2),
\]
with inverse
\[
 X=\frac{u^2+4a}{4u},
 \qquad
 Y=\frac{v}{2u},
\]
identifies $Y^4=X^2-a$ with
\[
 E_a:v^2=u^3-4au.
\]
Under $Y\mapsto \zeta_4 Y$, the coordinates in the elliptic model transform as
\[
 (u,v)\longmapsto\left(\frac{4a}{u},\frac{4\zeta_4 av}{u^2}\right).
\]
On $E_a$, let $W_a=(0,0)$.  Addition by $W_a$ is
\[
 (u,v)\longmapsto\left(-\frac{4a}{u},\frac{4av}{u^2}\right).
\]
Hence the quartic generator becomes
\[
 h_a(P)=\beta_a(P+W_a),
 \qquad
 \beta_a(u,v)=(-u,\zeta_4 v),
\]
and the quotient coordinate is
\begin{equation}
 \psi_a(u,v)=\frac{u^2+4a}{4u}.
\label{eq:244-psi}
\end{equation}
The function \eqref{eq:244-psi} is invariant under $h_a$ and has degree $4$,
so it is the quotient map by $\langle h_a\rangle$.  Over $\bar k$, a cyclic
quartic cover with full ramification at the two marked order-four points and
ramification index $2$ at infinity has Kummer exponents $1,1,2$ after choosing
a generator; hence it is, compatibly with the quotient coordinate, of the form
$Y^4=X^2-a$.  Thus, after base change to $\bar k$, $C_s$ and $C_t$ identify with the source
and target vertical $C_4$ covers in Theorem~\ref{thm:intro-forms}\textup{(i)}.  Under these
geometric identifications, the quotient isogeny supplied by
Theorem~\ref{thm:intro-forms}\textup{(i)} and the explicit lift
\eqref{eq:244-lift} lie over the same lower map.  Since the target vertical map
is a connected Galois $C_4$-cover, any two such lifts differ by a unique target
deck transformation.  Composing the target identification by that deck
transformation if necessary, we may transport \eqref{eq:244-lift} to that
quotient isogeny, whose geometric kernel is $\mathcal{N}$.  Only this
geometric identification is used here: the curves $C_s,C_t$ and the lift
\eqref{eq:244-lift} are already defined over $k$.
Write $E_s:=E_{a_s}$ and $E_t:=E_{a_t}$ for the source and target
specializations, with origins $\mathrm O_s,\mathrm O_t$ and distinguished
points $W_s:=W_{a_s}$ and $W_t:=W_{a_t}$.  Let $\beta_s,h_s$ and
$\beta_t,h_t$ denote the corresponding specializations of $\beta_a,h_a$.
Transporting \eqref{eq:244-lift} therefore gives a separable degree-$n$
$k$-morphism
\[
 \varphi:E_s\longrightarrow E_t
\]
intertwining the two shifted $C_4$ groups.  The fibre of $\psi_a$ above
$\infty$ is $\{\mathrm O_{E_a},W_a\}$.  Since $f(\infty)=\infty$, the map $\varphi$ sends
$\{\mathrm O_s,W_s\}$ into $\{\mathrm O_t,W_t\}$.  If necessary, compose on
the target with $\tau_{W_t}$.  This translation is $k$-rational and commutes
with the shifted $C_4$ action: indeed $\beta_t(W_t)=W_t$ and $2W_t=0$, so
$\tau_{W_t}h_t=h_t\tau_{W_t}$.  We may therefore assume
\[
 \varphi(\mathrm O_s)=\mathrm O_t.
\]
With these origins $\varphi$ is a separable $k$-isogeny
\cite[Theorem~III.4.8]{Silverman09}.  Equivariance then gives
\[
 \varphi(W_s)=\varphi h_s(\mathrm O_s)=h_t\varphi(\mathrm O_s)=W_t.
\]

It remains to identify the target with the source equivariantly.  If $e=2$,
Theorem~\ref{thm:translation-kernel} gives $\mathcal{N}=E_s[\ell]$ and we take
$\Phi=[\ell]$; since $\ell$ is odd, $\Phi(W_s)=W_s$.  If $e=1$, faithful
order-four action on the cyclic kernel gives $\ell\equiv1\pmod4$, and
Frobenius compatibility gives $q\equiv1\pmod4$, so $\beta_s$ is defined over
$k$.  Let its action on $\mathcal{N}$ be multiplication by
$\rho\in\mathbb{F}_\ell^\times$ with $\rho^2=-1$.  Choose integers $a_\ell,b_\ell$ with
\[
 a_\ell^2+b_\ell^2=\ell
\]
and, after replacing $b_\ell$ by $-b_\ell$ if necessary, with
$a_\ell+b_\ell\rho\equiv0\pmod\ell$.  Then
\[
 \Phi=[a_\ell]+[b_\ell]\beta_s
\]
annihilates $\mathcal{N}$.  Since $\widehat\beta_s=\beta_s^{-1}=-\beta_s$, its
dual is $\widehat\Phi=[a_\ell]-[b_\ell]\beta_s$, and therefore
\[
 \widehat\Phi\Phi=[a_\ell^2+b_\ell^2]=[\ell].
\]
Hence $\deg\Phi=a_\ell^2+b_\ell^2=\ell$.  Since $\mathcal{N}$ contains
$\ell$ distinct points, $\Phi$ is separable and
$\ker\Phi=\mathcal{N}$, even in the possible case $\ell=\operatorname{char}k$.
Moreover, since $\ell$ is odd and $a_\ell^2+b_\ell^2=\ell$, exactly one
of $a_\ell,b_\ell$ is odd, so $a_\ell+b_\ell$ is odd.  Since
$\beta_s(W_s)=W_s$,
\[
 \Phi(W_s)=[a_\ell+b_\ell]W_s=W_s.
\]
In both ranks $\Phi\beta_s=\beta_s\Phi$ and $\Phi(W_s)=W_s$, hence
$\Phi h_s=h_s\Phi$.

Thus in both ranks $\Phi$ and $\varphi$ have the same kernel.  The quotient
universal property gives a unique $k$-isomorphism
$\theta:E_t\xrightarrow{\sim}E_s$ with
$\Phi=\theta\varphi$.  From
$\varphi h_s=h_t\varphi$ and $\Phi h_s=h_s\Phi$ we obtain
$\theta h_t=h_s\theta$ by surjectivity of $\varphi$.  Hence the target is
identified with the source as a shifted $C_4$-curve.  Writing
$a=-4a_s$ gives the model and quotient coordinate in
Theorem~\ref{thm:intro-signatures}(iii).

\subsection{Marked normalizers in characteristics 3 and 2}
\label{sec:marked-normalizers}

The common tame theory depends on the normalizer of the marked cyclic linear
part, not on the full automorphism group of the elliptic curve.  In
characteristics $3$ and $2$ the full special-$j$ automorphism groups enlarge,
but the two marked normalizers needed below remain small.

\subsubsection{The marked \texorpdfstring{$C_4$}{C4} normalizer in characteristic 3}

Let
\[
 E_0/\mathbb F_3:y^2=x^3-x,
 \qquad
 U=\operatorname{Aut}(E_{0,\bar k},\mathrm O_{E_0}),
 \qquad
 Z=\{\pm1\},
\]
and let $H=\langle\beta\rangle\cong C_4$ be a chosen order-four subgroup.

\begin{lemma}[Characteristic-$3$ marked-$C_4$ normalizer]
\label{lem:char3-C4-normalizer}
One has
\[
 |U|=12,\qquad U/Z\cong S_3,\qquad N_U(H)=H.
\]
Consequently the marked $C_4$ quotient has one map-level geometric form.  Its
shift group is
\[
 \Delta=E_0[1-\beta]\cong C_2,
\]
and it has exactly two shift classes.
\end{lemma}

\begin{proof}
For the model $y^2=x^3-x$, Kronberg--Soomro--Top describe the twelve
origin-preserving automorphisms explicitly as
$\Phi_{u,r}(x,y)=(u^2x+r,u^3y)$, where $u^4=1$ and $r\in\mathbb F_3$
\cite[Proposition~2.1]{KST17}.  Modulo the central subgroup $Z$, their action
on the three roots of $x^3-x$ is the full affine group
$\operatorname{AGL}_1(\mathbb F_3)\cong S_3$.  Since $H/Z$ is a transposition
subgroup and a transposition in $S_3$ is self-normalizing,
$N_U(H)/Z=H/Z$, hence $N_U(H)=H$.

Thus the exact two-sided equivalence theorem leaves no residual geometric
normalizer outside the complement itself.  Moreover
$\deg(1-\beta)=2$, so $\Delta\cong C_2$.  Frobenius normalizes $H$ and hence
fixes the unique nonzero point of $\Delta$; therefore
$\Delta/(\sigma_q-1)\Delta=\Delta$ has two elements, and no residual
normalizer identifies them.
\end{proof}

\subsubsection{The marked \texorpdfstring{$C_3$}{C3} normalizer in characteristic 2}

Let
\[
 E_0/\mathbb F_2:y^2+y=x^3,
\]
choose a primitive cube root $\omega\in\mathbb F_4$, and put
\[
 \beta(x,y)=(\omega x,y),\qquad
 L_{E_0}=\langle\beta\rangle\cong C_3,
 \qquad
 \Delta=E_0[1-\beta]=\{O,(0,0),(0,1)\}.
\]

\begin{lemma}[Marked $C_3$ normalizer in characteristic $2$]
\label{lem:char2-C3-normalizer}
One has
\[
 N_{\operatorname{Aut}(E_0,\mathrm O_{E_0})}(L_{E_0})
 =C_{\operatorname{Aut}(E_0,\mathrm O_{E_0})}(L_{E_0})
 =\langle\beta,[-1]\rangle\cong C_6,
\]
and hence
\[
 N_{\operatorname{Aut}(E_0)}(L_{E_0})=\Delta\rtimes C_6.
\]
\end{lemma}

\begin{proof}
Kronberg--Soomro--Top, Proposition~3.1, describe the $24$ origin-preserving
automorphisms of $E_0$ as
\[
 \Phi_{u,r,t}(x,y)=
 \bigl(u^2x+r,\ y+u^2r^2x+t\bigr),
\]
where $u\in\mathbb F_4^\times$, $r\in\mathbb F_4$, and
$t^2+t+r^3=0$ \cite{KST17}.  Any automorphism normalizing $L_{E_0}$ preserves
its fixed-point set $\Delta$.  The two finite points of $\Delta$ have
$x$-coordinate $0$, so the displayed formula forces $r=0$.  Then
$t\in\{0,1\}$ and $u\in\mathbb F_4^\times$, giving exactly six
automorphisms.  They are generated by $\beta$ and the elliptic involution
$[-1]:(x,y)\mapsto(x,y+1)$ and centralize $\beta$.

For the full automorphism group of the genus-one curve, write an arbitrary
automorphism as $\tau_T\alpha$ with $\alpha$ origin-preserving.  If it
normalizes $L_{E_0}$, then $\alpha\beta\alpha^{-1}=\beta^{\pm1}$ and
\[
 (\tau_T\alpha)\beta(\tau_T\alpha)^{-1}
 =\tau_{(1-\beta^{\pm1})T}\beta^{\pm1}.
\]
The right-hand side lies in the origin-preserving group $L_{E_0}$ exactly
when $T\in E_0[1-\beta^{\pm1}]$.  Since
$E_0[1-\beta^{-1}]=E_0[1-\beta]=\Delta$, the translation part is exactly
$\Delta$.  This proves
$N_{\operatorname{Aut}(E_0)}(L_{E_0})=\Delta\rtimes C_6$.
\end{proof}

\begin{lemma}[Marked $C_3$ normalizer in characteristic different from $3$]
\label{lem:C3-normalizer-allchar}
Let $E/\bar k$ carry a marked order-three subgroup
$L_E=\langle\beta\rangle$ and suppose $\operatorname{char}k\ne3$.  Then
\[
 N_{\operatorname{Aut}(E,\mathrm O_E)}(L_E)
 =C_{\operatorname{Aut}(E,\mathrm O_E)}(L_E)
 =\langle\beta,[-1]\rangle\cong C_6,
\]
and, with $\Delta=E[1-\beta]$,
\[
 N_{\operatorname{Aut}(E)}(L_E)=\Delta\rtimes C_6.
\]
\end{lemma}

\begin{proof}
If $\operatorname{char}k>3$, an order-three automorphism forces
$\boldsymbol j(E)=0$, and the origin-preserving automorphism group is the
cyclic group $C_6$ \cite[Theorem~III.10.1]{Silverman09}.  In characteristic
$2$ the first assertion is Lemma~\ref{lem:char2-C3-normalizer}.  For the full
affine normalizer, write an automorphism as $\tau_Tu$ with
$u\in C_6$.  Since $u$ centralizes $\beta$,
\[
 (\tau_Tu)\beta(\tau_Tu)^{-1}=\tau_{(1-\beta)T}\beta,
\]
which lies in the linear subgroup $L_E$ exactly when $T\in\Delta$.
\end{proof}

\section{Cubic self-endomorphism realization and the \texorpdfstring{$3$}{3}-adic boundary}
\label{sec:333}

The cubic row alone has a stronger realization problem: over the prescribed
field, the affine shift need not extend to a same-$C_3$ self-endomorphism
model.  Fried and GMS supply the geometric and characteristic-zero models
\cite[Lemma~2.1 and Theorem~2.2]{Fried78} \cite[Theorem~6.6(b)]{GMS03}; the
new issue is arithmetic saturation of the shift.  Write $n=\deg f=\ell^e$.
The construction below follows the algebraic core of \cite[Theorem~6.6(b) and
its proof]{GMS03} and isolates this finite-field step.

\subsection{The cubic isogeny}

Let $\mathcal B=\operatorname{Br}(f)=\{\xi_1,\xi_2,\xi_3\}$ be the
branch locus of $f$.  Each branch fibre has partition $\{1,3,\ldots,3\}$, so let $\upsilon_i$ be the unique
unramified point above $\xi_i$.  Frobenius preserves ramification indices,
so the correspondence $\xi_i\mapsto\upsilon_i$ is Galois equivariant.
Over $\bar k$ there is a unique M\"obius transformation carrying the ordered
branch triple to the corresponding simple-preimage triple.  Galois equivariance
and uniqueness show that this transformation lies in $\operatorname{PGL}_2(k)$.  Replacing
$f$ by $\widetilde f=f\circ\eta$, we may therefore assume that every $\xi_i$ is its
own simple preimage.

Let $B(X,Z)\in k[X,Z]$ be a squarefree homogeneous binary cubic whose zero
divisor on $\mathbf P^1_{\bar k}$ is $\mathcal B$.  Write the normalized
lower map as
\[
 \widetilde f=[F:G],
\]
where $F,G\in k[X,Z]$ are coprime homogeneous forms of degree $n$.

\begin{proposition}[Cubic twist identity]
\label{prop:cubic-twist}
There are $\kappa\in k^\times$ and a homogeneous form
$S(X,Z)\in k[X,Z]$ of degree $n-1$ such that
\[
 B(F,G)=\kappa\,B(X,Z)S(X,Z)^3.
\]
\end{proposition}

\begin{proof}
The zero divisor of $B(F,G)$ is the pullback of the branch divisor
$\mathcal B$.  In each branch fibre, the distinguished point $\xi_i$ occurs
with multiplicity $1$, while every other point occurs with ramification index
$3$.  Since each $\xi_i$ is its own simple preimage,
\[
 \operatorname{div}_0 B(F,G)=\operatorname{div}_0 B(X,Z)+3D
\]
for an effective $k$-rational divisor $D$ of degree $n-1$ on $\mathbf P^1$.
Because $D$ is an effective $k$-rational divisor on $\mathbf P^1$, it is
the zero divisor of a homogeneous form $S(X,Z)\in k[X,Z]$ of degree $n-1$.
The homogeneous forms $B(F,G)$ and $B(X,Z)S(X,Z)^3$ then have the same degree and
the same zero divisor on $\mathbf P^1_{\bar k}$, so they differ by a scalar
$\kappa\in k^\times$.
\end{proof}

Consider the smooth plane cubics
\[
 \mathcal C:\ W^3=B(X,Z),
 \qquad
 \mathcal C':\ W'^3=\kappa^{-1}B(X',Z').
\]
They have genus one because $B$ is squarefree and $\operatorname{char}k\ne3$;
equivalently, tame Riemann--Hurwitz for the degree-three coordinate map gives
genus one.  The homogeneous identity above defines a $k$-morphism
\[
 [X:Z:W]\longmapsto [F(X,Z):G(X,Z):W S(X,Z)]
\]
from $\mathcal C$ to $\mathcal C'$, which we denote by $\Psi$.  If
$\pi,\pi'$ denote the degree-three coordinate maps, then
$\pi'\Psi=\widetilde f\pi$, so $3\deg\Psi=3\deg\widetilde f$ and
$\deg\Psi=n$.  Since $\widetilde f$, $\pi$, and $\pi'$ are separable, so is
$\Psi$.  By Hasse--Weil, both curves have $k$-points.  Choosing
$\mathrm O_{\mathcal C}\in\mathcal C(k)$ and
$\mathrm O_{\mathcal C'}=\Psi(\mathrm O_{\mathcal C})$ as origins turns
$\Psi$ into a separable $k$-isogeny.  With these origins, $\mathcal C$ and $\mathcal C'$ are elliptic curves.
After base change to $\bar k$, a projective change of the $(X,Z)$-coordinates
sends the three zeros of $B$ to $0,1,\infty$; hence both curves are isomorphic
to the standard cyclic cubic model and satisfy
$\boldsymbol{j}(\mathcal C)=\boldsymbol{j}(\mathcal C')=0$.

\begin{lemma}[Identification with the geometric Galois closure]
\label{lem:c3-cubic-identification}
After base change to $\bar k$, the two cubic covers above are the source and
target cyclic $C_3$ quotients of the geometric genus-one Galois closure, and
$\Psi$ identifies with the quotient isogeny of
Theorem~\ref{thm:translation-kernel}.  In particular,
\[
 \ker\Psi(\bar k)=\mathcal{N}.
\]
\end{lemma}

\begin{proof}
The source vertical cover of the geometric Galois closure is cyclic of degree
$3$ and is fully ramified exactly at the three marked simple preimages.  For a
cyclic cubic cover with three full branch points, the three Kummer exponents
are either all $1$ or all $2$: their sum is $0$ modulo $3$, and replacing the
generator of $C_3$ interchanges the two possibilities.  Hence over $\bar k$
the source cover is isomorphic, compatibly with the quotient coordinate, to
$W^3=B(X,Z)$.  The same argument applies to the target; the scalar $\kappa$ is irrelevant
over $\bar k$ because it becomes a cube.  Under these geometric identifications,
the lift $\Psi$ lies over the original lower map $\widetilde f$ and has degree
$n$.  The geometric Galois closure of $\widetilde f$ has degree $3n$ over the
target and degree $3$ over the source, so this cyclic cubic source cover is
precisely the top genus-one cover.  The translation deck subgroup of $\Psi$
is therefore $T_{\mathcal{N}}$, as in Theorem~\ref{thm:translation-kernel}.
\end{proof}

Via Lemma~\ref{lem:c3-cubic-identification}, we henceforth write
\[
 E:=\mathcal C,\qquad E':=\mathcal C',
\]
and identify $\Psi$ with the quotient isogeny $\varphi:E\to E'$ and its
kernel with $\mathcal N$.

The target is in fact $k$-isomorphic to the source.  If $e=2$, this is already
contained in Theorem~\ref{thm:translation-kernel}.  Suppose $e=1$.  The
order-three linear part acts faithfully on the cyclic kernel, so
$\ell\equiv1\pmod3$.  Frobenius acts by a scalar on that kernel, hence commutes
with the order-three action; comparing with Theorem~\ref{thm:cyclotomic}
forces $q\equiv1\pmod3$.  Thus $\beta$ is defined over $k$.  The Eisenstein
order $\mathbb Z[\beta]\cong\mathbb Z[\zeta_3]$ is a PID\@.  If
$\ell\ne\operatorname{char}k$, then $\ell\equiv1\pmod3$ splits in this order.
The two prime-factor endomorphisms have kernels equal to the two
$\beta$-eigenlines in $E[\ell]$, so one of them has kernel $\mathcal{N}$; choose a
generator $\varpi$ of the corresponding prime ideal.  Since
$\widehat\beta=\beta^{-1}=\beta^2$, one has
$\widehat\varpi\varpi=[\operatorname{Norm}(\varpi)]=[\ell]$, hence
$\deg\varpi=\ell$.  As $\ell\ne\operatorname{char}k$, this endomorphism is
separable; it is defined over $k$ because $\beta$ is.  If
$\ell=\operatorname{char}k$, the existence of the
reduced subgroup $\mathcal{N}\cong C_\ell$ forces $E$ to be ordinary
\cite[Corollary~III.6.4(c)]{Silverman09}.  Here $\ell\equiv1\pmod3$, so
$(\ell)=\mathfrak p\bar{\mathfrak p}$ in $\mathbb Z[\beta]$.  For an ordinary elliptic curve in characteristic $\ell$, the group scheme
$E[\ell]$ has a unique reduced subgroup of order $\ell$ and a unique connected
subgroup of order $\ell$.  An isogeny of degree $\ell$ is separable exactly
when its kernel group scheme is reduced.  Hence, among the two degree-$\ell$
prime-factor isogenies above $\ell$, exactly one has reduced kernel and is
separable; its geometric kernel is the unique reduced subgroup
$E[\ell](\bar k)$.  Thus that kernel is precisely $\mathcal{N}$.  Because $\beta$ is defined over
$k$, the separable prime factor is a $k$-endomorphism.  Hence $E'\cong_kE$ in all cases.

\subsection{Two-sided saturation}

Translate a geometric fixed point of $H$ to the origin, so that
$H=\langle\beta\rangle$.  Put
\[
 \delta:=1-\beta,
 \qquad
 \Delta:=E[\delta],
 \qquad |\Delta|=3.
\]
Under this linearization, let $s\in\Delta$ be the cocycle representative
arising from the shift class, so that $\mathfrak s_E(H)=[s]$; equivalently,
$s$ is the translation part of the descended arithmetic Frobenius.  Let
$\mathcal F_0$ denote the origin-preserving linear part of the chosen
arithmetic $q$-Frobenius descent and put
\[
 \widetilde{\mathcal F}_0(P)=\mathcal F_0(P)+s.
\]
Under the fixed geometric identification, arithmetic Frobenius acts on
$\Delta$ through $\mathcal F_0|_\Delta$; hence the original shift class is
represented in $\Delta/(\mathcal F_0-1)\Delta$.  Replacing the chosen lift of arithmetic
Frobenius by $\beta^j$ times that lift gives the three Frobenius lifts
\begin{equation}
 \widetilde{\mathcal F}_j(P)=\mathcal F_j(P)+s,
 \qquad
 \mathcal F_j=\beta^j\mathcal F_0,
 \qquad j=0,1,2.
\label{eq:c3-sections}
\end{equation}
Since $s\in\Delta$, one has $\beta(s)=s$, so the translation term is unchanged.
For each section $j$, let $E_j/k$ denote the elliptic curve obtained from the
corresponding descended genus-one form after choosing a $k$-rational origin;
such an origin exists by the Hasse--Weil bound.  Under a fixed geometric
identification $E_{j,\bar k}\cong E$, its Frobenius endomorphism is represented
by $\mathcal F_j$.  When the ambient object is explicitly torsion or
$T_3(E)$, the same symbol denotes the induced operator; statements on
$T_3(E)$ are equalities in
$\operatorname{End}_{\mathbb Z_3}(T_3(E))$.

\begin{definition}[Correct-kernel endomorphism for a section]
\label{def:correct-kernel}
For a fixed section index $j$, a \emph{correct-kernel endomorphism} is a
separable endomorphism $L\in\operatorname{End}(E_{\bar k})$ of degree $|\mathcal{N}|$ satisfying
\[
 \ker L(\bar k)=\mathcal{N},\qquad L\beta=\beta L,\qquad L\mathcal F_j=\mathcal F_jL.
\]
Thus, $L$ is exactly the geometric kernel isogeny expressed on the chosen
geometric elliptic curve with $\boldsymbol{j}(E)=0$ and is compatible with both the linear $C_3$ action and
the linear Frobenius descent for that section.  For such an $L$, the target
$k$-form has descent datum
\[
 \widetilde{\mathcal F}'_j(P)=\mathcal F_j(P)+L(s),
\]
because $L\widetilde{\mathcal F}_j=\widetilde{\mathcal F}'_jL$.
\end{definition}

The distinction between linear and affine descent is essential here.  A
correct-kernel endomorphism $L$ carries the source affine shift $s$ to the
target shift $L(s)$, so it gives a $k$-isogeny between the corresponding
descended forms, but not a priori a self-endomorphism of one form.  The
saturation problem asks whether the two affine descents can be identified,
within the $C_3$-normalizer freedom, so that this isogeny becomes a
same-$C_3$ self-model.

\begin{definition}[Same-$C_3$ realization]
\label{def:same-c3}
A \emph{same-$C_3$ self-endomorphism realization} of the lower cover means a
$k$-form of the geometric genus-one curve equipped with a $k$-rational origin,
a separable self-endomorphism fixing that origin, and a Frobenius-stable affine
$C_3$ subgroup commuting with the self-endomorphism, such that quotienting by
this same affine subgroup on source and target yields a lower rational map
two-sided $k$-M\"obius equivalent to the original cover.
\end{definition}

\begin{lemma}[Section and normalizer freedom]
\label{lem:c3-normalizer}
Every Frobenius section of the same arithmetic point-stabilizer extension is
represented by one of the three lifts in \eqref{eq:c3-sections}.  Moreover,
\[
 N_{\operatorname{Aut}(E_{\bar k})}(\langle\beta\rangle)
 =\Delta\rtimes C_6,
\]
and every origin-preserving element of this normalizer preserves
$\mathcal N$.
\end{lemma}

\begin{proof}
The extension
\[
 1\to H\to\widetilde H\to\Gamma_k\to1
\]
has fibre $H\widetilde F$ over arithmetic Frobenius; since $\Gamma_k$ is
procyclic, a continuous section is determined by the image of this generator.
This gives the first assertion.  The normalizer identity is
Lemma~\ref{lem:C3-normalizer-allchar}.  If $e=2$ then
$\mathcal N=E[\ell]$, while for $e=1$ the kernel is a $\beta$-stable cyclic
line.  Hence every unit in $\langle-1,\beta\rangle$ preserves $\mathcal N$.
\end{proof}

\begin{theorem}[Two-sided saturation criterion]
\label{thm:c3-saturation}
A $(3,3,3)$ cover is two-sided $k$-M\"obius equivalent to a same-$C_3$
self-endomorphism quotient if and only if there exist
$j\in\{0,1,2\}$, a correct-kernel endomorphism $L$ for that section, and
$P\in E(\bar k)$ such that
\begin{equation}
 (1-\mathcal F_j)P=s,
\label{eq:c3-descent-eq}
\end{equation}
and
\begin{equation}
 (1-L)P\in\Delta.
\label{eq:c3-normalizer-eq}
\end{equation}
\end{theorem}

\begin{proof}
Assume first that \eqref{eq:c3-descent-eq}--\eqref{eq:c3-normalizer-eq}
hold and set $T=(1-L)P\in\Delta$.  Equation
\eqref{eq:c3-descent-eq} says that $P$ is $k$-rational for the source genus-one
form with descent $P\mapsto \mathcal F_jP+s$.  On the target form,
\[
 \mathcal F_j(LP)+Ls=L(\mathcal F_jP+s)=LP,
\]
so $LP$ is $k$-rational.  Translation by $T$ sends $LP$ to $P$, and it is
defined over
$k$ between the two descended forms because
\[
 (1-\mathcal F_j)T=(1-L)(1-\mathcal F_j)P=(1-L)s.
\]
Since $T\in\Delta$, it commutes with $\beta$.  Hence
$\Phi:=\tau_TL$ identifies the two affine $C_3$ quotients and fixes the chosen
origin $P$.  It is separable of degree $|\mathcal{N}|$ because $L$ has reduced kernel $\mathcal{N}$; with $P$ as origin it is an elliptic endomorphism
\cite[Theorem~III.4.8]{Silverman09}.

Conversely, suppose that a two-sided $k$-M\"obius equivalent realization in
the sense of Definition~\ref{def:same-c3} exists.  By Lemma~\ref{lem:normal-closure-lifting}, the lower-cover equivalence lifts to
the geometric normal closures and conjugates their point stabilizers.  Transport
the self-model to the fixed geometric curve $E$ and
linearize its geometric $C_3$ subgroup as above.  By
Lemma~\ref{lem:c3-normalizer}, the transported $k$-descent is one of the three
data \eqref{eq:c3-sections}.  If $P\in E(\bar k)$ is the transported
$k$-rational origin of the self-model, then
\[
 \mathcal F_jP+s=P,
\]
which is \eqref{eq:c3-descent-eq}.

Let $\Phi_0:E\to E$ be the transported geometric self-map.  Relative to the
fixed origin $\mathrm O_E$, write its unique affine decomposition as
\[
 \Phi_0=\tau_TL,
\]
where $L$ is an origin-preserving isogeny
\cite[Theorem~III.4.8]{Silverman09}.  The translations between two points in a
fibre of $\Phi_0$ are exactly $T_{\ker L(\bar k)}$.  The lifted normal-closure
identification carries this regular translation deck group of the self-model to
$T_{\mathcal{N}}$; hence $\ker L(\bar k)=\mathcal{N}$.  Since the self-model has degree $|\mathcal{N}|$ and is separable, $L$ is a
separable degree-$|\mathcal{N}|$ isogeny.  Since the same affine $C_3$ subgroup is used
on source and target, $\Phi_0\beta=\beta\Phi_0$.  Comparing the unique affine
parts gives
\[
 L\beta=\beta L,
 \qquad
 T\in E[1-\beta]=\Delta.
\]
Likewise, $k$-descent of $\Phi_0$ means
$\Phi_0\widetilde{\mathcal F}_j=\widetilde{\mathcal F}_j\Phi_0$; comparison of linear parts yields
$L\mathcal F_j=\mathcal F_jL$.  Thus $L$ is a correct-kernel endomorphism for section $j$.
Finally, the self-endomorphism fixes its chosen origin, so
\[
 P=\Phi_0(P)=LP+T,
 \qquad
 T=(1-L)P\in\Delta.
\]
This is \eqref{eq:c3-normalizer-eq}.
\end{proof}

If the shift class vanishes, take $s=0$ and $P=\mathrm O_E$.  We henceforth assume
$s\ne0$.

\subsection{The split case}

The split argument is local at $3$.  We first choose a Frobenius section for which $1-\mathcal F_j$ has the
smallest possible Eisenstein valuation; we then choose a
correct-kernel endomorphism congruent to the identity modulo $1-\beta$.  The
two choices make the descent and normalizer equations of
Theorem~\ref{thm:c3-saturation} simultaneously solvable.

\begin{lemma}[The Eisenstein $3$-adic module]
\label{lem:eisenstein-3adic}
Let $\mathcal O_3=\mathbb Z_3[\beta]$.  Then $T_3(E)$ is free of rank one as
an $\mathcal O_3$-module, and
\[
 \operatorname{Cent}_{\operatorname{End}_{\mathbb Z_3}(T_3(E))}(\beta)=\mathcal O_3.
\]
Consequently, in the split case every $\mathcal F_j$ and every correct-kernel
endomorphism for section $j$ may be regarded as an element of
$\mathcal O_3$.
\end{lemma}

\begin{proof}
The action of $\beta$ is faithful and satisfies $\beta^2+\beta+1=0$, so
$\mathcal O_3$ acts faithfully on the free rank-two $\mathbb Z_3$-module
$T_3(E)$.  The ring $\mathcal O_3$ is the ramified quadratic discrete
valuation ring over $\mathbb Z_3$.  Hence $T_3(E)$ is a torsion-free rank-one
$\mathcal O_3$-module and is therefore free.  Its $\mathcal O_3$-linear
endomorphism ring is $\mathcal O_3$, which is exactly the centralizer of
$\beta$.  When $q\equiv1\pmod3$, Theorem~\ref{thm:cyclotomic} gives
$\mathcal F_j\beta=\beta \mathcal F_j$; a correct-kernel endomorphism commutes with $\beta$ by
Definition~\ref{def:correct-kernel}.
\end{proof}

\noindent\emph{Completion of the split case.}  Assume $q\equiv1\pmod3$.
We verify the two equations in Theorem~\ref{thm:c3-saturation}.  In rank one, use the separable $k$-endomorphism with kernel $\mathcal N$
constructed in the paragraph preceding the two-sided saturation setup; in rank
two, $\mathcal N=E[\ell]$ and one may take $[\ell]$.  In the split case these
endomorphisms commute with both $\mathcal F_0$ and $\beta$, hence with every
$\mathcal F_j=\beta^j\mathcal F_0$; thus correct-kernel endomorphisms are
available for all three sections.  Work in the ramified quadratic local ring
\[
 \mathcal O_3=\mathbb Z_3[\beta],
 \qquad \delta=1-\beta,
\]
and normalize $v_\delta(\delta)=1$.  Since the shift class is nonzero,
its image in $\Delta/(\mathcal F_0-1)\Delta$ is nonzero.  As
$\Delta\cong\mathbb F_3$, a nontrivial action of $\mathcal F_0$ on $\Delta$
would make $\mathcal F_0-1$ invertible and force this quotient to vanish.
Hence $\mathcal F_0$ acts trivially on $\Delta$, equivalently
$\mathcal F_0\equiv1\pmod\delta$.  Modulo $\delta^2$, multiplication by
$\beta^j\equiv1-j\delta$ shifts the first-order coefficient of
$\mathcal F_0-1$ by $-j$.  Hence one can choose $j\in\{0,1,2\}$ for which
that coefficient is nonzero modulo $3$; then
\[
 v_\delta(1-\mathcal F_j)=1.
\]
By Lemma~\ref{lem:eisenstein-3adic}, the relevant operators lie in
$\mathcal O_3$.  Since $\deg L=|\mathcal{N}|=\ell^e$ and $\ell\ne3$, the degree of a
correct-kernel endomorphism is prime to $3$.  If its
$3$-adic representative were divisible by $\delta$, its determinant on
$T_3(E)$, hence its degree, would be divisible by
$\operatorname{Norm}(\delta)=3$.  Thus it is a unit modulo $\delta$.
Since
$(\mathcal O_3/(\delta))^\times=\{\pm1\}$, multiplying by $[-1]$ if necessary
allows us to choose
\[
 L\equiv1\pmod\delta,
 \qquad 1-L=\delta\mathfrak u.
\]
Write $1-\mathcal F_j=\delta\mathfrak v$ with $\mathfrak v$ a unit.  Multiplication by $\delta$ maps
$E[\delta^2]$ onto $E[\delta]$, so
$1-\mathcal F_j:E[\delta^2]\to\Delta$ is surjective.  Choose
$P\in E[\delta^2]$ with $(1-\mathcal F_j)P=s$.  Then
$(1-L)P\in E[\delta]=\Delta$, and
Theorem~\ref{thm:c3-saturation} applies.

Thus every split $(3,3,3)$ case admits a same-$C_3$ self-endomorphism realization.
In particular, every prime-degree case is split: Frobenius acts by a scalar on
the cyclic kernel, while the faithful order-three action of $\beta$ must commute
with it; Theorem~\ref{thm:cyclotomic} then forces $q\equiv1\pmod3$.

\subsection{The nonsplit rank-two case}

In this subsection, formulas involving $T_3(E)$ use $\mathcal F_j$ for the
induced Tate-module operator fixed in the convention of Section~\ref{sec:333}.
Here the obstruction is entirely $3$-primary.  Correct-kernel endomorphisms collapse to the two scalar possibilities
$\pm[\ell]$; Frobenius has a trace-zero $3$-adic splitting; and the
same-$C_3$ realization problem reduces to comparing the depth of a preimage
under $1-\mathcal F_j$ with the $3$-adic valuation of
$1-\operatorname{sgn}_3(\ell)\ell$.  Fix the cocycle representative
$s\in\Delta$ from the two-sided saturation setup.

Assume
\[
 e=2,
 \qquad
 q\equiv-1\pmod3,
 \qquad
 s\ne0.
\]
Then $\mathcal{N}=E[\ell]$.

\begin{lemma}[Correct-kernel endomorphisms]
\label{lem:c3-correct-kernel}
Every separable degree-$\ell^2$ endomorphism with kernel $E[\ell]$ has the
form $\gamma[\ell]$ for some $\gamma\in\operatorname{Aut}(E,\mathrm O_E)$.
For every section $j$, such an endomorphism is correct-kernel if and only if
$\gamma=\pm1$; equivalently, the correct-kernel endomorphisms are exactly
$[\ell]$ and $[-\ell]$.
\end{lemma}

\begin{proof}
Let $L$ be a separable degree-$\ell^2$ endomorphism with kernel $E[\ell]$.
The maps $L$ and $[\ell]$ have the same kernel and degree, so the universal
property of the quotient by $E[\ell]$ gives $L=\gamma[\ell]$ with
$\gamma\in\operatorname{Aut}(E,\mathrm O_E)$.  Commutation with $\beta$ forces
$\gamma\in C_{\operatorname{Aut}(E,\mathrm O_E)}(\beta)=C_6$ by
Lemma~\ref{lem:C3-normalizer-allchar}.  Write
$\gamma=(-1)^a\beta^b$.  In the nonsplit case Frobenius fixes $[-1]$ and
sends $\beta$ to $\beta^{-1}$, so the condition
$L\mathcal F_j=\mathcal F_jL$ gives $b\equiv-b\pmod3$ and hence $b=0$.
Thus $\gamma=\pm1$.  Conversely $[\ell]$ and $[-\ell]$ have the required
kernel, are separable because $\ell\ne\operatorname{char}k$, and commute with
both $\beta$ and every $\mathcal F_j$.
\end{proof}

Define $\operatorname{sgn}_3(\ell)\in\{\pm1\}$ by
\[
 \operatorname{sgn}_3(\ell)\ell\equiv1\pmod3,
\]
and put $L=[\operatorname{sgn}_3(\ell)\ell]$.

\begin{lemma}[Trace-zero Frobenius and the $3$-adic splitting]
\label{lem:c3-tate-split}
For every section $j$, the descended elliptic curve $E_j/k$ defined above has
Frobenius trace zero.  Under the fixed geometric identification
$E_{j,\bar k}\cong E$, Frobenius acts on the common Tate module $T_3(E)$ as
$\mathcal F_j$, and hence
\[
 \mathcal F_j^2=-q\,\operatorname{id}_{T_3(E)}.
\]
Moreover
\[
 T_3(E)=T_+\oplus T_-
\]
with $\mathcal F_j$-eigenvalues $\vartheta_+\equiv1$ and
$\vartheta_-\equiv-1\pmod3$.  If the shift class is nonzero, then
$\Delta$ is the $+1$ line modulo $3$, and
\begin{equation}
 v_3(1-\vartheta_+)=v_3(q+1),
\label{eq:c3-plus-valuation}
\end{equation}
while $1-\mathcal F_j$ is a $3$-adic unit on $T_-$.
\end{lemma}

\begin{proof}
By Lemma~\ref{lem:eisenstein-3adic}, $T_3(E)$ is free of rank one over
$\mathcal O_3=\mathbb Z_3[\beta]$.  In the nonsplit case
$\mathcal F_j\beta\mathcal F_j^{-1}=\beta^{-1}$, so after choosing an
$\mathcal O_3$-basis there is $a\in\mathcal O_3\otimes\mathbb Q_3$ such that
\[
 \mathcal F_j(x)=a\bar x,
\]
where the bar is the Eisenstein conjugation $\beta\mapsto\beta^{-1}$.
Therefore $\mathcal F_j^2=N(a)$, whereas as a $\mathbb Z_3$-linear operator
$\det(\mathcal F_j)=-N(a)$.  Elliptic Frobenius has determinant $q$, hence
$N(a)=-q$ and
\[
 \mathcal F_j^2=-q\,\operatorname{id}_{T_3(E)}.
\]
The same anti-linear matrix has trace zero, so its characteristic polynomial
is $X^2+q$.  Modulo $3$ this polynomial is $(X-1)(X+1)$ with distinct roots,
and Hensel lifting yields $T_3(E)=T_+\oplus T_-$.  A nonzero class in
$\Delta/(\mathcal F_j-1)\Delta$ forces Frobenius to act trivially on
$\Delta$, so $\Delta$ is the $+1$ line modulo $3$.  Finally
\[
 (1-\vartheta_+)(1+\vartheta_+)=1+q,
\]
and $1+\vartheta_+$ is a $3$-adic unit; this proves
\eqref{eq:c3-plus-valuation}.  Since $\vartheta_-\equiv-1\pmod3$,
$1-\vartheta_-$ is a unit.
\end{proof}

Let $E[3^\infty]_+\subset E[3^\infty]$ denote the $3$-divisible subgroup whose
Tate module is the direct summand $T_+$ of
Lemma~\ref{lem:c3-tate-split}.

\begin{lemma}[Preimage depth and no cancellation]
\label{lem:c3-depth}
On $E[3^\infty]_+$, every solution of
\[
 (1-\mathcal F_j)P=s,
 \qquad 0\ne s\in\Delta,
\]
has exact order $3^{v_3(q+1)+1}$; any two solutions differ by a point of
order at most $3^{v_3(q+1)}$.  Consequently, $(1-L)P$ lies in $\Delta$ for some solution if and only if
\[
 v_3(\operatorname{sgn}_3(\ell)\ell-1)\geqslant v_3(q+1).
\]
If $v_3(\operatorname{sgn}_3(\ell)\ell-1)<v_3(q+1)$, adding an
$\mathcal F_j$-fixed point cannot change this conclusion.
\end{lemma}

\begin{proof}
Identify $E[3^\infty]_+$ with $\mathbb Q_3/\mathbb Z_3$.  By
Lemma~\ref{lem:c3-tate-split}, multiplication by $1-\mathcal F_j$ is multiplication by a unit times
$3^{v_3(q+1)}$.  The preimage of a nonzero point of order $3$ therefore has exact order $3^{v_3(q+1)+1}$, and the kernel has exponent $3^{v_3(q+1)}$.

Now $1-L=[1-\operatorname{sgn}_3(\ell)\ell]$ has $3$-adic valuation
$v_3(\operatorname{sgn}_3(\ell)\ell-1)$, so it lowers the $3$-power order of
such a point by exactly that amount.  The result lies in the order-three
subgroup $\Delta$ precisely when
\[
 v_3(q+1)+1-v_3(\operatorname{sgn}_3(\ell)\ell-1)\leqslant1,
\]
equivalently when
$v_3(\operatorname{sgn}_3(\ell)\ell-1)\geqslant v_3(q+1)$.  If this
inequality fails, adding a kernel element $Q$ of order at most
$3^{v_3(q+1)}$ cannot cancel the leading $3^{v_3(q+1)+1}$ layer: indeed,
\[
 3^{v_3(q+1)}(P+Q)=3^{v_3(q+1)}P\ne0.
\]
Thus every solution gives the same failure.  The $-$ component may be taken
zero because $1-\mathcal F_j$ is invertible there; prime-to-$3$ torsion cannot
interact with this primary obstruction.
\end{proof}

For the other correct-kernel choice $L=-[\operatorname{sgn}_3(\ell)\ell]$, one has
$\displaystyle v_3(1+\operatorname{sgn}_3(\ell)\ell)=0<v_3(q+1)$.  Applying the same argument to
this scalar choice shows that it cannot satisfy
\eqref{eq:c3-normalizer-eq}.  Thus no sign omitted above can remove the
obstruction.

Combining Theorem~\ref{thm:c3-saturation} and
Lemmas~\ref{lem:c3-tate-split}--\ref{lem:c3-depth} gives
\begin{equation}
 {v_3}(\operatorname{sgn}_3(\ell)\ell-1)\geqslant{v_3}(q+1)
\label{eq:c3-criterion-final}
\end{equation}
as the necessary and sufficient condition in the nonsplit rank-two,
nonzero-shift case.  This proves Theorem~\ref{thm:intro-signatures}(iv).

\subsection{A degree-49 obstruction to same-\texorpdfstring{$C_3$}{C3} realization}
\label{subsec:F17}

\begin{example}[A degree-49 nonsplit obstruction]
\label{ex:F17}
Let
\[
 k=\mathbb{F}_{17},
 \qquad
 E:y^2=x^3+1.
\]
Since $17\equiv-1\pmod3$, the cube map is a bijection of $k$, and the same
character-sum argument as above gives $|E(k)|=18$.  Let $\beta$ be an
order-three automorphism over $\mathbb{F}_{17^2}$ and put
\[
 \Delta=E[1-\beta]=\{\mathrm O_E,(0,1),(0,-1)\}\subset E(k).
\]
Choose $0\ne s\in\Delta$.  Lang's theorem
\cite[Theorem~1 and Corollary]{Lang56} gives $P\in E(\bar k)$ with
$\sigma_{17}(P)-P=s$.  Since $s\in\Delta$ is fixed by $\beta$ and
${}^{\sigma_{17}}\beta=\beta^{-1}$, one has
${}^{\sigma_{17}}H=H$.  Thus
$H=\tau_P\langle\beta\rangle\tau_{-P}$ is Frobenius-stable with nonzero shift
class.

Take $\mathcal{N}=E[7]$.  The two eigenvalues of $\beta$ on $E[7]$ are
$2$ and $4$, and Frobenius exchanges the corresponding eigenlines because
$\mathcal F_{\mathcal N}\beta\mathcal F_{\mathcal N}^{-1}=\beta^{-1}$.  Choosing an eigenvector on the $2$-line and
scaling the eigenvector on the $4$-line therefore gives a basis of
$\mathcal{N}\cong\mathbb{F}_7^2$ in which
\[
 \beta=
 \begin{pmatrix}2&0\\0&4\end{pmatrix},
 \qquad
 \mathcal F_{\mathcal N}=
 \begin{pmatrix}0&1\\4&0\end{pmatrix}.
\]
Then
\[
 \mathcal F_{\mathcal N}\beta\mathcal F_{\mathcal N}^{-1}=\beta^{-1},
 \qquad
 \mathcal F_{\mathcal N}^2
 =4\,\operatorname{id}_{E[7]}
 =-17\,\operatorname{id}_{E[7]}
 \quad\text{on }E[7].
\]
Moreover
\[
 \det(\mathcal F_{\mathcal N}-\beta^j)=4\ne0
 \qquad(j=0,1,2).
\]
Theorem~\ref{thm:exceptionality} therefore gives exceptionality.  The two
$\beta$-eigenlines are exchanged by $\mathcal F_{\mathcal N}$, so
$\langle\beta,\mathcal F_{\mathcal N}\rangle$ is irreducible and the degree-$49$ map is
indecomposable.  The converse theorem ensures the corresponding tame
$(3,3,3)$ rational cover over $k$.

Finally
\[
 {v_3}(17+1)=2,
 \qquad
 {v_3}(7-1)=1.
\]
Thus \eqref{eq:c3-criterion-final} fails.  The cover has the uniform
$\boldsymbol{j}(E)=0$ isogeny structure but no two-sided
$\mathbb{F}_{17}$-M\"obius equivalent same-$C_3$ self-endomorphism realization.
Its finite-field arithmetic is nevertheless completely explicit: permutation
and exceptionality have the same support in every extension degree.  Since
\[
 -17\equiv4\pmod7,\qquad
 h=\operatorname{ord}_{\mathbb F_7^\times}(4)=3,
\]
Theorem~\ref{thm:nonsplit-support} gives
\[
\begin{aligned}
 \operatorname{Supp}_{\mathrm{perm}}(g)=\operatorname{Supp}_{\mathrm{exc}}(g)
 &=\{r\geqslant1:r\equiv1,5\pmod6\},\\
 d_g&=6,\qquad \text{natural density}=\frac13.
\end{aligned}
\]
Moreover, Proposition~\ref{thm:C3-r2-count} shows that the prescribed
parameter $(q,\ell,m,e)=(17,7,3,2)$ supports exactly three two-sided
$\mathbb F_{17}$-M\"obius classes.
\end{example}

\subsection{Characteristic two in the cubic signature}
\label{sec:char2}

Let
\[
 E_0/\mathbb F_2:\quad y^2+y=x^3,
\]
and let $\mathcal F_{E_0}$ denote $2$-power Frobenius.  Since
$|E_0(\mathbb F_2)|=3$, the Frobenius polynomial is $T^2+2$, and hence
\begin{equation}
 \mathcal F_{E_0}^{2}=[-2].
\label{eq:char2-frob-square}
\end{equation}
The quotient of the marked normalizer by $C_3$ has order two.  We call the
marked parity represented by the untwisted Frobenius action the
\emph{positive marked parity}; its central $[-1]$-twist is the negative parity.
For $q=2^a$ with $a$ even, put
\[
 s_q=(-2)^{a/2}.
\]
The positive marked parity may then be chosen with scalar $q$-Frobenius
$[s_q]$.  Thus the characteristic-two cubic row requires no new geometric
classification: the marked normalizer is already supplied by
Lemma~\ref{lem:char2-C3-normalizer}, while \eqref{eq:char2-frob-square} makes
the split Frobenius data completely explicit.

\begin{corollary}[Characteristic-two realization boundary]
\label{thm:char2-C3-boundary}
Let $q=2^a$ with $a$ odd and let $\ell\ne2,3$.  In the nonsplit rank-two
cubic case with nonzero shift, the positive marked form admits a same-$C_3$
self-endomorphism realization if and only if
\[
 v_3(\operatorname{sgn}_3(\ell)\ell-1)\geqslant v_3(q+1).
\]
\end{corollary}

\begin{proof}
Characteristic $2$ is allowed in the characteristic-different-from-$3$
analysis above, so this is exactly \eqref{eq:c3-criterion-final}.
\end{proof}

\begin{example}[A characteristic-$2$ degree-$25$ obstruction]
\label{ex:F8-char2-obstruction}
Let $q=8$ and $\ell=5$, and put
$E=E_0\otimes_{\mathbb F_2}\mathbb F_8$.  The datum is nonsplit.  Since
$5\nmid9$, Theorem~\ref{thm:nonsplit-support} at $r=1$ gives exceptionality,
and Lemma~\ref{lem:F2} gives irreducibility.  On $E[5]$ one may choose a
basis in which
\[
 \beta=\begin{pmatrix}0&4\\1&4\end{pmatrix},\qquad
 \mathcal F_{\mathcal N}=\begin{pmatrix}1&2\\3&4\end{pmatrix}.
\]
Then
\[
 \mathcal F_{\mathcal N}\beta\mathcal F_{\mathcal N}^{-1}=\beta^{-1},\qquad
 \mathcal F_{\mathcal N}^2=2I=-8I,
\]
and
\[
 \det(\mathcal F_{\mathcal N}-\beta^j)=4\ne0\qquad(j=0,1,2).
\]
Here
\[
 h=\operatorname{ord}_{\mathbb F_5^\times}(2)=4,
\]
so Corollary~\ref{cor:char2-rank2-support} gives
\[
 d_g=8,
 \qquad
 \operatorname{Supp}_{\mathrm{perm}}(g)=\operatorname{Supp}_{\mathrm{exc}}(g)
 =\{r\geqslant1:8\nmid r\}.
\]
Its natural density is therefore $7/8$.  The exact enumeration in
Corollary~\ref{thm:char2-occurrence-count} gives three two-sided classes.
For the positive nonzero-shift class,
\[
 v_3(\operatorname{sgn}_3(5)5-1)=v_3(-6)=1<2=v_3(9),
\]
so Corollary~\ref{thm:char2-C3-boundary} shows that no same-$C_3$
self-endomorphism realization exists.
\end{example}

\section{Existence for prescribed parameters}
\label{sec:occurrence}

The classification and arithmetic-readout theorems reduce existence to a
finite Frobenius problem.  For a prescribed admissible $(q,\ell,m,e)$, we ask
whether any indecomposable exceptional datum is realized over $k$.  Because
the affine shift does not affect indecomposability or exceptionality, the
answer is linear once the relevant elliptic automorphism type has been fixed.

For an admissible tame row, write
\[
 \operatorname{Occ}_{m,e}(q,\ell)
\]
for the existence of an indecomposable exceptional Euclidean datum of rank
$e$, complement $C_m$, and degree $\ell^e$.

For a fixed map-level marked $k$-form in the relevant signature, call a
rank-$e$ translation kernel \emph{good} if it is reduced, stable under both
the corresponding Frobenius and cyclic linear action, satisfies the
semiregularity condition of Definition~\ref{def:datum}, and the resulting
finite linear action is irreducible and satisfies the exceptionality criterion
of Theorem~\ref{thm:exceptionality}.  These conditions depend only on
$\mathcal N$, $L_E$, and $\mathcal F_{\mathcal N}$, not on the affine shift.
Thus every allowed shift over a good kernel gives an indecomposable
exceptional datum.  Equivalently, $\operatorname{Occ}_{m,e}(q,\ell)$ holds if
and only if some relevant map-level marked $k$-form admits a good rank-$e$
kernel.

Let $\mathcal T(q)$ denote the Waterhouse trace set of elliptic curves over
$\mathbb F_q$.  For the split
quartic and hexagonal rows put
\[
 E_4^0:y^2=x^3+x,
 \qquad
 E_6^0:y^2=x^3+1,
\]
and, for $m=4,6$ with $q\equiv1\pmod m$, define
\begin{equation}
 A_m(q):=q^m+1-|E_m^0(\mathbb F_{q^m})|.
 \label{eq:Am}
\end{equation}
We say that the \emph{split CM bad condition} $\mathfrak B_m(q,\ell)$
holds when
\begin{equation}
 \begin{cases}
 q^m\equiv1\pmod\ell
 \text{ and }
 A_m(q)\equiv2\pmod\ell,&\ell\ne p,\\
 A_m(q)\equiv1\pmod p,&\ell=p.
 \end{cases}
 \label{eq:B-condition}
\end{equation}
Lemma~\ref{lem:CM-twist-killing} proves that $A_m(q)$ is independent of the
marked $k$-form.  The complete occurrence criterion is as follows.

\begin{theorem}[Prescribed-parameter occurrence]
\label{thm:master-occurrence}
Assume that $(m,\operatorname{char}k)$ is an admissible tame row in
\eqref{eq:four-types-intro} and that the degree condition holds.  Then
$\operatorname{Occ}_{m,e}(q,\ell)$ is given exactly by the following list.
\begin{enumerate}[label=\textup{(\arabic*)}]
\item $m=2,e=1$: either $\ell=p$, or $\ell\ne p$ and there are
$t\in\mathcal T(q)$ and
$\lambda\in\mathbb F_\ell^\times\setminus\{\pm1\}$ such that
$\lambda^2-t\lambda+q\equiv0\pmod\ell$.
\item $m=2,e=2$: $\ell\ne p$ and there is $t\in\mathcal T(q)$ such that
$\bigl(\frac{t^2-4q}{\ell}\bigr)=-1$.
\item $m=3,e=1$: $q\equiv1\pmod3$ and $\ell\equiv1\pmod3$.
\item $m=3,e=2$: $\ell\ne p$ and either
$q\equiv1\pmod3$, $\ell\equiv2\pmod3$, or
$q\equiv-1\pmod3$, $\ell\nmid q+1$.
\item $m=4,e=1$: $q\equiv1\pmod4$, $\ell\equiv1\pmod4$, and
$\mathfrak B_4(q,\ell)\text{ fails}$.
\item $m=4,e=2$: $\ell\ne p$ and either
$q\equiv1\pmod4$, $\ell\equiv3\pmod4$,
$\mathfrak B_4(q,\ell)\text{ fails}$, or
$q\equiv-1\pmod4$, $\ell\nmid q+1$.
\item $m=6,e=1$: $q\equiv1\pmod6$, $\ell\equiv1\pmod6$, and
$\mathfrak B_6(q,\ell)\text{ fails}$.
\item $m=6,e=2$: $\ell\ne p$ and either
$q\equiv1\pmod6$, $\ell\equiv5\pmod6$,
$\mathfrak B_6(q,\ell)\text{ fails}$, or
$q\equiv-1\pmod6$, $\ell\nmid q+1$.
\end{enumerate}
\end{theorem}

The propositions below prove the eight cases of
Theorem~\ref{thm:master-occurrence}.

\subsection{The involutive row}

For the involutive row, $\mathcal T(q)$ is the Waterhouse set of
$q$-Frobenius traces of elliptic curves over $\mathbb F_q$.  Waterhouse's classification \cite[Theorem~4.1]{Waterhouse69} describes it
explicitly in every characteristic.  In the generic ordinary range it contains
all integers $t$ satisfying
\[
 t^2\leqslant4q,
 \qquad
 p\nmid t,
\]
together with the supersingular traces permitted by the parity of $a$ and
the congruence class of $p$; we use $\mathcal T(q)$ as a compact notation for
this finite set.

\begin{proposition}[Exact $C_2$ occurrence]
\label{thm:C2-occurrence}
Assume $\ell^e\geqslant5$.
\begin{enumerate}[label=\textup{(\roman*)}]
\item If $e=1$ and $\ell=p$, occurrence always holds.  If $e=1$ and
$\ell\ne p$, occurrence holds if and only if there are
$t\in\mathcal T(q)$ and
\[
 \lambda\in\mathbb F_\ell^\times\setminus\{\pm1\}
\]
such that
\begin{equation}
 \lambda^2-t\lambda+q\equiv0\pmod\ell.
 \label{eq:C2-occ-r1}
\end{equation}
\item If $e=2$, occurrence holds if and only if $\ell\ne p$ and there is
$t\in\mathcal T(q)$ such that
\begin{equation}
 \left(\frac{t^2-4q}{\ell}\right)=-1.
 \label{eq:C2-occ-r2}
\end{equation}
\end{enumerate}
\end{proposition}

\begin{proof}
Here $\beta=-I$.  For $\ell\ne p$, a rank-one Frobenius-stable kernel is an
invariant line in $E[\ell]$, hence corresponds to a root $\lambda$ of the
Frobenius polynomial modulo $\ell$; exceptionality is exactly
$\lambda\ne\pm1$.  In rank two the kernel is all of $E[\ell]$, and
indecomposability is equivalent to irreducibility of the Frobenius polynomial,
which is \eqref{eq:C2-occ-r2}; then exceptionality is automatic.

If $\ell=p$ and $e=1$, choose an ordinary elliptic curve of trace $2$ in the
Waterhouse list.  Here $p=\ell\geqslant5$ by the degree hypothesis, so
$2\not\equiv\pm1\pmod p$; moreover $p\nmid2$, so trace $2$ lies in the
ordinary Waterhouse range.  Its unique reduced $p$-line is Frobenius-stable and the
multiplier is $2\not\equiv\pm1\pmod p$.  Rank two with $\ell=p$ is excluded
by Theorem~\ref{thm:translation-kernel}.
\end{proof}

\begin{remark}
The trace quantifier in Proposition~\ref{thm:C2-occurrence} is genuine.  For
$q=5$ one has
\[
 \mathcal T(5)=\{-4,-3,-2,-1,0,1,2,3,4\}.
\]
A direct quadratic-residue check shows that
$X^2-tX+5$ has no root modulo $79$ for any $t\in\mathcal T(5)$, while
$t^2-20$ is a square or zero modulo $229$ for every
$t\in\mathcal T(5)$.  Hence there is no rank-one occurrence for $\ell=79$
and no rank-two occurrence for $\ell=229$.  Thus the $C_2$ row does not
collapse to a uniform congruence in $(q,\ell)$.
\end{remark}

\subsection{Cyclotomic rank constraints}

For $m=3,4,6$, the characteristic polynomial of $\beta$ on $E[\ell]$ is
\[
 \Phi_3(X)=X^2+X+1,
 \qquad
 \Phi_4(X)=X^2+1,
 \qquad
 \Phi_6(X)=X^2-X+1.
\]

\begin{lemma}[Rank constraints]
\label{lem:rank-constraints}
Let $m\in\{3,4,6\}$ be a surviving tame row, so
$\operatorname{char}k\nmid m$.
\begin{enumerate}[label=\textup{(\roman*)}]
\item Rank one requires
\[
 q\equiv1\pmod m,
 \qquad
 \ell\equiv1\pmod m.
\]
\item In split rank two, indecomposability is possible exactly when
$\Phi_m$ is irreducible modulo $\ell$, namely
\[
 \begin{array}{c|c}
 m&\text{rank-two congruence}\\ \hline
 3&\ell\equiv2\pmod3,\\
 4&\ell\equiv3\pmod4,\\
 6&\ell\equiv5\pmod6.
 \end{array}
\]
\item In nonsplit rank two, $q\equiv-1\pmod m$, irreducibility is automatic.
\end{enumerate}
\end{lemma}

\begin{proof}
In rank one, $\beta$ acts by a scalar of exact order $m$, so $m\mid\ell-1$;
then \eqref{eq:normalizer-linear} forces $q\equiv1\pmod m$.  In split rank
two, $\mathcal F_{\mathcal N}$ commutes with $\beta$: if $\Phi_m$ splits, its two
$\beta$-eigenlines are common invariant lines, while if $\Phi_m$ is
irreducible, $\beta$ alone acts irreducibly.  In the nonsplit case Frobenius
exchanges the two eigenlines whenever they exist, giving irreducibility.
\end{proof}

\subsection{The cubic row}

The order-three case is distinguished by
\[
 C_3=\langle\beta\rangle
 \subsetneq
 \operatorname{Aut}(E_{\bar k},\mathrm O_E)\cong C_6.
\]
The residual unit $[-1]$ gives a quadratic twist not absorbed by the
geometric complement.

\begin{lemma}[Sign-form escape in the $C_3$ row]
\label{lem:C3-twist}
Assume $\operatorname{char}k\ne3$ and $q\equiv1\pmod3$.  The two marked
$C_3$-form parities differ by the central unit $[-1]$.  Hence, on every
Frobenius-stable reduced cyclic kernel of odd order, their Frobenius
multipliers differ by a sign.  Consequently, for every nonzero multiplier
$\lambda$, at least one of $\lambda$ and $-\lambda$ lies outside the group of
cube roots of unity.
\end{lemma}

\begin{proof}
By Lemma~\ref{lem:C3-normalizer-allchar}, the marked normalizer is $C_6$ and
its quotient by $L_E=C_3$ is generated by the class of $[-1]$.  Thus the
nontrivial marked form multiplies Frobenius by $[-1]$ without changing the
order-three subgroup.  On every reduced cyclic kernel of odd order, $[-1]$ acts as $-1$; this
includes the unique reduced $p$-line in the ordinary characteristic-dividing
rank-one case.  Since $-1\notin\mu_3$, the sets $\mu_3$ and $-\mu_3$ are
disjoint.
\end{proof}

\begin{proposition}[Exact $C_3$ occurrence]
\label{thm:C3-occurrence}
Assume $\operatorname{char}k\ne3$ and the degree condition for $m=3$.
\begin{enumerate}[label=\textup{(\roman*)}]
\item Rank one occurs exactly when
\[
 q\equiv1\pmod3,
 \qquad
 \ell\equiv1\pmod3.
\]
This includes $\ell=p$.
\item Rank two occurs exactly when $\ell\ne p$ and either
\[
 q\equiv1\pmod3,
 \qquad
 \ell\equiv2\pmod3,
\]
or
\[
 q\equiv-1\pmod3,
 \qquad
 \ell\nmid q+1.
\]
\end{enumerate}
\end{proposition}

\begin{proof}
The rank congruences are Lemma~\ref{lem:rank-constraints}.  In split rank
one, if a Frobenius multiplier lies in $\mu_3$, Lemma~\ref{lem:C3-twist}
makes the multiplier on the other marked parity lie outside $\mu_3$; hence
one of the two parities is exceptional.  The same argument works on the unique reduced $p$-line when $\ell=p$.  In split rank two,
$\beta$ is irreducible; if one quadratic parity is bad, twisting replaces an
eigenvalue $\alpha$ by $-\alpha$ and changes $\alpha^3=1$ to
$(-\alpha)^3=-1$, so the other parity is good.  The nonsplit assertion is
Theorem~\ref{thm:nonsplit-support} at $r=1$.
\end{proof}

\begin{corollary}[Characteristic-two cubic occurrence]
\label{cor:char2-occurrence}
Let $q=2^a$ and let $\ell\ne2,3$.
\begin{enumerate}[label=\textup{(\roman*)}]
\item Rank one occurs exactly when $a$ is even and
$\ell\equiv1\pmod3$.
\item If $a$ is even, split rank two occurs exactly when
$\ell\equiv2\pmod3$.
\item If $a$ is odd, nonsplit rank two occurs exactly when
$\ell\nmid q+1$.
\end{enumerate}
\end{corollary}

\begin{proof}
Since $2\equiv-1\pmod3$, one has $q\equiv(-1)^a\pmod3$.  The result is
therefore the characteristic-two specialization of
Proposition~\ref{thm:C3-occurrence}.
\end{proof}

\subsection{The quartic and hexagonal rows}

For $m=4,6$, the standard marked CM curves $E_4^0,E_6^0$ and
$A_m(q)$ from \eqref{eq:Am} encode the split arithmetic.
\begin{lemma}[Marked twist killing and $m$-step trace invariance]
\label{lem:CM-twist-killing}
Let $m\in\{4,6\}$ be a surviving tame row with $q\equiv1\pmod m$.  For the
marked subgroup $L_E=C_m$, one has
\[
 N_{\operatorname{Aut}(E,\mathrm O_E)}(L_E)=L_E.
\]
Every marked $k$-form becomes isomorphic to the standard marked form over
$\mathbb F_{q^m}$.  Consequently the trace of $q^m$-Frobenius is independent
of the marked form and equals $A_m(q)$.
\end{lemma}

\begin{proof}
For $m=4$ the normalizer assertion is standard in characteristic greater than
$3$ and is Lemma~\ref{lem:char3-C4-normalizer} in characteristic $3$; the row
does not occur in characteristic $2$.  For $m=6$ tameness forces
$\operatorname{char}k>3$, where the origin-preserving group is $C_6$.
Thus a marked descent cocycle takes values in $C_m$.  Its value on
$q^m$-Frobenius is
\[
 c\,{}^{\sigma_q}\!c\cdots{}^{\sigma_q^{m-1}}\!c
 =c^{1+q+\cdots+q^{m-1}}=1,
\]
because $q\equiv1\pmod m$.  Hence the marked twist is trivial over
$\mathbb F_{q^m}$, and the Frobenius trace there is the common value $A_m(q)$.
\end{proof}

Equivalently, if $a_m(q)=q+1-|E_m^0(k)|$, then
\begin{align*}
 A_4(q)&=a_4(q)^4-4q\,a_4(q)^2+2q^2,\\
 A_6(q)&=a_6(q)^6-6q\,a_6(q)^4+9q^2a_6(q)^2-2q^3.
\end{align*}

The split CM bad condition $\mathfrak B_m(q,\ell)$ is the condition
\eqref{eq:B-condition} fixed at the beginning of this section.

For $\ell\ne p$, this says precisely that the full $\ell$-torsion is rational
over $\mathbb F_{q^m}$; for $\ell=p$ it says that the reduced $p$-line is
rational over that field.

\begin{proposition}[Exact split $C_4$ and $C_6$ occurrence]
\label{thm:C46-split-occurrence}
Let $m\in\{4,6\}$ be a surviving tame row, assume
$q\equiv1\pmod m$, and assume the corresponding degree condition.
\begin{enumerate}[label=\textup{(\roman*)}]
\item Rank one occurs if and only if
\[
 \ell\equiv1\pmod m
 \qquad\text{and}\qquad
 \mathfrak B_m(q,\ell)\text{ fails}.
\]
\item Rank two occurs if and only if $\ell\ne p$, the corresponding
irreducibility congruence in Lemma~\ref{lem:rank-constraints} holds, and
\[
 \mathfrak B_m(q,\ell)\text{ fails}.
\]
\end{enumerate}
\end{proposition}

\begin{proof}
In split rank one, the two $\beta$-eigenlines are Frobenius-stable and a line
is bad exactly when its multiplier has $m$th power $1$.  Thus a good line
exists exactly when $q^m$-Frobenius does not fix the full reduced
$\ell$-torsion.  In rank two the centralizer of the irreducible $\beta$ is the field
$\mathbb F_{\ell^2}$, so $\mathcal F_E^m$ acts by multiplication by a field
element.  If it has an eigenvalue $1$, that element is $1$, and therefore
$\mathcal F_E^m=I$.  Thus there is no Jordan ambiguity, and the same condition
results.  The characteristic-dividing rank-one case is identical on the unique
reduced line.  Unlike $C_3$, the marked normalizer is already the complement
for $C_4$ and $C_6$, so no additional marked form can escape the bad coset.
\end{proof}

\begin{proposition}[Exact nonsplit occurrence]
\label{thm:nonsplit-occurrence}
Let $m\in\{3,4,6\}$ and suppose $q\equiv-1\pmod m$.  Rank-one occurrence is
impossible.  Rank-two occurrence holds if and only if
\begin{equation}
 \ell\ne p
 \qquad\text{and}\qquad
 \ell\nmid q+1.
 \label{eq:nonsplit-occurrence}
\end{equation}
\end{proposition}

\begin{proof}
Rank one is excluded by Lemma~\ref{lem:rank-constraints}.  Rank two is
irreducible by Lemma~\ref{lem:F2}, and its ground-field exceptionality is
exactly \eqref{eq:nonsplit-ground}.
\end{proof}

\begin{example}[A sharp occurrence contrast between $C_3$ and $C_6$]
\label{ex:F7-C3-C6-occurrence}
Take $q=\ell=7$.  Rank-one $C_3$ occurrence in degree $7$ holds because
$q\equiv\ell\equiv1\pmod3$; the residual quadratic twist in
\(\operatorname{Aut}(E_{\bar k},\mathrm O_E)/C_3\cong C_2\) moves a bad multiplier
outside the cubic root-of-unity sector.

The corresponding rank-one $C_6$ occurrence fails.  For the canonical
$j=0$ curve
\[
 E_6^0:y^2=x^3+1
\]
one has
\[
 A_6(7)=-286\equiv1\pmod7.
\]
Hence the characteristic-$7$ obstruction $\mathfrak B_6(7,7)$ holds, so the
unique reduced $7$-torsion line has a Frobenius multiplier whose sixth power
is $1$.  There is no residual quadratic twist outside $C_6$ to remove this obstruction.
Thus the availability of a residual twist distinguishes the split $C_3$ and
$C_6$ occurrence problems even on the same $j=0$ CM geometry.
\end{example}

\section{Enumeration of fixed-field forms}
\label{sec:CM-counts}

The occurrence theorem decides whether a prescribed parameter is realized.
The classification becomes arithmetically complete only after counting the
resulting fixed-field forms.  For every admissible parameter put
\[
 \mathrm{Cl}_{m,e}(q,\ell)
 =|\{\text{two-sided $k$-M\"obius classes with parameter $(q,\ell,m,e)$}\}|,
\]
with the convention that this number is zero when
$\operatorname{Occ}_{m,e}(q,\ell)$ fails.

We now specialize the marked-orbit principle of
Proposition~\ref{prop:marked-orbit} to this arithmetic problem.  For each
map-level marked stratum $c$, take $K_c$ to be its set of good rank-$e$
translation kernels.  This family is $R_c$-stable, and
\[
 (K_c\times S_c)/R_c
\]
is exactly the set of two-sided classes contributed by the stratum.  Thus
Section~\ref{sec:occurrence} asks whether the good-kernel locus is nonempty,
whereas the present section counts residual orbits of the same good kernels
together with their affine shift classes.

Here a \emph{map-level marked $k$-form} means a twisted-conjugacy stratum in
Proposition~\ref{prop:marked-orbit}; it should not be confused with an elliptic
$k$-form, or equivalently an elliptic $k$-isomorphism class of the underlying
curve.  In the involutive row, the Waterhouse trace/order data and Schoof-corrected
class-number sums first count the latter.  Generically, the two quadratic
elliptic twists determine one map-level form, which is the source of the
one-half conversion below.

The next theorem summarizes this arithmetic closure before the individual
formulas are introduced.

\begin{theorem}[Arithmetic closure of the fixed-field classification]
\label{thm:master-enumeration}
Let $(q,\ell,m,e)$ be an admissible tame parameter from
\eqref{eq:four-types-intro}.  Then $\mathrm{Cl}_{m,e}(q,\ell)$ is finite and
is determined exactly by the fixed-field datum as follows.
\begin{enumerate}[label=\textup{(\roman*)}]
\item For the cubic signature $m=3$, the count is a finite marked-orbit
calculation involving the marked $k$-form, the good translation kernel, and the
affine shift class.  When $\ell\ne p$, the rank-one and split-rank-two formulas are read from
the multiplicities at $1$ and $-1$ in the three-step Frobenius polynomial;
when $\ell=p$ in rank one, the same role is played by the multiplier on the
unique reduced $p$-line.  Every occurring nonsplit rank-two $C_3$ parameter
supports exactly three classes.
\item For $m=4$ and $m=6$, rank two is rigid after occurrence:
\[
 \mathrm{Cl}_{4,2}(q,\ell)=2,
 \qquad
 \mathrm{Cl}_{6,2}(q,\ell)=1.
\]
The rank-one counts are the corresponding finite-step Frobenius
multiplicity formulas, with the reduced-line analogue when $\ell=p$.
\item For the involutive signature $m=2$, the answer is an exact finite
residue-degree-weighted class-number sum built from Waterhouse trace/order data
and Schoof's fixed-order multiplicities.  In characteristic $3$, the unique
supersingular special-$j$ stratum is replaced by an explicit $S_3$-Burnside
correction.  These are literal class counts, not asymptotic or stack-weighted
masses; the exact formulas are stated later in this section.
\end{enumerate}
The characteristic-two cubic row is already included in \textup{(i)}: the
same marked $C_3$ orbit calculation applies, with no separate classification
or counting mechanism.
\end{theorem}

Thus enumeration is the final arithmetic layer of the classification rather
than an auxiliary counting problem.  The exact equivalence theorem reduces
the count to finite marked orbit spaces, while Frobenius relative to the
marked cyclic action determines which kernels survive the exceptionality
test.  The row formulas below establish
Theorem~\ref{thm:master-enumeration}.

\subsection{CM enumeration from finite-step Frobenius}

For the CM signatures $m=3,4,6$, the fixed-field equivalence theorem separates
the exact count into three finite pieces: the map-level $k$-form, the good
translation kernel, and the affine shift class.

For a fixed geometric marked pair put
\[
 U^\sharp=N_{\operatorname{Aut}(E_{\bar k},\mathrm O_E)}(L_E),
 \qquad L_E=\langle\beta\rangle.
\]
The marked-normalizer calculations of Section~\ref{sec:236-244} give
\[
\begin{array}{c|c|c|c}
 m&\text{allowed characteristic}&U^\sharp&U^\sharp/L_E\\ \hline
 3&p\ne3&C_6&C_2,\\
 4&p\ne2&C_4&1,\\
 6&p>3&C_6&1.
\end{array}
\]
The form and shift multiplicities are therefore instances of
Proposition~\ref{prop:marked-orbit}.

\begin{corollary}[CM forms and shift orbits]
\label{lem:form-quotient}
\label{lem:kernel-shifts}
For the surviving tame CM rows, the marked-orbit principle has
\[
\begin{array}{c|c|c|c}
 m&W&\text{map strata}&\text{shift-orbit numbers}\\ \hline
 3&C_2&+,-&2,1,\\
 4&1&\text{one}&2,\\
 6&1&\text{one}&1.
\end{array}
\]
In the $C_3$ row the three combined form/shift strata have branch-point
Frobenius orbit types
\[
 1+1+1,\qquad 3,\qquad 1+2.
\]
In split rank one with $\ell\ne p$, the two $\beta$-eigenlines remain distinct
two-sided classes whenever both are exceptional; for $\ell=p$ the reduced
line is unique, and in rank two the kernel is $E[\ell]$.
\end{corollary}

\begin{proof}
For $m=3$, $U^\sharp=C_6$ and $W=C_2$; the two twisted-conjugacy classes are
the quadratic parities.  On the positive parity
$S_c\cong C_3$, and the residual involution identifies the two nonzero
classes, giving two shift orbits; on the negative parity $F_c-1$ is invertible
on $\Delta\cong C_3$, giving one.  For $m=4,6$, $U^\sharp=L_E$, hence
$W=1$, while $|\Delta|=2,1$.  The branch-orbit types follow from
Proposition~\ref{prop:cyclic-signature}.  Finally, the marked normalizer is abelian in the $C_3$ case and equals the
complement in the $C_4$ and $C_6$ cases, so it preserves each $\beta$-eigenspace,
giving the stated kernel rigidity.
\end{proof}

\subsubsection{The weighted \texorpdfstring{$C_3$}{C3} formulas}

Take the positive marked parity
\[
 E_3^+:\quad
 \begin{cases}
 y^2=x^3+1,&\operatorname{char}k\ne2,3,\\
 y^2+y=x^3,&\operatorname{char}k=2,
 \end{cases}
\]
and define
\begin{equation}
 A_3(q)=q^3+1-|E_3^+(\mathbb F_{q^3})|.
 \label{eq:A3}
\end{equation}
The value is intrinsic to the positive marked parity: all marked forms within
that parity become isomorphic over $\mathbb F_{q^3}$ by the same cocycle-norm
argument as
in Lemma~\ref{lem:CM-twist-killing}.  If
$a_3(q)=q+1-|E_3^+(\mathbb F_q)|$, then
\[
 A_3(q)=a_3(q)^3-3q\,a_3(q).
\]
For $\ell\ne p$ put
\[
 P_3(T)=T^2-A_3(q)T+q^3\in\mathbb F_\ell[T],
\]
and let
\[
 \nu_+=\operatorname{mult}_{T=1}P_3(T),
 \qquad
 \nu_-=\operatorname{mult}_{T=-1}P_3(T).
\]

\begin{proposition}[Exact rank-one $C_3$ count]
\label{thm:C3-r1-count}
Assume $\operatorname{char}k\ne3$ and
\[
 q\equiv1\pmod3,
 \qquad
 \ell\equiv1\pmod3.
\]
If $\ell\ne p$, then
\begin{equation}
 \mathrm{Cl}_{3,1}(q,\ell)=6-2\nu_+-\nu_-.
 \label{eq:C3-r1-count}
\end{equation}
If $\ell=p$ (necessarily $p>3$), put
\[
 \varepsilon_+=\mathbf 1_{A_3(q)\equiv1\pmod p},
 \qquad
 \varepsilon_-=\mathbf 1_{A_3(q)\equiv-1\pmod p};
\]
then
\begin{equation}
 \mathrm{Cl}_{3,1}(q,p)=3-2\varepsilon_+-\varepsilon_-.
 \label{eq:C3-r1-p-count}
\end{equation}
For $\ell\ne p$, every value $2,3,4,5,6$ occurs for suitable parameters.
\end{proposition}

\begin{proof}
On the positive parity, a kernel line is bad exactly when its Frobenius
multiplier has cube $1$.  Since Frobenius commutes with the split semisimple $\beta$, it is diagonal
on the two distinct $\beta$-eigenlines; hence there is no Jordan ambiguity.
The number of bad positive-parity lines is therefore exactly $\nu_+$, and
every good line contributes two shift classes.  On the negative marked parity,
Frobenius is $-\mathcal F_E$.  A line is bad exactly when the original
multiplier has cube $-1$, giving $\nu_-$ bad lines, and each good line
contributes one class.  Thus
\[
 2(2-\nu_+)+(2-\nu_-)=6-2\nu_+-\nu_-.
\]
When $\ell=p$, only the unique reduced line remains, and the multiplier of
$q^3$-Frobenius on it is $A_3(q)$ modulo $p$; the same $2+1$ weighting gives
\eqref{eq:C3-r1-p-count}.

For completeness, the final attainment assertion is witnessed by the
following prime-field parameters; the displayed pair is
$(\nu_+,\nu_-)$:
\[
\begin{array}{c|c|c|c|c}
 \mathrm{Cl}_{3,1}&q&\ell&A_3(q)&(\nu_+,\nu_-)\\ \hline
 2&67&7&-880&(2,0)\\
 3&13&7&-70&(1,1)\\
 4&31&13&308&(1,0)\\
 5&7&13&20&(0,1)\\
 6&7&19&20&(0,0).
\end{array}
\]
In each row $q\equiv\ell\equiv1\pmod3$ and $\ell\ne p$, so the hypotheses
are satisfied.
\end{proof}

For split rank two define
\[
 \epsilon_+
 =\mathbf 1_{\{q^3\equiv1\ (\ell),\ A_3(q)\equiv2\ (\ell)\}},
 \qquad
 \epsilon_-
 =\mathbf 1_{\{q^3\equiv1\ (\ell),\ A_3(q)\equiv-2\ (\ell)\}}.
\]

\begin{proposition}[Exact rank-two $C_3$ count]
\label{thm:C3-r2-count}
Assume $\operatorname{char}k\ne3$, $\ell\ne p$, and the rank-two degree
condition.
\begin{enumerate}[label=\textup{(\roman*)}]
\item If $q\equiv1\pmod3$ and $\ell\equiv2\pmod3$, then
\begin{equation}
 \mathrm{Cl}_{3,2}(q,\ell)=3-2\epsilon_+-\epsilon_-.
 \label{eq:C3-r2-split-count}
\end{equation}
Thus the count is $1$, $2$, or $3$.
\item If $q\equiv-1\pmod3$ and $\ell\nmid q+1$, then
\begin{equation}
 \mathrm{Cl}_{3,2}(q,\ell)=3.
 \label{eq:C3-r2-nonsplit-count}
\end{equation}
If $\ell\mid q+1$, the count is zero.
\end{enumerate}
\end{proposition}

\begin{proof}
In split rank two the kernel is $E[\ell]$ and $\beta$ is irreducible.  On the
positive parity, failure of exceptionality is $\mathcal F_E^3=I$, equivalent in the
field centralizer to
$q^3\equiv1$ and $A_3(q)\equiv2$ modulo $\ell$; this removes the two positive
shift classes.  On the negative parity, failure is $\mathcal F_E^3=-I$, which removes
the single negative class.  This proves \eqref{eq:C3-r2-split-count}.  In the nonsplit row, exceptionality is $\ell\nmid q+1$ and is unchanged by
the central sign form, so both parities occur and contribute $2+1=3$
classes.
\end{proof}

\begin{corollary}[Characteristic-two cubic class counts]
\label{thm:char2-occurrence-count}
Let $q=2^a$ and let $\ell\ne2,3$.
If $a$ is even, write $s_q=(-2)^{a/2}$ as in
Section~\ref{sec:char2} and put
\[
 \epsilon_+=\mathbf1_{s_q^3\equiv1\pmod\ell},\qquad
 \epsilon_-=\mathbf1_{s_q^3\equiv-1\pmod\ell}.
\]
Then every occurring rank-one cubic parameter satisfies
\[
 \mathrm{Cl}_{3,1}(q,\ell)=6-4\epsilon_+-2\epsilon_-,
\]
and every occurring split rank-two cubic parameter satisfies
\[
 \mathrm{Cl}_{3,2}(q,\ell)=3-2\epsilon_+-\epsilon_-.
\]
If $a$ is odd, every occurring nonsplit rank-two cubic parameter supports
exactly three classes.
\end{corollary}

\begin{proof}
For even $a$, equation~\eqref{eq:char2-frob-square} makes the positive
$q^3$-Frobenius scalar $[s_q^3]$.  Hence the trace $A_3(q)$ defined above is
$2s_q^3$, and the polynomial $P_3(T)$ is $(T-s_q^3)^2$.  Thus $\nu_+=2\epsilon_+$ and $\nu_-=2\epsilon_-$, and the two formulas follow from
Propositions~\ref{thm:C3-r1-count} and~\ref{thm:C3-r2-count}.  The odd-$a$
nonsplit count is Proposition~\ref{thm:C3-r2-count}\textup{(ii)}.
\end{proof}

\subsubsection{The \texorpdfstring{$C_4$ and $C_6$}{C4 and C6} formulas}

For $m=4,6$, use $A_m(q)$ from \eqref{eq:Am}.  In the split case and for
$\ell\ne p$ put
\[
 P_m(T)=T^2-A_m(q)T+q^m\in\mathbb F_\ell[T],
 \qquad
 \nu_m=\operatorname{mult}_{T=1}P_m(T).
\]

\begin{proposition}[Exact rank-one $C_4$ and $C_6$ counts]
\label{thm:C46-r1-count}
Assume $m\in\{4,6\}$ is a surviving tame row,
$q\equiv1\pmod m$, and $\ell\equiv1\pmod m$.  If $\ell\ne p$, then
\begin{equation}
 \mathrm{Cl}_{4,1}(q,\ell)=4-2\nu_4,
 \qquad
 \mathrm{Cl}_{6,1}(q,\ell)=2-\nu_6.
 \label{eq:C46-r1-count}
\end{equation}
Under occurrence, the possible values are $2,4$ for $C_4$ and $1,2$ for
$C_6$.

If $\ell=p$, put
\[
 \varepsilon_m=\mathbf 1_{A_m(q)\equiv1\pmod p}.
\]
Then
\begin{equation}
 \mathrm{Cl}_{4,1}(q,p)=2(1-\varepsilon_4),
 \qquad
 \mathrm{Cl}_{6,1}(q,p)=1-\varepsilon_6.
 \label{eq:C46-r1-p-count}
\end{equation}
Thus occurrence yields exactly two $C_4$ classes and one $C_6$ class in the
characteristic-dividing case.
\end{proposition}

\begin{proof}
For $\ell\ne p$, the two $\beta$-eigenlines are distinct map classes.
Since $\mathcal F_E$ commutes with the split semisimple $\beta$, it is diagonal
on these two lines.  A line is bad exactly when $\mathcal F_E^m$ acts as $1$
on that line, so the number of bad lines is exactly $\nu_m$.  There is one
map-level CM form.  Multiplying the number of good lines by the shift-orbit
number---two for $C_4$ and one for $C_6$---gives
\eqref{eq:C46-r1-count}.  For $\ell=p$ there is only the reduced line, whose
$q^m$-Frobenius multiplier is $A_m(q)$ modulo $p$.
\end{proof}

\begin{proposition}[Rank-two rigidity for $C_4$ and $C_6$]
\label{thm:C46-r2-count}
Whenever the rank-two occurrence criterion holds,
\begin{equation}
 \mathrm{Cl}_{4,2}(q,\ell)=2,
 \qquad
 \mathrm{Cl}_{6,2}(q,\ell)=1.
 \label{eq:C46-r2-count}
\end{equation}
This is valid in both the split and nonsplit rows.
\end{proposition}

\begin{proof}
The rank-two kernel is uniquely $E[\ell]$.  For $C_4$ and $C_6$ there is one
map-level CM form, so after occurrence the only remaining multiplicity is the
shift orbit: two for $C_4$ and one for $C_6$.
\end{proof}

\subsection{Non-CM enumeration from Waterhouse strata}
\label{sec:C2-count}

The $C_2$ row is qualitatively different because $j(E)$ varies.  Its exact
count is nevertheless finite and explicit once the Waterhouse strata are
refined by endomorphism order.  This is the natural level at which repeated
Frobenius eigenvalues can be distinguished from scalar Frobenius.

For a Waterhouse-admissible trace $t$, put
\[
 \mathcal E_q(t)
 =\{[E]_k:q+1-|E(k)|=t\},
 \qquad
 C_q(t)=|\mathcal E_q(t)|,
 \qquad
 D_t=t^2-4q.
\]
Waterhouse determines the possible endomorphism orders in each isogeny
class \cite[Theorem~4.2]{Waterhouse69}.  For the number of $k$-isomorphism
classes with a fixed commutative endomorphism order, we use Schoof's correction
of Waterhouse's Theorem~4.5 \cite[Theorem~4.5]{Schoof87}.

For $(n,p)=1$ and $a\in\mathbb Z/n\mathbb Z$, define
\begin{equation}
 C_q(t;n,a)
 =|\{[E]\in\mathcal E_q(t):\mathcal F_E|_{E[n]}=[a]\}|.
 \label{eq:scalar-count-def}
\end{equation}

\begin{proposition}[Order-refined scalar-level count]
\label{prop:scalar-level}
Suppose $K_t=\mathbb Q(\pi_t)$ is imaginary quadratic, where $\pi_t$ is a
root of $X^2-tX+q$, and let $\mathcal R_q(t)$ be the finite set of orders that
occur as $k$-endomorphism rings $\operatorname{End}_k(E)$ in the trace-$t$
isogeny class.  For $R\in\mathcal R_q(t)$, let $f_p(R)$ be the residue class degree of $p$
in $R$ in the sense of Schoof's Theorem~4.5; explicitly,
\[
 f_p(R)=
 \begin{cases}
 2,&\left(\dfrac{\operatorname{disc}(R)}p\right)=-1,\\
 1,&\text{otherwise}.
 \end{cases}
\]
Then
\begin{equation}
 C_q(t;n,a)
 =\sum_{\substack{R\in\mathcal R_q(t)\\ \pi_t-a\in nR}}f_p(R)h(R).
 \label{eq:scalar-class-number}
\end{equation}
For ordinary classes, $\mathcal R_q(t)$ consists of the orders containing
$\mathbb Z[\pi_t]$ that are permitted by Waterhouse; in the commutative
supersingular cases one includes the corresponding maximality condition at
$p$.  If $t=\pm2\sqrt q$ so that Frobenius is the integer $[t/2]$, then
\[
 C_q(t;n,a)=
 \begin{cases}
 C_q(t),&a\equiv t/2\pmod n,\\
 0,&\text{otherwise}.
 \end{cases}
\]
\end{proposition}

\begin{proof}
For an elliptic curve with $k$-endomorphism ring $R=\operatorname{End}_k(E)$
and $(n,p)=1$, Frobenius is
scalar $[a]$ on $E[n]$ exactly when $\mathcal F_E-[a]$ kills $E[n]$.  Since $[n]$ is
separable, this is equivalent to
\[
 \mathcal F_E-[a]=[n]\psi
\]
for some $\psi\in R$, or $\pi_t-a\in nR$.  The condition therefore depends
only on $R$.  By Schoof's correction of Waterhouse's Theorem~4.5, the
number of $k$-isomorphism classes with $k$-endomorphism ring $R$ is
$f_p(R)h(R)$.  Summing over the orders satisfying
$\pi_t-a\in nR$ gives \eqref{eq:scalar-class-number}.
\end{proof}

\begin{remark}[Repeated roots]
If $X^2-tX+q$ has a repeated root modulo $\ell$, trace data alone do not
distinguish a scalar Frobenius from a non-scalar Jordan block.  The latter has
one invariant $\ell$-line, whereas the former has $\ell+1$.  The divisibility
condition $\pi_t-a\in\ell R$ in Proposition~\ref{prop:scalar-level} is exactly
the refinement needed to record this jump.
\end{remark}

\subsubsection{The shift mass}

For $[E]\in\mathcal E_q(t)$ put
\[
 s(E)=\bigl|E[2]/(\mathcal F_E-1)E[2]\bigr|=|E[2](k)|
\]
and define
\[
 \Sigma_2(t)=\sum_{[E]\in\mathcal E_q(t)}s(E).
\]

\begin{lemma}[Exact shift mass]
\label{lem:shift-mass}
One has
\begin{equation}
 \Sigma_2(t)=
 \begin{cases}
 C_q(t),&t\text{ odd},\\[1mm]
 2C_q(t)+2C_q(t;2,1),&t\text{ even}.
 \end{cases}
 \label{eq:shift-mass}
\end{equation}
\end{lemma}
\begin{proof}
Modulo $2$ the Frobenius polynomial is $X^2+X+1$ when $t$ is odd, so there is
no nonzero rational $2$-torsion and $s(E)=1$.  When $t$ is even, the
polynomial is $(X-1)^2$; the fixed space has dimension one or two, giving
$s(E)=2$ or $4$.  The latter case is precisely the scalar condition
$\mathcal F_E|_{E[2]}=1$, counted by $C_q(t;2,1)$.
\end{proof}
\subsubsection{Exact global masses}
For rank two define
\begin{equation}
 \mathfrak R_2(q,\ell)
 =\sum_{t\in\mathcal T(q)}
 \mathbf 1_{\left(\frac{D_t}{\ell}\right)=-1}\,\Sigma_2(t).
 \label{eq:raw-r2}
\end{equation}
This is the exact raw mass over $k$-isomorphism classes: the kernel is
uniquely $E[\ell]$, and the Legendre condition is exactly irreducibility of
Frobenius on it.
If $e=1$ and $\ell=p$, ordinary curves have one reduced $p$-line and
supersingular curves have none.  Hence
\begin{equation}
 \mathfrak R_1(q,p)
 =\sum_{\substack{t\in\mathcal T(q)\\ p\nmid t\\ t\not\equiv\pm1\pmod p}}
 \Sigma_2(t).
 \label{eq:raw-r1-p}
\end{equation}
Now assume $\ell\ne p$.  Put
\[
 P_t(X)=X^2-tX+q\in\mathbb F_\ell[X]
\]
and
\[
 G_\ell(t)
 =\{\lambda\in\mathbb F_\ell^\times\setminus\{\pm1\}:P_t(\lambda)=0\},
 \qquad
 g_\ell(t)=|G_\ell(t)|.
\]
For $\lambda\in G_\ell(t)$, let $\widetilde\lambda$ be the residue class
modulo $2\ell$ with
\[
 \widetilde\lambda\equiv1\pmod2,
 \qquad
 \widetilde\lambda\equiv\lambda\pmod\ell,
\]
and define
\[
 \Sigma_{\ell,\lambda}(t)=
 \begin{cases}
 C_q(t;\ell,\lambda),&t\text{ odd},\\[1mm]
 2C_q(t;\ell,\lambda)+2C_q(t;2\ell,\widetilde\lambda),&t\text{ even}.
 \end{cases}
\]

\begin{proposition}[Exact raw rank-one mass]
\label{prop:raw-r1}
For $\ell\ne p$,
\begin{equation}
 \mathfrak R_1(q,\ell)
 =\sum_{t\in\mathcal T(q)}
 \left(
 g_\ell(t)\Sigma_2(t)
 +\ell\sum_{\lambda\in G_\ell(t)}\Sigma_{\ell,\lambda}(t)
 \right).
 \label{eq:raw-r1}
\end{equation}
\end{proposition}

\begin{proof}
For a fixed curve of trace $t$, each good root $\lambda\in G_\ell(t)$ gives
at least one invariant line.  Distinct roots give one line each, while a
repeated good root gives one line in the Jordan case and $\ell+1$ lines in
the scalar case.  Thus
\[
 |K_{\mathrm{good}}(E)|
 =g_\ell(t)
 +\ell\sum_{\lambda\in G_\ell(t)}
 \mathbf 1_{\mathcal F_E|E[\ell]=[\lambda]}.
\]
Multiplying by $s(E)$ and summing gives the first term of
\eqref{eq:raw-r1}.  When $t$ is even, the extra factor from full rational
$2$-torsion occurs simultaneously with scalar $\ell$-Frobenius exactly when
Frobenius is scalar $[\widetilde\lambda]$ on $E[2\ell]$, by the Chinese
remainder theorem.  Proposition~\ref{prop:scalar-level} gives the stated
formula.
\end{proof}

\subsubsection{Generic half and special-\texorpdfstring{$j$}{j} corrections}

The special-$j$ correction data used below are elementary consequences of
the actions of the enlarged automorphism groups on $E[2]$.

\begin{lemma}[Special-$j$ form and shift data]
\label{lem:C2-special-j-data}
In the $C_2$ row the following hold.
\begin{enumerate}[label=\textup{(\roman*)}]
\item If $j=1728$, then $U=C_4$, $L_E=C_2$, and there are two map-level
form classes.  They may be represented by $c_0,c_1$ so that their shift
modules have cardinalities $4$ and $2$, respectively.  Each map-level form
contains
\[
 m_4(q)=
 \begin{cases}
 2,&q\equiv1\pmod4,\\
 1,&q\equiv3\pmod4
 \end{cases}
\]
elliptic $k$-isomorphism classes in the raw mass.  The residual group
$U/L_E\cong C_2$ acts on the shift module of $c_0$ with exactly two fixed
classes and acts trivially on the two-element shift module of $c_1$.
\item If $j=0$ and $q\equiv1\pmod6$, then $U=C_6$, $L_E=C_2$, and there
are three map-level form classes $c_0,c_1,c_2$, with shift-module
cardinalities $4,1,1$.  Each contains two elliptic $k$-isomorphism classes in
the raw mass.  The residual group $U/L_E\cong C_3$ fixes exactly one
shift class for $c_0$ and, necessarily, the unique shift class for $c_1,c_2$.
\item If $j=0$ and $q\equiv5\pmod6$, there is one map-level form class.  Its
shift module has cardinality $2$, it contains two elliptic
$k$-isomorphism classes in the raw mass, and the stabilizer of the chosen
form modulo $L_E$ is trivial.
\end{enumerate}
\end{lemma}

\begin{proof}
For $j=1728$, an order-four unit $\iota$ acts on the two-dimensional
$\mathbb F_2$-space $E[2]$ as the unique nontrivial element of order two:
it fixes one nonzero point and interchanges the other two.  Frobenius
normalizes $\langle\iota\rangle$ and, on $E[2]$, commutes with this action
because $\iota^{-1}$ and $\iota$ have the same reduction modulo $2$.
Hence the Frobenius action on $E[2]$ is either $1$ or $\iota$.  Multiplying
the descent cocycle by $\iota$ interchanges these two cases, giving shift
quotients of sizes $4$ and $2$.  The map-level form quotient is
$U/(L_E U^{q-1})$, of order two.  The ordinary twist set
$U/U^{q-1}$ has order $4$ when $q\equiv1\pmod4$ and order $2$ when
$q\equiv3\pmod4$, which gives the stated raw multiplicity.  Finally,
$\iota$ fixes two elements of $E[2]$ in the first shift module, while on
the quotient by $(\iota-1)E[2]$ its induced action is trivial.

For $j=0$, the action of $U=C_6$ on $E[2]$ factors through
$U/\{\pm1\}\cong C_3$; an order-three unit $\zeta$ cycles the three nonzero
points of $E[2]$.  If $q\equiv1\pmod6$, Frobenius commutes with $\zeta$.
The three map-level forms correspond on $E[2]$ to Frobenius actions
$1,\zeta,\zeta^2$, so their shift quotients have sizes $4,1,1$.
The twist set $U/U^{q-1}$ has six elements, hence two raw elliptic forms over
each of the three map-level classes.  Burnside uses the residual $C_3$;
$\zeta$ fixes only the zero shift in the four-element module, and the other
two modules are singletons.

If $q\equiv5\pmod6$, Frobenius conjugates $\zeta$ to $\zeta^{-1}$ and acts
on $E[2]$ as a transposition, so the shift quotient has size $2$.
Now $U/(L_E U^{q-1})$ is trivial, whereas $U/U^{q-1}$ has two elements.
Moreover, an element $u\in U$ stabilizes the chosen map-level cocycle modulo
$L_E$ only if $u^{q-1}\in L_E$; for $q-1\equiv4\pmod6$ this forces
$u\in L_E$.  Thus no residual order-three action remains.
\end{proof}

The passage from raw elliptic-form mass to two-sided $C_2$ map classes follows
a uniform rule.  If $j(E)\ne0,1728$, then
\[
 \operatorname{Aut}(E_{\bar k},\mathrm O_E)=\{\pm1\}=L_E.
\]
The two quadratic twists of a fixed geometric curve define the same two-sided
map form and have the same good-kernel count, while $-1$ acts trivially on
both projective kernel lines and $E[2]$.  Hence a generic stratum contributes
exactly one half of its raw elliptic-form mass.  A stratum with enlarged
automorphism group must instead be removed with its full raw multiplicity and
recomputed by the exact Burnside orbit count on its good-kernel and shift
data.

For a special map-level form $c$, let $K_{c,e}$ be its set of good rank-$e$
kernels, and put
\[
 \kappa_{c,e}=|K_{c,e}|.
\]
For rank two, $\kappa_{c,2}\in\{0,1\}$.  For rank one with $\ell\ne p$,
\[
 \kappa_{c,1}=g_\ell(t_c)+\ell\sigma_c,
\]
where $\sigma_c=1$ exactly when a repeated good root is scalar on $E[\ell]$;
for $\ell=p$, $\kappa_{c,1}$ is one precisely when the reduced line exists and
has multiplier different from $\pm1$.

At $j=1728$, write $U=C_4=\langle\iota\rangle$ and $L_E=C_2$.
There are two map-level forms $c_0,c_1$, labelled so that their shift-module
sizes are $4$ and $2$.  Let
\[
 \kappa_{4,r,e}^{\iota}=|K_{c_r,e}^{\iota}|
 \qquad(r=0,1).
\]
Burnside's lemma gives
\begin{equation}
 J_{4,e}
 =2\kappa_{4,0,e}+\kappa_{4,1,e}
 +\kappa_{4,0,e}^{\iota}+\kappa_{4,1,e}^{\iota}.
 \label{eq:J4}
\end{equation}
The raw contribution of these forms is
\begin{equation}
 B_{4,e}=m_4(q)
 \bigl(4\kappa_{4,0,e}+2\kappa_{4,1,e}\bigr),
 \qquad
 m_4(q)=
 \begin{cases}
 2,&q\equiv1\pmod4,\\
 1,&q\equiv3\pmod4.
 \end{cases}
 \label{eq:B4}
\end{equation}
For rank two,
$J_{4,2}=3\kappa_{4,0,2}+2\kappa_{4,1,2}$.

At $j=0$, write $U=C_6$ and $L_E=C_2$.  If $q\equiv1\pmod6$, there are
three map-level forms $c_0,c_1,c_2$, with shift-module sizes $4,1,1$.
Let $\zeta$ be an order-three unit and put
$\kappa_{6,r,e}^{\zeta}=|K_{c_r,e}^{\zeta}|$.  Then
\begin{equation}
 J_{6,e}
 =\frac{4\kappa_{6,0,e}+2\kappa_{6,0,e}^{\zeta}}3
 +\sum_{r=1}^2
 \frac{\kappa_{6,r,e}+2\kappa_{6,r,e}^{\zeta}}3,
 \label{eq:J6split}
\end{equation}
and the raw special mass is
\begin{equation}
 B_{6,e}
 =2\bigl(4\kappa_{6,0,e}+\kappa_{6,1,e}+\kappa_{6,2,e}\bigr).
 \label{eq:B6split}
\end{equation}
For rank two,
$J_{6,2}=2\kappa_{6,0,2}+\kappa_{6,1,2}+\kappa_{6,2,2}$.
If $q\equiv5\pmod6$, there is one map-level $j=0$ form; write
$\kappa_{6,e}$ for its good-kernel count.  Then
\begin{equation}
 J_{6,e}=2\kappa_{6,e},
 \qquad
 B_{6,e}=4\kappa_{6,e};
 \label{eq:J6nonsplit}
\end{equation}
this stratum therefore contributes no correction to the generic one-half
rule.

\begin{theorem}[Global exact $C_2$ class count]
\label{thm:global-C2-count}
Let
\[
 \mathfrak R_e(q,\ell)=
 \begin{cases}
 \mathfrak R_1(q,\ell),&e=1,\\
 \mathfrak R_2(q,\ell),&e=2,
 \end{cases}
\]
with $\mathfrak R_1$ and $\mathfrak R_2$ given above.  Assume
$\operatorname{char}k>3$.  Then, for every odd prime $\ell$ and
$e\in\{1,2\}$ with $\ell^e\geqslant5$,
\begin{equation}
 \mathrm{Cl}_{2,e}(q,\ell)
 =\frac12\bigl(\mathfrak R_e(q,\ell)-B_{4,e}-B_{6,e}\bigr)
 +J_{4,e}+J_{6,e}.
 \label{eq:global-C2-count}
\end{equation}
Every term on the right is an explicit finite residue-degree-weighted
class-number sum or one of the finite special-CM Burnside terms above.  In
particular,
\eqref{eq:global-C2-count} is an exact prescribed-$(q,\ell)$ class count,
not an asymptotic formula and not a stack-weighted count.
\end{theorem}

\begin{proof}
The raw masses sum, over all $k$-isomorphism classes of elliptic curves, the
product of the number of good kernels and the number of shift classes.  The
scalar-level
refinement, the exact shift mass, and Proposition~\ref{prop:raw-r1} show that
these are exact Schoof-corrected class-number sums, including the
scalar-versus-Jordan correction in rank one.

By the raw-to-map conversion above, every generic stratum contributes one
half of its raw mass.  The terms $B_{4,e}$ and $B_{6,e}$ remove the two special
geometric $j$-strata with their full raw multiplicities, while the Burnside
terms $J_{4,e}$ and $J_{6,e}$ recompute their exact two-sided orbit counts.
This proves \eqref{eq:global-C2-count}.
\end{proof}

\subsubsection{The characteristic-3 supersingular correction}
\label{subsec:C2-char3}

Assume in this subsection that $q=3^a$.  The generic Waterhouse masses above
remain unchanged.  Thus only the supersingular geometric $j=0$ stratum
requires a new marked-automorphism correction.  Put
\[
 E_0/\mathbb F_3:y^2=x^3-x,\qquad
 U=\operatorname{Aut}(E_{0,\bar k},\mathrm O_{E_0}),\qquad
 Z=\{\pm1\}.
\]

\begin{lemma}[Characteristic-$3$ special-$j$ strata]
\label{lem:char3-C2-special}
The quotient $U/Z$ is $S_3$, and arithmetic $q$-Frobenius acts trivially on
this quotient.  Hence the two-sided $C_2$ map-level forms with geometric
$j=0$ are the three conjugacy types
\[
 c_0=1,\qquad c_3=\text{a $3$-cycle},\qquad
 c_2=\text{a transposition}.
\]
Let $V=E_0[2]\cong\mathbb F_2^2$ and
$S_c=V/(c-1)V$.  The shift modules and residual shift-orbit counts are
\[
\begin{array}{c|c|c|c}
 c&|S_c|&C_{S_3}(c)&|S_c/C_{S_3}(c)|\\ \hline
 c_0&4&S_3&2,\\
 c_3&1&C_3&1,\\
 c_2&2&C_2&2.
\end{array}
\]
Moreover, the numbers of elliptic $k$-isomorphism classes lying over
$(c_0,c_3,c_2)$ are respectively
\[
 (1,2,1)\quad\text{if $a$ is odd},
 \qquad
 (2,2,2)\quad\text{if $a$ is even}.
\]
\end{lemma}

\begin{proof}
Kronberg--Soomro--Top, Proposition~2.1, gives
$U=\{\Phi_{u,r}:u^4=1,\ r\in\mathbb F_3\}$ and lists all Frobenius-conjugacy
classes over $\mathbb F_{3^a}$ \cite{KST17}.  Modulo $Z$, the automorphism
$\Phi_{u,r}$ acts on the three roots of $x^3-x$ as
$x\mapsto u^2x+r$, giving
$U/Z\cong\operatorname{AGL}_1(\mathbb F_3)\cong S_3$.  Frobenius fixes $r$
and sends $u$ either to $u$ or to $u^{-1}$; these differ by an element of
$Z$, so the induced action on $U/Z$ is trivial.  Nonabelian $H^1$ for the
procyclic finite-field Galois group therefore reduces here to conjugacy
classes in $S_3$, giving the three map-level strata.

All four points of $E_0[2]$ are rational over $\mathbb F_3$, so base
Frobenius is the identity on $V$.  Twisting by $c$ changes this action to
$c$.  The displayed shift modules and orbit counts are the elementary actions
of $S_3\cong\operatorname{GL}_2(\mathbb F_2)$ on $V$.  Finally,
Proposition~2.1 of \cite{KST17} gives four elliptic twists when $a$ is odd,
projecting as one identity, two $3$-cycle, and one transposition class; when
$a$ is even it gives six twists, two in each projected class.
\end{proof}

For $c\in\{c_0,c_3,c_2\}$ and $e\in\{1,2\}$, let $K_{c,e}$ be the set of good
translation kernels on a representative of the corresponding map-level
stratum: for $e=1$ these are Frobenius-stable lines in $E_0[\ell]$ with
multiplier different from $\pm1$, while for $e=2$ the set is
$\{E_0[\ell]\}$ if Frobenius is irreducible and is empty otherwise.  This
set is well defined up to the residual centralizer action: conjugating the
descent datum transports good kernels, while multiplying Frobenius by the
central involution $[-1]$ preserves both invariant lines and the condition
that the multiplier avoid $\{\pm1\}$.  Put
\[
 \kappa_{c,e}=|K_{c,e}|.
\]
For the residual centralizers define
\[
 \kappa_{0,e}^{(2)},\ \kappa_{0,e}^{(3)},\
 \kappa_{3,e}^{(3)},\ \kappa_{2,e}^{(2)}
\]
to be the numbers of good kernels fixed by a representative transposition or
$3$-cycle, as indicated by the superscript.

\begin{theorem}[Global exact $C_2$ count in characteristic $3$]
\label{thm:char3-C2-count}
Let $\mathfrak R_e(q,\ell)$ be the literal raw mass defined by the formulas of
Section~\ref{sec:C2-count}, now using the full characteristic-$3$ trace set.
After the special $j=0$ contribution below is removed, the remaining terms
are evaluated by the same Schoof-corrected order-refined formulas as in that
section.  Define
\begin{equation}
\begin{split}
 J^{(3)}_e={}&
 \frac{4\kappa_{0,e}+6\kappa_{0,e}^{(2)}+2\kappa_{0,e}^{(3)}}6
 +\frac{\kappa_{3,e}+2\kappa_{3,e}^{(3)}}3\\
 &\qquad+\kappa_{2,e}+\kappa_{2,e}^{(2)},
\end{split}
\label{eq:char3-J}
\end{equation}
and
\begin{equation}
 B^{(3)}_e=
 \begin{cases}
 4\kappa_{0,e}+2\kappa_{3,e}+2\kappa_{2,e},&a\text{ odd},\\[1mm]
 8\kappa_{0,e}+2\kappa_{3,e}+4\kappa_{2,e},&a\text{ even}.
 \end{cases}
\label{eq:char3-B}
\end{equation}
Then
\begin{equation}
 \mathrm{Cl}^{(3)}_{2,e}(q,\ell)
 =\frac12\bigl(\mathfrak R_e(q,\ell)-B^{(3)}_e\bigr)+J^{(3)}_e.
\label{eq:char3-C2-count}
\end{equation}
Every term is an explicit finite Schoof-corrected class-number quantity or one
of the finite good-kernel fixed-set counts in the three supersingular strata.
\end{theorem}
\begin{proof}
For $j\ne0$, Kronberg--Soomro--Top show that
$\operatorname{Aut}(E_{\bar k},\mathrm O_E)=\{\pm1\}$ in characteristic $3$
\cite[\S2]{KST17}.  Hence the generic two quadratic elliptic twists
form one two-sided $C_2$ map form and carry equal raw weight, exactly as in the
one-half rule of Section~\ref{sec:C2-count}.  The term $B^{(3)}_e$ removes the
raw contribution of the supersingular $j=0$ elliptic twists, using the
multiplicities and shift-module sizes from
Lemma~\ref{lem:char3-C2-special}.
It remains to restore the three special map strata after quotienting by their
residual centralizers.  For $c_0$, Burnside's lemma on
$S_{c_0}\times K_{c_0,e}$ gives
\[
 \frac{4\kappa_{0,e}+3\cdot2\kappa_{0,e}^{(2)}
       +2\cdot1\kappa_{0,e}^{(3)}}6.
\]
For $c_3$, the shift module is trivial and its centralizer is $C_3$, giving
$(\kappa_{3,e}+2\kappa_{3,e}^{(3)})/3$.  For $c_2$, the residual $C_2$ acts
trivially on the two-element shift module, giving
$\kappa_{2,e}+\kappa_{2,e}^{(2)}$.  Their sum is $J^{(3)}_e$, and the stated
formula follows.
\end{proof}

\begin{corollary}[Rank-two $C_2$ simplification]
\label{cor:C2-r2-simplification}
In rank two no scalar-level correction at $\ell$ is needed.  The generic raw
mass is
\[
 \sum_{t\in\mathcal T(q)}
 \mathbf 1_{\left(\frac{D_t}{\ell}\right)=-1}
 \begin{cases}
 C_q(t),&t\text{ odd},\\
 2C_q(t)+2C_q(t;2,1),&t\text{ even},
 \end{cases}
\]
followed by the explicit special-$j$ correction appropriate to the
characteristic.  Thus the entire rank-two count depends only on Waterhouse trace strata,
order-refined rational $2$-torsion masses, and a finite marked-automorphism
correction.
\end{corollary}

\begin{proof}[Proof of Theorem~\ref{thm:master-enumeration}]
If the prescribed occurrence condition fails, there is nothing to count and,
by convention, $\mathrm{Cl}_{m,e}(q,\ell)=0$.  Assume occurrence.  For the CM
signatures,
Corollary~\ref{lem:form-quotient} reduces the count to the marked form, the
good translation kernel, and the shift orbit.  The resulting formulas are
Propositions~\ref{thm:C3-r1-count},~\ref{thm:C3-r2-count},~\ref{thm:C46-r1-count},
and~\ref{thm:C46-r2-count}; these also give the
rank-two rigidities and the nonsplit cubic count stated above.

For the involutive signature, Proposition~\ref{prop:scalar-level},
Lemma~\ref{lem:shift-mass}, and Proposition~\ref{prop:raw-r1} express the
generic contribution as exact finite Waterhouse trace/order data and
Schoof-corrected class-number sums.
The enlarged special-$j$ automorphism strata are then restored by the
finite Burnside corrections, giving Theorem~\ref{thm:global-C2-count} in
characteristic greater than $3$ and Theorem~\ref{thm:char3-C2-count} in
characteristic $3$.  Finally, in characteristic $2$, only the cubic signature survives, and
Corollary~\ref{thm:char2-occurrence-count} gives the explicit specialization
of the same marked $C_3$ counting formulas.  This exhausts the characteristic
range in Corollary~\ref{cor:intro-characteristic-range}.
\end{proof}

\section{Outlook}
\label{sec:outlook}

The prescribed-field form of a tame exceptional cover with genus-one
geometric Galois closure is controlled by a Frobenius-stable affine elliptic
quotient datum.  Its elliptic quotient, cyclic linear action, affine shift,
and Frobenius action on the translation kernel separate three layers of the
problem: fixed-field forms and branch arithmetic; monodromy, constant fields,
and exact finite-extension support; and occurrence and enumeration.  The affine sector
criterion strengthens this arithmetic readout: over every finite extension,
permutation is equivalent to exceptionality, so the two supports coincide
without accidental small-degree exceptions.  Thus the same datum that
resolves the descent ambiguity also organizes the full arithmetic closure of
the classification.

This description is complete for every tame genus-one signature.  In
characteristic greater than $3$, tameness is forced by the genus-one
hypothesis, while the surviving cases in characteristics $2$ and $3$ are
included through the marked-normalizer analysis.  The fixed-field
classification itself has no further obstruction.  The cubic signature does, however, exhibit a sharper arithmetic phenomenon:
in the nonsplit rank-two, nonzero-shift case, the existence of a same-$C_3$ self-endomorphism
realization is governed by an exact $3$-adic condition.  The degree-$49$
example over $\mathbb F_{17}$ shows that this stronger realization boundary is
genuine without creating an exceptional subcase in the classification.

Two natural extensions remain.  The first is the genuinely wild genus-one
theory in characteristics $2$ and $3$.  The second is the fixed-field
classification of genus-one monodromy configurations beyond the exceptional
affine setting.  In both directions, the central question is whether an
appropriate affine quotient datum continues to separate the geometric form of
the cover from the arithmetic carried by Frobenius relative to the marked
affine symmetry.

\newcommand{\etalchar}[1]{$^{#1}$}

\end{document}